%% file: Ellipt-arXiv.tex
\documentclass[12pt]{amsart}
\usepackage{amsmath}
\usepackage{color}
\usepackage{enumerate}	
\input xy
\xyoption{all}
\theoremstyle{plain}
  \newtheorem{thm}{Theorem}[section]
\newtheorem*{thm*}{Theorem}
  \newtheorem{lem}[thm]{Lemma}
  \newtheorem{cor}[thm]{Corollary}
  \newtheorem{obsvtn}[thm]{Observation}
  \newtheorem{prop}[thm]{Proposition}
  \newtheorem{fact}[thm]{Fact}
\newtheorem*{fact*}{Fact}
\theoremstyle{definition}
  \newtheorem{defn}[thm]{Definition}
 \newtheorem{assumption}[thm]{Assumption}
  \newtheorem{clm}[thm]{Claim}
  
  \newtheorem{qstn}[thm]{Question}

  \newtheorem{notation}{Notation\!\!}

\theoremstyle{remark}
  \newtheorem{rem}[thm]{Remark}
  \newtheorem{case}{Case}
  
  \newtheorem{stng}[thm]{Setting}

\newtheorem{acknowledgment}{Acknowledgment}

\numberwithin{equation}{section}
\newcommand{\ad}{\operatorname{ad}}
\newcommand{\bfi}{\bar{f_i}}
\newcommand{\breta}{\bar{\eta}}
\newcommand{\calCN}{{\mathcal C}{\mathcal N}}

\newcommand{\cf}{{\mathfrak f}}

\newcommand{\dotimes}{\underline{\otimes}}
\newcommand{\FF}{{\mathbb{F}}}
\newcommand{\gr}{\operatorname{gr}}
\newcommand{\grCk}{\operatorname{gr}^{JH}_0(Cok)}
\newcommand{\grlz}{\operatorname{gr}^{JH}_{l_0}}
\newcommand{\HH}{\operatorname{H}}
\newcommand{\HN}{\operatorname{HN}}
\newcommand{\Id}{\operatorname{Id}}
\newcommand{\JH}{\operatorname{JH}}
\newcommand{\JHCk}{\operatorname{JH}_0(Cok)}
\newcommand{\JHlz}{\operatorname{JH}_{l_0}}
\newcommand{\niB}{\nu_0^{-1}B}
\newcommand{\pur}{\operatorname{pur}}
\newcommand{\Rc}{R^{\wedge}}
\newcommand{\sF}{{\mathcal F}}
\newcommand{\sG}{{\mathcal G}}
\newcommand{\sH}{{\mathcal H}}
\newcommand{\sL}{{\mathcal L}}
\newcommand{\sLi}{{\mathcal L}^{\vee}}
\newcommand{\sO}{{\mathcal O}}
\newcommand{\sQ}{{\mathcal Q}}
\newcommand{\sR}{{\mathcal R}}
\newcommand{\sRKoG}{{\mathcal R}(K_X)/{\mathcal G}}
\newcommand{\tor}{\operatorname{tor}}

\input{newcmds.tex}

\begin{document}
\bibliographystyle{amsplain}

\setcounter{thm}{0}
\title[Kodaira dimension and singularities of moduli]
{The Kodaira dimension and singularities 
of moduli \\
of stable sheaves on some elliptic surfaces}


%
\author{Kimiko Yamada}
\date{\today}
\email{yamada@xmath.ous.ac.jp}
\address{Department of applied mathematics,
Faculty of Science, 
Okayama University of Science, Japan}
\thanks{This work was supported by the Grants-in-Aid for 
Young Scientists (B), JSPS, No. 23740037. }
\begin{abstract}
Let $X$ be an elliptic surface over ${\bf P}^1$ with $\kappa(X)=1$, and
$M=M(c_2)$ be the moduli scheme of rank-two stable sheaves $E$ on $X$ with 
$(c_1(E),c_2(E))=(0,c_2)$ in $\operatorname{Pic}(X)\times\mathbb{Z}$.
We look into defining equations of $M$ at its singularity $E$,
partly because if $M$ admits only canonical singularities, then
the Kodaira dimension $\kappa(M)$ can be calculated. We show the following. \par
(A) $E$ is at worst canonical singularity of $M$
if the restriction of $E_{\eta}$ 
to the generic fiber of $X$ has no rank-one subsheaf, and 
if the number of multiple fibers of $X$ is a few.\par
(B) We obtain that $\kappa(M)=\{1+\dim(M)\}/2$ and the Iitaka program of $M$ can be
described in purely moduli-theoretic way for $c_2\gg 0$,
when $\chi({\mathcal O}_X)=1$, $X$ has just two multiple fibers, and
one of its multiplicities equals $2$. \par
(C) On the other hand, 
when $E_{\eta}$ has a rank-one subsheaf, 
it may be insufficient to look at
only the degree-two part of defining equations to judge whether $E$ is at worst canonical singularity or not.
%
%
%
%
%
\end{abstract}
\maketitle
%
%
\setcounter{section}{0}
\section{Introduction}\label{sctn:update-intro}
%
For an ample line bundle $H$ on projective smooth surface $X$ over $\CC$,
there is the coarse moduli scheme $M(c_2)$ (or $M(c_2,H)$)
of rank-two $H$-stable sheaves $E$ with
Chern classes $(c_1,c_2)=(0, c_2)$ in $\Pic(X)\times \ZZ$ (\cite{Gi:moduli}). 
For every point $E\in M(c_2)$,
the completion ring $\sO_{M,E}^{\wedge}$ of $M(c_2)$ at $E$ gives
the formal universal moduli of the functor assigning deformation of $E$ over 
local Artinian $\CC$-algebra (e.g. \cite[Thm. 19.3]{Ha10:deform}).
\begin{qstn}
(1) Is $M(c_2)$ nonsingular or not? \\
(2) Suppose that $M(c_2)$ has a singular point $E$, that is, an obstructed sheaf.
How is the analytic structure of $M(c_2)$ at $E$, in other words,
the ring structure of $\sO_{M,E}^{\wedge}$? \\
(3) How is the birational structure of $M(c_2)$; for example, its Kodaira dimension $\kappa(M)$?
\end{qstn}
For example, it is known that $M(c_2)$ is nonsingular when $X$ is $K3$ surface or
$-K_X$ is ample.
In this paper, we shall work in the setting below.
\begin{stng}\label{stng:XandH}
$X$ is a minimal surface over $\CC$ whose
Kodaira dimension $\kappa(X)$ equals $1$ and
irregularity $q(X)$ equals $0$,
so there is an elliptic fibration $\pi:X \rightarrow \PP^1$.
Every singular fiber of $X$ is
either rational integral curve with one node $(I_1)$ or
multiple fiber with smooth reduction $(mI_0)$.
We denote the number of multiple fibers of $\pi$ by $\Lambda(X)$ and $d=\chi(\sO_X)$.
Fix an integer $c_2> 0$, and let $H$
be an ample line bundle which is $c_2$-suitable (Definition \ref{defn:suitable}).
\end{stng}

%
%
%
Any sheaf $E$ in $M(c_2)$ induces a rank-two vector bundle $E_{\eta}$ with degree $0$
on the generic fiber $X_{\eta}$, where $\eta=\Spec(k(\PP^1))$.
This can be classified into three cases (Fact \ref{fact:RestrGenericFib}): \par
\begin{enumerate}[{Case} I]
\item : $E_{\eta}$ has no sub line bundle with fiber degree $0$. 
\item :
$E_{\eta}$ has a sub line bundle with fiber degree $0$, but $E_{\eta}$ is not decomposable. 
\item :
$E_{\eta}$ is decomposable into line bundles with fiber degree $0$. 
\end{enumerate}
%
%
In this classification, sheaves of Case I appear most frequently in $M(c_2)$, because
Case I is an open condition by Lemma \ref{lem:I-open-property}.
Recall that, by the deformation theory of sheaves (Fact \ref{fact:moduli-lci}),
if $E$ is a singular point of $M(c_2)$ then for
$b=\dim\Hom(E,E(K_X))^{\circ}\neq 0$ and $D=\dim\Ext^1(E,E)-\dim\Hom(E,E(K_X))^{\circ}$, 
that is the expected dimension of $M(c_2)$, we have
\begin{equation}\label{eq:formal-moduli-intro}
 \sO_{M,E}^{\wedge}\simeq \CC[[t_1,\cdots, t_{D+b}]]/\langle F_1,\dots, F_b \rangle,
\end{equation}
where $F_i$ is a power series starting from degree-two terms.
\begin{thm}[Theorem \ref{thm:rk-geq-3ext-ver1},
Proposition \ref{prop:Ex-Sing}] \label{thm:rk-geq-2ext-1-Intro}
Suppose that
$E$ is a singular point of $M(c_2)$ and applies to Case I.
If $7(d+2)/4\geq \Lambda(X)$ or $2\geq \Lambda(X)$, then the following holds: \\
(1) Let $G$ be any nonzero $\CC$-linear combination of $F_1,\cdots F_b$ at \eqref{eq:formal-moduli-intro} and
we indicate $G$ as
\begin{equation}\label{eq:defn-h-R}
 G=t_1^2+\cdots+ t_R^2+O(3),
\end{equation}
where $O(3)$ are terms whose degrees are more than $2$, 
and $R$ is an integer depending on $G$. Then $R\geq 2b+1$. \\
(2) $E$ is a canonical singularity of $M(c_2)$. \\
Moreover, there actually exist such obstructed sheaves on $M(c_2)$ if $c_2\gg 0$.
\end{thm}
Theorem \ref{thm:rk-geq-2ext-1-Intro}
has an important application Theorem \ref{thm:(2,m)}, as explained in Section \ref{sctn:update-backgrnd} below.\par
%
%
%
%
Next let us consider obstructed sheaves applying to Case II.
Sheaves of Case II appear second frequently in the above classification from Lemma \ref{lem:I-open-property}.
Here we consider the case where $d=1$ and $\Lambda(X)=2$;
one might say that such surfaces have the smallest number of singular fibers and multiple fibers
among all the elliptic surfaces of Kodaira dimension one,
since the number of singular fibers of type $(I_1)$ is known to be $12d$.
In this case, every obstructed sheaf $E$ in $M(c_2)$ satisfies $b=1$ at \eqref{eq:formal-moduli-intro},
that is, $(M(c_2),E)$ is always a hypersurface singularity.
\begin{equation}\label{eq:formal-moduli-hypersurf-intro}
  \sO_{M,E}^{\wedge}\simeq \CC[[t_1,\cdots, t_{D+1}]]/\langle F \rangle \qquad (F=t_1^2+\cdots+ t_R^2+O(3)).
\end{equation}
%
\begin{thm}[Proposition \ref{f_*-tor-nonzero}, \ref{prop:locfree-l(Z)-l(ZB)=1}, \ref{prop:example-Case2-R1}]\label{thm:R-geq-1}
Assume that $d=1$ and $\Lambda(X)=2$.\\
(1) For any locally-free obstructed sheaf $E$ applying to Case II,
it holds that $R\geq 1$ at \eqref{eq:formal-moduli-hypersurf-intro}.\\
(2) There actually exist locally-free obstructed stable sheaves of Case II satisfying $R=1$
for every $c_2\gg 0$, when $X$ fit one of conditions (1)--(4) in Proposition \ref{prop:example-Case2-R1}.
\end{thm}
When $c_2$ is sufficiently large, 
$M(c_2)$ is normal and its dimension equals $4c_2-3d$ by Proposition \ref{prop:irred}.
When an obstructed sheaf $E$ of $M(c_2)$ satisfies $R=1$,
the defining equation $F$ of $M(c_2)$ at $E$ \eqref{eq:formal-moduli-hypersurf-intro}
never equals its degree-two part. 
As a result, we can summarize as follows:
%
%
\begin{obsvtn}\label{obsvtn:high-deg}
Assume that $\kappa(X)=1$.
When we want to grasp $\sO_{M,E}^{\wedge}$ at \eqref{eq:formal-moduli-intro},
it may be insufficient to look only at the degree-two part of defining equations $F_i$.
It is possible that we have to look also into its higher-degree part.
\end{obsvtn}
%
%
%
\setcounter{subsection}{1}
\subsection{Birational geometric property of moduli scheme}\label{subsctn:birat}
Here we consider the birational structure of $M(c_2)$.
When $c_2$ is sufficiently large, $M(c_2)$ is normal
and then its $K$-dimension $\kappa(K,M)$ and its Kodaira dimension $\kappa(M)$ are defined
(Definition \ref{def:D-dim}(1)).
We want to calculate $\kappa(M)$ because it is one of the most important birational invariants.
We can calculate $\kappa(K,M)$ at Corollary \ref{cor:KodDim-Moduli} (ii)
using Friedman's work \cite{Fri:vb-reg-ellipt} as the key.
It would be very favorable if all singularities of $M(c_2)$ are {\it canonical singularities}
(Definition \ref{def:D-dim}),
because $\kappa(K,M)$ equals the Kodaira dimension $\kappa(M)$ in such a case.
This new course also has the advantage that
we can interpret the Iitaka program of $M$ (Fact \ref{fact:IitakaPgrm}(4)), that is a fundamental part
in birational geometry,
using moduli theory as pointed out
at Corollary \ref{cor:KodDim-Moduli} (ii).\par
Assume that $E$ is a hypersurface singularity of $M=M(c_2)$ as at \eqref{eq:formal-moduli-hypersurf-intro}.
If $R\geq 3$, then $(M,E)$ is a canonical singularity.
If $R=2$ and $c_2\gg 0$, then $(M,E)$ is a canonical singularity since $M(c_2)$ is normal.
If $R\leq 1$, then one can not judge whether $(M,E)$ is a canonical singularity from degree-two part of $F$.
%
%
%
Roughly speaking, 
Theorem \ref{thm:rk-geq-2ext-1-Intro} (2)
states that an obstructed sheaf $E$ of $M(c_2)$ of Case I
is always a canonical singularity
when $\Lambda(X)$ is small relatively to the number of singular fibers of type $(I_1)$.
As an application, we obtain the following theorem.
%
%
\begin{thm}[Theorem \ref{thm:(2,m)}]\label{thm:(2,m)-Intro}
 In Setting \ref{stng:XandH}, we suppose that $X$ has
 just two multiple fibers
 with multiplicities
$m_1=2$ and $m_2=m\geq 3$, and $d=\chi(\sO_X)=1$.\\
(i) If $M(c_2)$ is singular at a stable sheaf $E$, then 
$E$ always comes under Case I in Fact \ref{fact:RestrGenericFib},
 and consequently $E$ is always a canonical
singularity of $M(c_2)$.\\
(ii) We consider in Setting \ref{stng:XHS}.
If $c_2\geq 3$ and if $M(c_2)$ is compact (for example, $c_2$ is odd), then
the Kodaira dimension $\kappa(M(c_2))$ of $M(c_2)$ equals
$(\dim M(c_2)+1)/2$.
\end{thm}
\subsection{Contents of this paper}
In Section \ref{section:prelim}, we recall background materials, including some facts in birational geometry.
In Section \ref{sctn:K-dim}, we compare pluricanonical map $\Phi:M \dashrightarrow |mK|$ with
Friedman's map $\psi:M \dashrightarrow \PP^N$ constructed using moduli theory (Sect. \ref{sbsctn:EvenFibDeg}).
Consequently we calculate $\kappa(K,M)$ and
use $\psi$ to understand Iitaka program of $M$ at Corollary \ref{cor:KodDim-Moduli}.
In Section \ref{sctn:canonical}, 
we obtain a sufficient condition for singularities to be canonical at Theorem \ref{thm:2k+1-CanSing}.
This theorem itself is purely ring-theoretic. Applying it, we use
$H^1(\ad(f))$ defined at \eqref{eq:def-H1f} 
to know an obstructed sheaf $E$ is canonical singularity or not
at Theorem \ref{thm:suffcdtn-moduli-can}.
In Section \ref{sctn:rkH1-caseA1}, we try to estimate $H^1(\ad(f))$  for sheaves of Type I. 
In Section \ref{sctn:SigM-CaseA1}, we prove Theorem \ref{thm:rk-geq-2ext-1-Intro} (1)(2).
In Section \ref{sctn:ElptSurf-FewSingFib}, we prove Theorem \ref{thm:(2,m)-Intro}.
In Section \ref{sctn:eg-sing}, we show 
"actually exists" part of Theorem \ref{thm:rk-geq-2ext-1-Intro} 
at Proposition \ref{prop:Ex-Sing}.
In Section \ref{sctn:CaseII}, we consider sheaves of type II and show Theorem \ref{thm:R-geq-1}.
\subsection{Background and previous researches}\label{sctn:update-backgrnd}
First we mention obstructed {\it stable} sheaves.
Every torsion-free $H$-stable sheaf $E$ is unobstracted 
if  $K_X=0$ ($K3$ surfaces, Abelian surfaces) or $K_X\cdot H<0$ (\cite[Prop. 6.17.]{Frd:holvb}).
When $X$ are Enriques surfaces, the author \cite{Yamada:Enriques} showed that
there actually exist obstructed stable sheaves, and that every such sheaf gives hypersurface singularity
\eqref{eq:formal-moduli-hypersurf-intro} satisfying that $R\geq 2c_2-2\chi(\sO_X)$.
On moduli of sheaves on $K3$ surfaces,
singularities coming from {\it strictly semi-stable} sheaves on $K3$ surfaces are actively researched.
Here we just cite some references; for example, \cite{OG03:6dim}, \cite{KLS06}, \cite{Yos17:FM-K3},
\cite{Arb-Sac18}, \cite{BZ18}.\par
Now let $X$ be an elliptic surface, and consider the moduli scheme $M$ of rank-two stable sheaves with Chern classes
$(c_1,c_2)$. If $c_1\cdot\cf=2e+1$ is odd, then $M$ is non-singular by
\cite[Lem. 8.8.]{Frd:holvb} and birationally equivalent to $\Hilb^{\dim(M)/2}(J^{e+1}(X))$
and $\kappa(M)=\dim(M)/2$, for $\kappa(J^{e+1}(X))=1$ from \cite[p.80, l.33]{Fr:so3ellipt}.
Here, $J^{e+1}(X)$ indicates the relative Picard scheme of line bundles on the fibers of degree $e+1$.
%
%
About this case, we cite \cite{Fr:so3ellipt}, \cite{Br:FMtrans}, \cite{Yoshioka99-stableVB-ElliptSurf}.
On the other hand, in case where $c_1(E)=0$, 
$\kappa(M)$ is not known although its upper bound
was given at \cite[p.328]{Fri:vb-reg-ellipt}. In our case,
it is more difficult to grasp the geometry of $M$ 
rather than in case where $(r,c_1\cdot\cf)=1$.
This is because the restriction of any stable sheaf to
the generic fiber is not stable but strictly semistable in the former case,
though it is stable in the latter case.
When $p_g(X)\neq 0$, one can use induced two forms and Poisson structure on $M$ 
(\cite{Li:kodaira}, \cite{Mk:symplectic}, \cite{Ty:symplectic}).
However we can not adopt them in our case since $p_g(X)=0$.
Therefore we take another course as explained at Section \ref{subsctn:birat}.\par
\begin{acknowledgment}
The author would like to sincerely thank Professor Shihoko Ishii
for information about Elkik's work \cite{Elkik:DefRatSing},
especially for Remark \ref{rem:RatAnalytic}.
The author would like to sincerely thank Professor Masataka Tomari for
giving valuable advices about Theorem \ref{thm:2k+1-CanSing} and Remark \ref{rem:a-inv}.
\end{acknowledgment}
\begin{notation}
For a real number $r$, the symbol $\lceil r \rceil$ 
means the smallest integer that is not less than $r$,
and $\lfloor r \rfloor$ the largest integer that is not greater than $r$.
A bilinear form $B$ on a vector space $W$ is
a symmetric linear function from $W\otimes_k W$ to $k$.
All schemes are of locally finite type over $\CC$.
For Weil divisors, $\sim$ stands for the linear equivalence, and $\equiv$ the
numerical equivalence.
Let $X$ be an integral projective surface over $\CC$.
For coherent sheaves $F$ and $E$ on $X$, $h^i(E)$ means $\dim H^i(X,E)$ and
$\ext^i(E,F)$ means $\dim\Ext^i(E,F)$.
For a line bundle $L$ on $X$,
we denote the kernel of trace map ${\rm tr}: \Ext^i(E,E\otimes L) \rightarrow H^i(L)$
by $\Ext^i(E,E\otimes L)^{\circ}$, and its dimension by
$\ext^i(E,E\otimes L)^{\circ}$.
%
\end{notation}

%
%
%
%
%
%
\section{Background materials}\label{section:prelim}
%
\begin{defn}
Let $\bar{M}(c_2,H)$ be the coarse moduli scheme of S-equivalence classes of
$H$-semistable sheaves with
Chern classes $(r,c_1,c_2)=(2,0, c_2)\in \Pic(X)\times \ZZ$
(\cite{Gi:moduli}). It is projective over $\CC$, 
and contains $M(c_2, H)$ as an open subscheme.
\end{defn}

\subsection{Deformation theory of stable sheaves}
\begin{fact}\cite{Lau:Massey}\label{fact:moduli-lci}
%
Let $E$ be a stable sheaf on a non-singular projective surface.
Put $\dim\Ext^1(E,E)=D+b$ and $\dim\Ext^2(E,E)^{\circ}=b$, and
let $f_1, \dots, f_b$ be a basis of $\Hom(E,E(K_X))^{\circ}$.
Then the completion ring of moduli of sheaves at $E$ is isomorphic to
$\CC[[t_1,\dots, t_{D+b}]]/(F_1,\dots, F_b)$. Here
$F_i$ is a power series starting with degree-two term, which comes from
\begin{equation}\label{eq:Ff}
 F_{f_i}: \Ext^1(E,E) \otimes \Ext^1(E,E) \longrightarrow \Ext^2(E,E) 
\longrightarrow \CC  
\end{equation}
defined by $F_{f_i}(\alpha\otimes \beta)
=\operatorname{tr}(f_i\circ\alpha\circ\beta + f_i\circ\beta\circ\alpha)
=\operatorname{tr}(H^1(\ad(f_i))(\alpha)\circ\beta)$, 
where $H^1(\ad(f_i))$ is defined at \eqref{eq:def-H1f},
and its dual map
\begin{equation*}
 F^{\vee}_{f_i}: \CC \longrightarrow \Ext^1(E,E)^{\vee}\otimes 
\Ext^1(E,E)^{\vee} \longrightarrow \operatorname{Sym}^2(\Ext^1(E,E)^{\vee}).
\end{equation*}
\end{fact}

\begin{defn}\label{defn:potentially-obstructed}
A stable sheaf $E$ on $X$ is {\it obstructed} if
$\ext_X^2(E,E)^0 = \hom_X(E,E(K_X))^0 \neq 0$.
\end{defn}

\subsection{Canonical singularities and birational classification}\label{sbsctn:birat}
\begin{defn}\label{def:D-dim}
(1)\cite[Sect.10.5]{Iitaka:AG}
Let $D$ be a $\QQ$-Cartier divisor on a complete normal variety $V$.
The {\it $D$-dimension} $\kappa(D,V)$ of $V$ is defined to be
\[ \kappa(D,V)=\max\{\dim\Phi_{|mD|}:V \dashrightarrow |mD| \bigm| 
 m\in \NN, mD \text{ is Cartier and } h^0(mD)\neq 0\}.\]
The {\it Kodaira dimension} $\kappa(V)$ of $V$ is
$\kappa(K_{\tilde{V}},\tilde{V})$, where $\tilde{V}$ is a non-singular complete variety
birationally equivalent to $V$. Kodaira dimension is birationally invariant.
Remark $\kappa(K,V)$ does not equal $\kappa(V)$ in general.\\
%
(2)(\cite[Def. 6.2.4.]{Ishii:text}) A normal singularity $(V,p)$ is a
{\it canonical singularity} if 
(a) the Weil divisor $rK_V$ is Cartier for some $r\in\NN$ and 
(b) if $f: W\rightarrow V$ is a resolution of singularities,
$E_1,\dots, E_r$ are its prime exceptional divisors and one denotes
$K_W= f^*(K_V) +\sum a_i E_i$, then $a_i\geq 0$. \par
When $V$ is complete and has only canonical singularities,
its $K$-dimension and its Kodaira dimension are equal, so
we need not consider desingularization $\tilde{V}$ of $V$
in calculating $\kappa(V)$. \\
%
(3)(\cite[Def. 6.2.10.]{Ishii:text})
Let $p$ be a singular point of a normal variety $V$.
$(V, p)$ is said to be {\it rational singularity} when the following holds:
Suppose $\tilde{V}$ is non-singular and
a proper morphism $f:\tilde{V}\rightarrow V$ 
is isomorphic on some open subsets in $V$ and $\tilde{V}$.
Then $R^i f_*{\mathcal O}_{\tilde{V}}=0$ for all $i>0$.
\end{defn}
\begin{fact}\label{fact:Can=Rat}
(1)(\cite[Cor. 6.2.15]{Ishii:text})
If the normal singularity $(V,p)$ satisfies that $K_V$ is Cartier,
then $(V,p)$ is a canonical singularity if and only if it is a rational singularity.\\
(2)(\cite{Elkik:DefRatSing})
Let $({\mathcal S},p)\rightarrow (C,0)$ be a flat deformation of
a rational singularity $({\mathcal S}_{k(0)},p)$, where ${\mathcal S}$ 
is of finite type over $\CC$.
By replacing ${\mathcal S}$ with an open neighborhood of $p$,
one can assume that ${\mathcal S}_{k(t)}$ has only rational singularities
when $t\in C$ is sufficiently close to $0$.
\end{fact}
\begin{rem}\label{rem:RatAnalytic}
Fact \ref{fact:Can=Rat} (2) holds
also when ${\mathcal S}$ and $C$ are analytic varieties.
The proof of Fact \ref{fact:Can=Rat} (2) proceeds
similarly when ${\mathcal S}$ is algebraic, since
Hironaka's theorem on resolution of singularities and
Grauert-Riemenschneider's vanishing theorem (\cite[Thm. 6.1.12]{Ishii:text})
hold in analytic category.
\end{rem}
%
%
%
%
%
Let us recall some methods and facts in birational geometry.
\begin{fact}\label{fact:IitakaPgrm}
(1) Let $V$ be a normal and proper variety such that
$K_V$ is $\QQ$-Cartier and nef, and that $V$ has only canonical singularities.
Then $V$ is a {\rm minimal model}.\\
%
(2) Let $D$ be a Cartier divisor on a projective scheme $V$ such
that $n_0 D$ is base-point free for some $n_0\in \NN$.
Then the ring $\oplus_{n\geq 0} H^0(V, \sO(n D))$ is finitely generated
over $\CC$, and 
the natural map $\Phi: V \rightarrow  \operatorname{Proj}\left(\oplus_{n\geq 0} H^0(V, \sO(n D))\right)$
is a dominant morphism.\\
%
(3)(Abundance Conjecture)
Let $V$ be a variety as in (1). Then it is conjectured that $|n_0K_V|$ will be 
base-point free for some $n_0\in\NN$.\\
%
(4)(Iitaka Fibration)
Let $V$ be a variety as in (1) such that 
$|n_0K_V|$ is base-point free for some $n_0\in\NN$.
Then 
\[\Phi:V \rightarrow V_{can}:=
 \operatorname{Proj}\left(\oplus_{n\geq 0} H^0(V, \sO(n K_V))\right) \]
satisfies the following:
$\Phi$ is a surjective morphism with connected fibers;
$mK_V=\Phi^*({\mathcal K})$
for some ample divisor ${\mathcal K}$ on $V_{can}$;
$\kappa(V)=\dim V_{can}$;
general fibers of $\Phi$ are normal varieties with $lK = 0$ for some $l\in\NN$.
We call $\Phi$ the {\rm Iitaka fibration} of $V$ onto its {\rm canonical model} $V_{can}$.
\end{fact}
Here we just cite some references;
\cite{Matsuki:UTX}(e.g. p.1, pp.4-5, Thm. 3.3.6.),
\cite{KM:birat}(e.g. Def. 3.50, Thm. 3.11, Conj. 3.12),
\cite[Thm. 10.7]{Iitaka:AG}.
%
%
\subsection{Line bundles on moduli of sheaves}\label{sbsctn:LBmoduli}
(\cite[Sect. 8]{HL:text})
By Luna's \'{e}tale slice theorem,
there is an \'{e}tale covering $S\rightarrow M(c_2)$ and
a universal family ${\mathcal E}\in \Coh(X\times S)$ of $M(c_2)$.
For $u\in K(X)$,
denote $\chi(E\dotimes u)$, where $E$ is a member of $M(c_2)$,
by $\chi(u\cdot c_2)$.
If $\chi(u\cdot c_2)=0$,
then $\det R \pi_{1*}({\mathcal E}\dotimes \pi_2^*(u)) \in \Pic(S)$ 
descends to $\lambda(u)\in \Pic(M(c_2))$.
So we have 
$\lambda(u_1(c_2))\in\Pic(M(c_2))$,
where $k_X=[\Ox]-[K_X^{-1}] \in K(X)$, $x$ is a closed point on $X$, and
$u_1(c_2)=-rk_X+\chi(c_2 \cdot k_X)\cdot[{\mathcal O}_x]$.
Let $M_{gd}\subset M(c_2)$ be the good locus of $M(c_2)$ defined at
\eqref{eq:GoodLcs} below.
Then 
$K_{M_{gd}}=\lambda(u_1(c_2))|_{M_{gd}}$.
\subsection{Basics of elliptic surfaces}\label{subsct:basic}
Let $X$ be an elliptic surface over $\PP^1$.
By \cite[III.11.2 and V.12.2]{BPV:text}, 
$\chi(\sO_X)=d\geq 0$ and $R^1\pi_*(\sO_X)=\sO(-d)$.
When every singular fiber of $X$ is
either rational integral curve with one node $(I_1)$ or
multiple fiber with smooth reduction $(mI_0)$,
the number of singular fibers of type $(I_1)$ is $12d$
(\cite[p.177-178]{Frd:holvb}).
By Kodaira's canonical bundle formula \cite[Thm. V.12.1]{BPV:text},
if $\pi:X \rightarrow \PP^1$ is a relatively minimal elliptic surface such that its
multiple fibers are $m_iF_i$ with multiplicity $m_i$, then
\begin{equation}\label{eq:CanBdleFormula}
 K_X= \pi^*\sO_{\PP^1}(-2+d) + \sum_i (m_i-1) F_i.  
\end{equation}
Next let us recall the {\it Jacobian surface} $J(X)$ of $X$, which is an elliptic
surface over $\PP^1$ with a section.
From $X$ we get an analytic elliptic surface $X'\rightarrow \PP^1$
without multiple fibers,
by reversing the logarithmic transformation (\cite[Sect. V.13.]{BPV:text}).
$J(X)$ is the Jacobian fibration of $X'$ described
at \cite[Sect. V.9.]{BPV:text}. We can refer also to 
\cite[Def. I.3.15]{FM:4mfds}.
Remark that $\chi(\sO_{J(X)})$ agrees with $d=\chi(\sO_X)$
by \cite[Lem. I.3.17]{FM:4mfds}. \par
For fixed $d\geq 0$, there exists an elliptic surface $B$ over $\PP^1$
with a section such that $d=\chi(\sO_B)$ by \cite[Prop. I.4.3.]{FM:4mfds}.
Fix an elliptic surface $B$ over $\PP^1$ with a section,
$t_1,\dots,t_k \in\PP^1$ 
and line bundles $\xi_i$ on $\pi^{-1}(t_i)$ of order $m_i$.
Let $T$ denote the set of (analytic) elliptic surfaces $X$ such that 
$J(X)\simeq B$, 
$X$ has multiple fibers just over $t_i$ with multiplicities $m_i$, and
for some disk $\Delta_i \subset \PP^1$ centered at $t_i$,
$X|_{\Delta_i}$ equals to the logarithmic transformation 
of $B|_{\Delta_i}$ corresponding to $t_i$ and $\xi_i$
(\cite[Sect. V.13.]{BPV:text}, \cite[Sect. I.1.6.]{FM:4mfds}).
Then $T$ is non-empty and
its subset consisting of {\it algebraic} elliptic surfaces is
dense unless $B$ is a product elliptic surface
(\cite[Thm. I.6.12.]{FM:4mfds}).
By \cite[Thm. B.3.4.]{Ha:text},
compact complex algebraic surfaces are projective algebraic.
\begin{prop}\cite[p.169, Lem.5]{Frd:holvb}\label{prop:order-O(F)F}
Let $\cf=m F$ be a multiple fiber of $X$.
The the order of the line bundle $\sO(F)|_F$ is $m$.
\end{prop}

%
%
\subsection{Stable sheaves on elliptic surfaces}\label{sbsctn:EvenFibDeg}
%
Let us recall some results in \cite{Fri:vb-reg-ellipt}.
\begin{defn}\cite[Def. 2.1]{Fri:vb-reg-ellipt}\label{defn:suitable}
On any elliptic surface $X$, a polarization $H$ is {\it $c_2$-suitable} if
$\operatorname{sign}(L\cdot f)=\operatorname{sign}(L\cdot H)$ 
for all $L\in\Pic(X)$ such that
$L^2\geq -c_2$ and $L\cdot f\neq 0$.
This implies that $H$ is not separated from the fiber class $f$
by any wall of type $(2,0,c_2)$ in the nef cone of $X$. 
Such a polarization exists by \cite[Lem. 2.3.]{Fri:vb-reg-ellipt}.
\end{defn}
In Setting \ref{stng:XandH},
we set $\eta=\Spec(k(\PP^1)),\ \bar{\eta}=\Spec((\overline{k(\PP^1)}))$,
and define $X_{\eta}=X \times_{\PP^1} \eta$, $X_{\bar{\eta}}=X\times_{\PP^1} \bar{\eta}$.
$X_{\bar{\eta}}$ is a smooth elliptic curve over $\bar{\eta}$.
Any sheaf $F$ on $X$ induces $F_{\eta}$ on $X_{\eta}$, 
and $F_{\bar{\eta}}$ on $X_{\bar{\eta}}$.
\begin{fact}\cite[Thm. 3.3]{Fr:so3ellipt}\label{fact:XeStable}
Let $E$ be a rank-two torsion-free sheaf with $c_1=0$ on $X$.
If $E_{\eta}$ is stable, then 
$E$ is stable with respect to any $c_2(E)$-suitable ample line bundle.
\end{fact}
\begin{fact}\cite[Lem. 2.5]{Fri:vb-reg-ellipt}\label{fact:RestrGenericFib}
Let $E$ be a rank-two sheaf on $X$ with Chern classes $(c_1,c_2)=(0,c_2)$, which is stable
with respect to a $c_2$-suitable ample line bundle $H$.
Then one of the following holds:\\
{\rm (Case I)}:
$E_{\eta}$ has no sub line bundle of degree zero.
In this case $E_{\eta}$ is stable, and $E_{\bar{\eta}}$ is decomposable as
\begin{equation*}
E_{\bar{\eta}} \simeq \sO_{X_{\bar{\eta}}}(F) \oplus \sO_{X_{\bar{\eta}}}(-F)\quad
\text{on}\quad X_{\bar{\eta}}\quad \text{with}\quad \deg(\sO_{X_{\bar{\eta}}}(F))=0,
\end{equation*}
and $\sO_{X_{\bar{\eta}}}(F)$ is not rational over $k(\PP^1)$. 
Let $C\rightarrow \PP^1$ be the double cover
corresponding to the stabilizer subgroup of 
$\sO_{X_{\bar{\eta}}}(F)$ in $\operatorname{Gal}(\overline{k}(\PP^1)/k(\PP^1))$.
Then $\sO_{X_{\bar{\eta}}}(F)$ is rational over $\eta'=\Spec(k(C))$ and
$E_{\eta'}$ is $\eta'$-isomorphic to 
$\sO_{X_{\eta'}}(F) \oplus \sO_{X_{\eta'}}(-F)$.\\
{\rm (Case II)}:
$E_{\eta}$ has a sub line bundle with fiber degree $0$, 
but $E_{\eta}$ is not decomposable. In this case, 
also $E_{\bar{\eta}}$ is not decomposable, and there is an exact sequence 
\begin{equation}\label{eq:rel-JH}
0 \longrightarrow \sO_{X_{\eta}}(F) \longrightarrow E_{\eta} \longrightarrow
\sO_{X_{\eta}}(F) \longrightarrow 0
\end{equation}
on $X_{\eta}$, where $\sO_{X_{\eta}}(F)$ is a line bundle of order $2$ on $X_{\eta}$.\\
{\rm (Case III)}:
$E_{\eta}$ is decomposable into line bundles with fiber degree $0$ on $X_{\eta}$. 
\end{fact}
\begin{lem}\label{lem:I-open-property}
Let $U_1$ (resp. $U_{12}$) be the subset of $M(c_2)$ consisting of sheaves $E$
which apply to Case I (resp. Case I or Case II) in Fact \ref{fact:RestrGenericFib}.
Then both $U_1$ and $U_{12}$ are open in $M(c_2)$.
\end{lem}
\begin{proof}
$E\in M(c_2)$ applies to Case II or III if and only if it holds that $\Hom(F,E)\neq 0$ for some sheaf
$F=\sO(D)\otimes I_W$, where
$D$ is a divisor on $X$ such that $D\cdot \cf=0$ and $0\leq -D^2 \leq c_2(E)$ and $W$ is a closed subscheme of $X$
such that $0\leq l(W) \leq c_2(E)$.
By \cite[Lem. 2.1]{Qi:birational}, $|\sO(D)\cdot \sO(1)|\leq dc_2$, where $d$ is a constant depending only on $\sO(1)$.
Thus such sheaves $F$ form a bounded family, so $M(c_2)\setminus U_1$ is closed.
Next, let ${\mathcal E}\rightarrow T$ be a $T$-flat family of sheaves in $M(c_2)$ applying to Case II or III.
It induces a $T$-flat family ${\mathcal E}_{\bar{\eta}}$ of semistable sheaves on $X_{\bar{\eta}}$.
For $t\in T$, ${\mathcal E}_{\bar{\eta}}\otimes k(t)$ applies to Case III if and only if
$\dim_{\bar{k}(\PP^1)} \Hom(L\oplus L^{-1}, {\mathcal E}_{\bar{\eta}}\otimes_T k(t))\geq 2$
for some $L\in\Pic^0(X_{\bar{\eta}})$. Thus $M(c_2)\setminus U_{12}$ is closed.
\end{proof}
\begin{fact}\cite[Cor. 3.13]{Fri:vb-reg-ellipt}\label{fact:GenericVB-CaseI}
If $c_2> \max( 2p_g+1, 2p_g+(2/3)\Lambda(X))$, then general $H$-stable sheaves
in $M(c_2)$ are of type Case I.
\end{fact}
By \cite[Sect. 7]{Fri:vb-reg-ellipt},
if $X$ is generic and $c_2>\max(2(1+p_g), 2p_g(X)+(2/3)\Lambda(X))$, 
then there is such a dense open set $M_0$ of $M(c_2)$ as follows.
$M_0$ is contained in
the good locus $M_{gd}$ of $M(c_2)$ defined by
\begin{equation}\label{eq:GoodLcs}
 M_{gd}=\left\{ [E]\in M(c_2,H) \bigm| \ext^2(E,E)^0= 0 \right\}.
\end{equation}
Every sheaf $E$ in $M_0$ is locally-free and corresponds Case I in
Fact \ref{fact:RestrGenericFib}, and so it induces a double cover 
$C\rightarrow \PP^1$. $T=X\times_{\PP^1} C$ is non-singular, and gives
a double cover $\nu: T \rightarrow X$.
The decomposition of $E_{\eta'}$ at Case I extends to
an exact sequence on $T$
\begin{equation}\label{eq:relJH}
 0 \longrightarrow {\mathcal O}_T(D) \longrightarrow \nu^*E \longrightarrow
{\mathcal O}_T(-D)\otimes I_Z \longrightarrow 0,
\end{equation}
and $E$ is isomorphic to $\nu_* {\mathcal O}_T(-D)$.\par
%
The Jacobian surface $J(X)$ has a natural involution defined by $\times (-1)$,
and its quotient is
a smooth ruled surface $\FF_{2k} \rightarrow \PP^1$
called the $k$-th Hirzebruch surface \cite[p. 140]{BPV:text}, where $k=p_g+1$.
The divisor $D$ at \eqref{eq:relJH}
induces a morphism $C \rightarrow J(T)$.
The composition of this and
$J(T) \rightarrow J(X) \rightarrow \FF_{2k}$ is
invariant under $\operatorname{Gal}(k(C)/k(\PP^1))$-action, so it induces
a morphism $\PP^1 \rightarrow \FF_{2k}$, which is a section
$A$ of $\FF_{2k}\rightarrow \PP^1$.
This $A$ belongs to a linear system 
$|\sigma+(2k+r)l|\simeq \PP^{2c_2-2p_g-1}$ of $\FF_{2k}$, 
where $\sigma$ is a section of $\FF_{2k}$ with $\sigma^2=-2k$.
Thereby we get a morphism 
$\psi:M_0 \rightarrow \PP^{2c_2-2p_g-1}$ sending $E$ to $A$.
In fact, this is a surjective map onto a nonempty open subset 
$U\subset \PP^{2c_2-2p_g-1}$.
For detail, refer \cite[p.328]{Fri:vb-reg-ellipt}.
%
We recall statements in \cite[Sect. 7]{Frd:holvb}.
There are a $U$-flat subscheme ${\mathcal C}\subset U\times J(X)$ 
and ${\mathcal T}={\mathcal C}\times_{\PP^1} X$ such that
$U$-flat family ${\mathcal T}\rightarrow {\mathcal C}$
parametrizes $T=C\times_{\PP^1}X \rightarrow C$ above.
For some normal finite cover $U'\rightarrow U$,
there is a line bundle
${\mathcal O}({\mathcal D})$ on ${\mathcal T}'={\mathcal T}\times_U U'$
which parametrizes ${\mathcal O}_T(D)$ at \eqref{eq:relJH}.
In the relative Picard scheme $\Pic({\mathcal T}'/U')$, there is 
a subscheme $\Pic^s({\mathcal T}'/U')$ as follows by \cite[p. 329]{Fri:vb-reg-ellipt}.
Denote the fiber of ${\mathcal T}' \rightarrow {\mathcal C}\times_{U} U'$ and
$\Pic^s({\mathcal T}'/U') \rightarrow U'$ over $u'\in U'$ by
$T \rightarrow C$ and $\Pic^s(T)$.
By \cite[Lem. 7.4]{Fri:vb-reg-ellipt},
$\Pic^0(T)$ is isomorphic to $\Pic^0(C)$, and
$\Pic^s(T)$ is a principal homogeneous space under a group $\Pic^{\tau}(T)$,
which has a natural exact sequence
\begin{equation}\label{eq:Pic-tau}
 0 \longrightarrow \Pic^0(T) \longrightarrow \ \Pic^{\tau}(T) \longrightarrow G 
 \longrightarrow 0, 
\end{equation}
where $G$ is a finite subgroup of $H^2(T,\ZZ)_{\text{tors}}$.
We have a vector bundle over $X\times \Pic^s({\mathcal T}'/U')$
\begin{equation}\label{eq:mathcal-V}
 {\mathcal V}=q_*(\pi_1^*{\mathcal O}_{{\mathcal T}'}(-{\mathcal D})\otimes 
{\mathcal L}^{-1}),
\end{equation}
where ${\mathcal L}$ is the Poincare bundle of $\Pic^s({\mathcal T}'/U')$,
$\pi_1: {\mathcal T}'\times_{U'} \Pic^s({\mathcal T}'/U') \rightarrow {\mathcal T}'$
is natural projection, and 
$q: {\mathcal T}'\times_{U'} \Pic^s({\mathcal T}'/U') \rightarrow
X\times \Pic^s({\mathcal T}'/U')$
is induced by ${\mathcal T}'\rightarrow {\mathcal T}\rightarrow U\times X$.
\begin{fact}\cite[Cor.7.3]{Fri:vb-reg-ellipt}
${\mathcal V}$ gives an {\it isomorphism}
$\Pic^s({\mathcal T}'/U')/\sim\ \rightarrow M_0$,
where $\sim$ is the equivalence relation on $U'$ defined by $U'\rightarrow U$.
\end{fact}
%
\section{The $K$-dimension of moduli scheme of sheaves}\label{sctn:K-dim}
\begin{stng}\label{stng:XHS}
An elliptic surface $X$ with $\kappa(X)=1$ is generic, 
that is, it lies outside of a countable union of proper subvarieties of the
parameter space.
Fix a compact polyhedral cone ${\mathcal S}$ in the closure of the ample 
cone $\Amp(X)$ of $X$ such that ${\mathcal S}\cap \partial \Amp(X) \subset \RR K_X$.
%
$H$ is $c_2$-suitable and belongs to ${\mathcal S}$.

\end{stng}
\begin{lem}\label{lem:GoodLcs}
Assume that the good locus $M_{gd}$ of $M(c_2,H)$ defined at \eqref{eq:GoodLcs}
satisfies that 
$\operatorname{codim}(\bar{M}(c_2,H)\setminus M_{gd}, \bar{M}(c_2,H))\geq 2$;
this assumption holds
if ${\mathcal S}$ and $H$ satisfy Setting \ref{stng:XHS} and
if $c_2$ is sufficiently large w.r.t. $X$ and ${\mathcal S}$.
Then $\bar{M}(c_2, H)$ is of expected dimension, locally normal, l.c.i., 
and $M(c_2,H)$ is dense in it.
The canonical class $K_{\bar{M}}$ of $\bar{M}(c_2,H)$ is $\QQ$-Cartier and 
\[ K_{\bar{M}} = \lambda(u_1((2,0,c_2)))=\lambda(-2k_X) \quad
\text{in}\  \Pic(\bar{M}(c_2,H))_{\QQ},\]
where $\lambda(-2k_X)$ appeared in Subsection \ref{sbsctn:LBmoduli}.
\end{lem}
\begin{proof}
Fix a polarization $H_0\in {\mathcal S}$. By \cite{Li:kodaira},
if $c_2$ is sufficiently large w.r.t. $H_0$, then 
$\operatorname{codim}(\bar{M}(c_2,H_0)\setminus M_{gd}, \bar{M}(c_2,H_0))\geq 2$.
If $c_2$ is sufficiently large w.r.t. ${\mathcal S}$,
then $M(c_2,L)$ are mutually isomorphic in codimension one for every 
$L\in {\mathcal S}$ by \cite[Lem. 2.4.]{Yam:flip}, and thus
the first statement is valid.
Then the second statement follows from Subsection \ref{sbsctn:LBmoduli},
since $\chi(c_2\cdot k_X)=0$.
\end{proof}
On the Gieseker-Maruyama compactification $\bar{M}'(c_2)$ of $M(c_2)$,
$\sO(n_0 \lambda(-2k_X))$ is base-point free for some $n_0\in\NN$ by 
\cite[p.224]{HL:text}, and so there is a morphism 
$\Phi_{\lambda}: \bar{M}'(c_2) \rightarrow N'(c_2):=
\operatorname{Proj}\left(\oplus_{n\geq 0} H^0(\bar{M}, \sO(n\lambda(-2k_X)))\right)$
by Fact \ref{fact:IitakaPgrm} (2).
\begin{prop}\label{prop:plurican}
Let $M_0$ be the open subset of $M(c_2)$ in Section \ref{sbsctn:EvenFibDeg},
$\psi: M_0 \rightarrow U$ the morphism recalled at Subsection \ref{subsct:basic},
$\bar{M}'(c_2)$ an irreducible component of $\bar{M}'(c_2)$,
$N'(c_2)=\Phi_{\lambda}(\bar{M}'(c_2))$, and $M'_0=\bar{M}'(c_2)\cap M_0$.
By the Stein factorization,
we factor $\psi:M'_0\rightarrow U$ into $g\circ \tilde{\psi}$, where
$\tilde{\psi}: M'_0 \rightarrow V$ is a projective morphism with connected fibers,
and $g:V \rightarrow U$ is finite.
Then $V$ is normal and
there is a quasi-finite morphism $j: V \rightarrow N'(c_2)$ such that 
\begin{equation}\label{eq:Phi-psi}
\xymatrix{
M'_0 \ar@{^{(}-}[rr] \ar[dr]^{\tilde{\psi}} \ar[d]_{\psi} & & 
 \bar{M}'(c_2) \ar[d]^{\Phi_{\lambda}} \\
U & V \ar[l]_{g} \ar[r]^{j} & N'(c_2) }
\end{equation}
is commutative.
When $\bar{M}'(c_2)$ is integral, let
$n(j): V \rightarrow n(N'(c_2))$ denote the normalization of $j$.
Then $n(j)$ is an open immersion.
%
%
\end{prop}
\begin{proof}
We shorten $\Phi_{\lambda}$ to $\Phi$.
Let us show that $\lambda(-2k_X)|_Z=0$ for any fiber $Z$ of $\psi'$.
Let $F_j$ run
over the set of all multiple fibers of $X$ and $m_j$ is the multiplicity of $F_j$.
By \eqref{eq:CanBdleFormula}, one can verify in $K(X)$ that
\begin{equation}\label{eq:kx}
 k_X=\sO_X-\sO(-K_X)=(d-2)\sO_f+\sum_j \sum_{i=0}^{m_j-2} \sO(-iF_j)|_{F_j},
\end{equation}
since 
$\sO_f\cdot \sO_{F_j}=0$ and $\sO_{F_j}\cdot \sO_{F_k}=0$ when $j\neq k$.
Because of Section \ref{sbsctn:EvenFibDeg}, Lemma \ref{lem:GoodLcs} and
\eqref{eq:kx},
we have a morphism $\iota: \Pic^s({\mathcal T}'/U') \rightarrow M_0$,
a sheaf ${\mathcal V}$ on $X\times \Pic^s({\mathcal T}'/U')$ defined at
\eqref{eq:mathcal-V}.
Now we pick a point $u'\in U'$ over $u\in U$, and restrict these to the fiber over $u'$;
we shall abbreviate ``$(\cdot)\times_{U'}k(u')$'' to ``$(\cdot)_{u'}$''.
Denote ${\mathcal T}'_{u'}$ by $T$, that has a double covering map
$\nu:T\rightarrow X$.
Then we can regard $\Pic^s({\mathcal T}'/U')_{u'}=\Pic^s(T)$ as a subscheme
of $M_0$, and $Z$ as a connected component of $\Pic^s(T)$.
${\mathcal V}_{u'}=\nu_*(\sO_T(-D)\otimes {\mathcal L}|_{u'}^{-1})$
is a sheaf on $X\times \Pic^s(T)$, and
\begin{multline}\label{eq:KM|Z}
 \lambda(-2k_X)|_{\Pic^s(T)}= 2(2-d)\det R\pi_{2*} (\sO_f\otimes {\mathcal V}_{u'}) \\
 - \sum_j \sum_{i=0}^{m_j-2} 2\det R\pi_{2*}(\sO_{F_j}(-iF_j)\otimes {\mathcal V}_{u'}) 
\end{multline}
by virtue of \eqref{eq:kx}. 
Now we take any multiple fiber $F$ of $X$.
Because of the definition of $M_0$ at \cite[p.328]{Fri:vb-reg-ellipt},
$\nu: T \rightarrow X$ is unramified near $F$ and thereby
it holds either that
(i) $\nu^{-1}(F)$ is splitting as $F^{(1)}\coprod F^{(2)}$, or that
(ii) $\nu^{-1}(F)$ is connected and $\nu: \nu^{-1}(F)\rightarrow F$ is \'{e}tale.
We consider in Case (i); the proof goes similarly in Case (ii).
A line bundle ${\mathcal L}_{u'}|_{F^{(k)}}$ on $F^{(k)}\times \Pic^s(T)$
gives a morphism $\tau_k:\Pic^s(T) \rightarrow \Pic(F^{(k)})$.
\begin{clm}\label{clm:1pt}
The set $\tau_k(Z)$ is a point. Thus
${\mathcal L}|_{F^{(k)}\times Z}$ is isomorphic to 
$\pi_1^*(L_k) \otimes \pi_2^*(\sO_Z(G_k))$,
where $L_k$ and $\sO(G_k)$ are line bundles on $F^{(k)}$ and $Z$, respectively.
\end{clm}
As reviewed in Subsection \ref{sbsctn:EvenFibDeg},
${\mathcal L}_{u'}$ is a Poincare bundle of $\Pic^s(T)$, and $Z$ is
a principal homogeneous space under
a group $\Pic^0(T)$, that is isomorphic to $\Pic^0(C)$.
Thereby ${\mathcal L}_{s_1}|_{F^{(k)}} \simeq {\mathcal L}_{s_2}|_{F^{(k)}}$ for
$s_1,\ s_2\in Z$, and then we get Claim \ref{clm:1pt}.
Because of this claim, it holds that
for the projection $\pi_2: F\times Z \rightarrow Z$
\begin{multline*}
 \det R\pi_{2*}(\sO_F(-iF)\otimes {\mathcal V}|_Z)=\\
 \otimes_{k=1}^2 \det R\pi_{2*}\bigl(\pi_{1}^* \bigl(\sO_X(-iF)\otimes 
\sO_T(-D)|_{F^{(k)}}\otimes L_k \bigr) \otimes \pi_2^*(\sO_Z(G_k))\bigr)=
 \otimes_{k=1}^2 R(k)\sO_Z(G_k),
\end{multline*}
where $R(k)=\chi(F^{(k)}, \sO_X(-iF)\otimes \sO_T(-D)|_{F^{(k)}}\otimes L_k)$.
However, the degree of $\sO_X(-iF)\otimes \sO_T(-D)|_{F^{(k)}}\otimes L_k$
is zero by its construction, so $R(k)=0$.
From \eqref{eq:KM|Z}, we get $\lambda(-2k_X)|_Z=0$.
Fibers of $\tilde{\psi}$ are connected Abelian varieties with the same dimension
as mentioned in Sect. \ref{sbsctn:EvenFibDeg}.
From the base change theorem \cite[Thm. III.12.9]{Ha:text}
and the fact $\lambda(-2k_X)|_Z=0$ for fibers $Z$ of $\tilde{\psi}$,
one can check that $\tilde{\psi}_*(\sO(\lambda(-2k)))$ is a line bundle on $V$,
$\tilde{\psi}_*(\sO(\lambda(-2k)))\otimes_V k(t) \simeq H^0(\tilde{\psi}^{-1}(t),\lambda(-2k)|_{k(t)})$
for $t\in V$, and a natural map
$\tilde{\psi}^{*}\tilde{\psi}_* (\lambda(-2k_X)) \rightarrow \lambda(-2k_X)|_{M_0}$ 
is isomorphic.
As a result, $\lambda(-2k_X)|_{M'_0}\simeq \tilde{\psi}^* L_V$ 
with a line bundle $L_V$ on $V$, and
$H^0(M'_0, n\lambda(-2k_X)|_{M'_0})=H^0(M'_0, \tilde{\psi}^*(nL_V))=
H^0(V,\tilde{\psi}_* \sO_{M'_0}\otimes nL_V)=H^0(V,nL_V)$.
Thereby a finite-dimensional subspace $W(n) \subset H^0(V,nL_V)$ defines a rational
map $|W(n)|: V \dashrightarrow \PP^N$ such that
$|W(n)|\circ \tilde{\psi}:M'_0 \rightarrow V \dashrightarrow \PP^N$ equals to
the restriction of $|n\lambda(-2k_X)|:\bar{M}'(c_2) \rightarrow \PP^N$ to $M'_0$.
When $n$ is sufficiently large and divisible,
one can check that $|W(n)|$ gives such a morphism $j$ as \eqref{eq:Phi-psi} 
is commutative.\par
%
Suppose that $s_1,\ s_2\in M'_0$ satisfies that $\psi(s_1)\neq \psi(s_2)$.
By the definition of $\psi$ (\cite[p.307, (4.5)]{Fri:vb-reg-ellipt}),
the restriction of $E_{s_1},\ E_{s_2}$ to general fiber of $X\rightarrow \PP^1$
are semistable but not S-equivalent.
Then \cite[p.224]{HL:text} deduces that
some element of $\Gamma(\bar{M}'(c_2),N\lambda(-2k_X))$ separates $s_1$ and $s_2$,
and thus $\Phi(s_1)\neq \Phi(s_2)$. 
Accordingly, it holds for any $s\in M'_0$ that
\begin{equation}\label{eq:3fibers}
 \psi^{-1}(\psi(s)) \supset \Phi^{-1}(\Phi(s)) \supset  \tilde{\psi}^{-1}(\tilde{\psi}(s)). 
\end{equation}
$\tilde{\psi}^{-1}(\tilde{\psi}(s))$ is a connected component
of $\psi^{-1}(\psi(s))$ by the definition of $\tilde{\psi}$, so $j$ is quasi-finite.
Since $M'_0$ is non-singular, $\tilde{\psi}$ factors as $n(\tilde{\psi}):M'_0\rightarrow n(V)$
and natural morphism $\iota:n(V)\rightarrow V$.
By the construction of $\tilde{\psi}$, $\sO_V= \tilde{\psi}_*\sO_{M'_0}= 
\iota_*( n(\tilde{\psi})_*\sO_{M'_0}) \supset \iota_* \sO_{n(V)} \supset \sO_V$.
This deduces that $\iota_* \sO_{n(V)}=\sO_V$, so $\iota$ is isomorphic
and $V$ is normal.
When $\bar{M}'(c_2)$ is integral, general fiber of $\Phi$ is
connected by \cite[Thm. 10.3]{Iitaka:AG}, and then
$\Phi^{-1}(\Phi(s)) \supset  \tilde{\psi}^{-1}(\tilde{\psi}(s)) $ for general $s\in M'_0$
by \eqref{eq:3fibers}, so $j$ is generically injective.
$N'(c_2)$ is integral so
its normalization $n(N'(c_2))$ is defined, and $j$ and $\Phi$ induce their normalization
$n(j): V\rightarrow n(N'(c_2))$ and $n(\Phi):\bar{M}'(c_2)\rightarrow n(N'(c_2))$
such that $n(j)\circ \tilde{\psi}: M'_0 \rightarrow n(N'(c_2))$ equals $n(\Phi)|_{M'_0}$.
Every fiber of $n(\Phi)$ is connected from \cite[Thm. 10.3]{Iitaka:AG}
and \cite[Cor. 4.3.12]{EGA:III-1}, and thereby $n(j)$ is injective.
By \cite[Thm. III.10.7]{Ha:text} $n(j)$ is generically smooth, so 
$n(j)$ is birational by \cite[Thm. 17.9.1]{EGA:IV-4} and then
$n(j)$ is an open immersion by \cite[Cor. 4.4.9.]{EGA:III-1}.
%
\end{proof}
\begin{cor}\label{cor:KodDim-Moduli}
In Setting \ref{stng:XHS}, the following holds:\\
(i) Let $\bar{M}'(c_2)\subset \bar{M}(c_2)$ be arbitrary connected component
of the Gieseker--Maruyama compactification of $M(c_2)$.
When $c_2>\max(2(1+p_g), 2p_g(X)+(2/3)\Lambda(X))$, 
$M(c_2)$ is of expected dimension, generically smooth, and
$ \kappa\left( \lambda(-2k_X),\ \bar{M}'(c_2) \right)=(\dim M(c_2)+1-p_g(X))/2.$\\
(ii) Suppose $c_2$ is sufficiently large w.r.t. $X$ and ${\mathcal S}$.
Then conclusions in Lemma \ref{lem:GoodLcs} hold and $\bar{M}(c_2)$ is irreducible.
The abundance (Fact \ref{fact:IitakaPgrm} (3)) holds on $\bar{M}(c_2)$.
Let $\tilde{\psi}$ be the morphism at Prop. \ref{prop:plurican},
which is obtained from the Stein factorization of
Friedman's morphism $\psi$ in Sect. \ref{sbsctn:EvenFibDeg}, and
$\Phi=\Phi_{|mK|}$  for a large and divisible number $m\in\NN$.
Then there are a normal variety $V$ and a morphism $j$ such that
\begin{equation}\label{eq:psid-Phi}
\xymatrix{
M_0 \ar@{^{(}-}[rr] \ar[dr]^{\tilde{\psi}} \ar[d]_{\psi} & & \bar{M}(c_2) 
\ar[d]^{\Phi} \\
U & V \ar[l]_{g} \ar[r]^{j} & \bar{M}(c_2)_{can} }
\end{equation}
is commutative and that its normalization
$n(j):V \rightarrow n(\bar{M}(c_2)_{can})$ is an open immersion.
For a general member $E$ of $M(c_2)$, $\Phi^{-1}\Phi(E)$ equals the connected
component of
\[ \left\{ E'\in M(c_2) \bigm| E'|_f\simeq E|_f\ \text{for general fiber 
$f$ of $\pi$}\right\}\]
containing $E$. It is the Jacobian of a hyperelliptic curve.
The $K$-dimension $\kappa\left(K,\bar{M}(c_2)\right)$ equals 
$(\dim M(c_2)+1-p_g(X))/2$.\\
(iii) In addition to the assumptions in (ii), suppose that
all singularities of $\bar{M}(c_2)$ are canonical.
Then the Kodaira dimension
$\kappa\left( \bar{M}(c_2) \right)$ is $(\dim M(c_2)+1-p_g(X))/2$.
\end{cor}
\begin{proof}
(i) As reviewed at Section \ref{sbsctn:EvenFibDeg},
$M_0$ is open and dense in $M(c_2)$, and is contained in
the good locus $M_{gd}$ of $M(c_2)$, so
$M(c_2)$ is of expected dimension.
From Proposition \ref{prop:plurican} and Section \ref{sbsctn:EvenFibDeg}, 
$\kappa\left( \lambda(-2k_X),\ \bar{M}'(c_2) \right)=\dim\Phi_{\lambda}$ equals
$\dim\psi=\dim(U)=2c_2+2p_g-1= (\operatorname{exp dim}(M(c_2))+1-p_g)/2$.
Item (ii) results from (i), Lemma \ref{lem:GoodLcs}, Prop. \ref{prop:plurican}
and its foremost part, Fact \ref{fact:IitakaPgrm} (4), Prop. \ref{prop:irred}
and \cite[p.328, line 22]{Fri:vb-reg-ellipt}.
Item (iii) results from (ii) and Defn. \ref{def:D-dim} (2).
\end{proof}

\section{A sufficient condition for singularities to be canonical}\label{sctn:canonical}
\begin{thm}\label{thm:2k+1-CanSing}
Let $(R,p)$ be a local ring that is smooth over $\CC$ and
$I$ be its ideal generated by 
$f_1,\dots, f_k \in {\frak m}_p^2$.
Let $B_g$ 
designate the bilinear form on $({\frak m}_p/{\frak m}^2_p)^{\vee}$
induced by an element $g$ of ${\frak m}_p^2$.
Suppose that any nonzero $\CC$-linear combination $g$ of $f_1,\dots,f_k$
satisfies $\rk B_g\geq 2k+1$.
Then $R/I$ is c.i., normal, and $p$ is its canonical singularity.
\end{thm}
\begin{rem}\label{rem:a-inv}
After this section was written, Prof. Masataka Tomari kindly teached to the author
that, by using $a$-invariant mentioned in \cite{GotoW-78}, 
one can prove Theorem \ref{thm:2k+1-CanSing} in another way.
The proof presented here is not so long, and so elementary that one can read it without 
advanced knowledge on ring theory.
Thus we here adopt this proof without change.
%
\end{rem}

%
\begin{proof}
Let $\Rc$ be the completion of $R$ at $p$, and let $\Rc[z]$ denote $\Rc\otimes \CC[z]$.
Since $R$ is smooth over $\CC$, $\Rc$ is isomorphic to 
the formal power series ring $\CC[[x_1,\dots, x_N]]$.
We can regard $f_i$ as a power series $f_i(x_1,\dots, x_N)$, and then
put $f_i(zx_1,\dots, zx_N)/z^2$ as $\bfi$,
which belongs to $\Rc[z]$ since $f_i\in {\frak m}^2_p$.
Then one can define a ideal ${\mathcal I}=\langle \bar{f}_1,\dots, \bar{f}_k \rangle$ 
of $\Rc[z]$ and a ring ${\mathcal S}=\Rc[z]/{\mathcal I}$ over $\CC[z]$.
For $z_0\in \Aa^1$, 
we denote the fiber of
$\Rc[z],\ \bfi(z)\in\Rc[z]$ (resp. ${\mathcal S}$)
over $\CC[z]\rightarrow k(z_0)$ as
$\Rc[z_0],\ \bfi(z_0)$ (resp. ${\mathcal S}_{k(z_0)}$).
\begin{rem}\label{rem:GeneralFib}
If $z_0\in \Aa^1$ is not zero, then the correspondence
$x_i\mapsto z_0x_i$ gives the isomorphism
of ${\mathcal S}_{k(z_0)}$ and ${\mathcal S}_{k(1)}$, that is simply
the completion of $R/I$ at $p$.  
\end{rem}
\begin{clm}\label{clm:S-A1-flat}
${\mathcal S}$ is flat over $\CC[z]$, and
its fibers are of dimension $N-k$ and normal.
\end{clm}
\begin{proof}
$\bfi(0)\in \Rc$ is degree-two homogeneous polynomial with variables $x_i$,
so we can define a closed subscheme $S$ of $\Aa^N$ as
$S=\Spec(\CC[x_1,\dots,x_N]/\langle \bar{f}_1(0),\dots, \bar{f}_k(0)\rangle)$.
The singular locus of $S$ is contained in
\[R_{<k}= \left\{  a\in \Aa^N \bigm| 
 \rk(\partial \bar{f}_i(0)( a)/\partial x_j)_{i,j} 
\leq k-1 \right\}, \]
where $(\partial \bar{f}_i(0)( a)/\partial x_j)_{i,j}$ is the Jacobian matrix
of $S$ at $a$ (\cite[Thm. 4.1.7.]{Ishii:text}).
If $a\in R_{<k}$, then there is a $\CC$-linear combination $g\neq 0$ of $\bar{f}_i(0)$
such that
\begin{equation}\label{eq:Sing-f}
 {}^t(\partial g(a)/\partial x_1,\dots, \partial g(a)/\partial x_N)={\bf 0}.
\end{equation}
Let $V$ denote the $\CC$-vector space spanned by $\bar{f}_i(0)\ (i=1,\dots,k)$,
and define the set
\[ \PP(V^{\vee})\times \Aa^N \supset W:= \{ (g,a) \bigm| \partial g(a)/\partial x_i=0
\quad (1\leq \forall i\leq N).\}\]
By \eqref{eq:Sing-f}, it holds that $\dim W\geq \dim R_{<k}\geq \dim\Sing(S)$.
For $g\in V$, the inverse image of $[g]\in\PP(V^{\vee})$ by
$W \subset \PP(V^{\vee})\times \Aa^N \overset{\pi_1}{\longrightarrow} \PP(V^{\vee})$
is of dimension $N-\rk(B_g)$, where $B_g$ is the bilinear form induced by
degree-two homogeneous polynomial $g$,
since one can present $g$ as
$g=x_1^2+\dots x_{\rk B(g)}^2$ by choosing $x_i$ suitably.
Moreover,
$\rk(B_g)\geq 2k+1$ by assumptions in Theorem \ref{thm:2k+1-CanSing}.
Thus $\dim W \leq k-1+N-(2k+1)=N-k-2$, and then
\[ \dim S-\dim\Sing(S) \geq \dim S- \dim R_{<k} \geq N-k-(N-k-2)=2.\]
Hence
$S$ is regular in codimension one, and of dimension $N-k$.
Because ${\mathcal S}_{k(0)}$ is the completion of $S$ at $p$,
${\mathcal S}_{k(0)}$ is regular in codimension one, of dimension $N-k$,
and so is normal by \cite[Thm. 23.8]{Mat:text}.
From the upper-semicontinuity of dimension of fibers, openness of regularity
and Remark \ref{rem:GeneralFib},
${\mathcal S}_{k(z_0)}$ is of dimension $N-k$,
regular in codimension one and normal for all $z_0\in \Aa^1$.
\end{proof}
Now let us show Theorem \ref{thm:2k+1-CanSing} by induction on $k$.
When $k=0$, $R/I$ is non-singular at $p$ and the statement is obvious.
Next, suppose that this theorem is true when $k\leq k_0-1$.
From Fact \ref{fact:Can=Rat} and Remark \ref{rem:RatAnalytic}, 
we only have to show that $(R/I,p)$ is a canonical singularity 
when $k=k_0$, $R=\CC[x_1,\dots,x_N],\ p=(0,\dots,0)=: 0_N\in\Spec(R)$ 
and $f_i$ are degree-two homogeneous polynomial
by the same reason as in the proof of Claim \ref{clm:S-A1-flat}.
The ideal of $R=\CC[x_1,\dots,x_N]$ generated by
$\partial f_i/\partial x_j\ (1\leq i\leq k_0,\ 1\leq j\leq N)$ defines
a non-singular closed subscheme $S$ of $\Spec(R/I)$.
Let $\psi: U_1 \rightarrow U=\Spec(R)$ be the blowing-up along $S$, and
let $\sigma:M_1 \rightarrow M=\Spec(R/I)$ be the strict transform of $\Spec(R/I)$.
\begin{clm}\label{clm:M1-cansing}
Every point $q$ of $\sigma^{-1}(0_N)$
is at worst a canonical singularity of $M_1$.
\end{clm}
\begin{proof}
By using terms at \cite[Prop. II.2.5.]{Ha:text}
$U_1$ is covered by $D_+(\partial f_i/\partial x_j)$,
where $1\leq i\leq k_0,\ 1\leq j\leq N$.
By choosing $x_i$ fitly, one can describe $f_1$ as $f_1=x_l^2+\dots+x_N^2$,
and the ideal of the center $S$ as 
$I_S=\langle x_{l'},\dots, x_N \rangle$, where $l'\leq l\leq N$.
We shall show this claim on $D_+(\partial f_1/\partial x_N)=D_+(x_N)$.
$D_+(x_N)$ is an affine scheme with ring 
\[ R_1=\CC[x_1,\dots,x_{l'-1}, y_{l'},\dots, y_{N-1}, x_N] \qquad
 (y_j=(x_j/x_N) \text{ for } l'\leq j\leq n-1).\]
Let $\tilde{f_i}\in R_1$ denote the strict transform of $f_i$.
When $i=1$, it holds that
\begin{equation}\label{eq:tld-f-1}
\tilde{f}_1=1+y_l^2+\dots+ y_{N-1}^2=f_1/x_N^2. 
\end{equation}
When $i\geq 2$, one can describe $f_i=z_1^2+\dots z_{r_i}^2$, where
$z_j$ are $\CC$-linear combinations of $x_1,\dots,x_N$ and linearly independent.
In fact $z_j$ are linear combinations of $x_{l'},\dots, x_N$.\par
%
%
Thereby $\tilde{f}_i$ equals $f_i/x^2_N$, and is a polynomial with variables 
$y_{l'},\dots,y_{N-1}$ of degree $\leq 2$.
One can assume that $\tilde{f}_i(y_{l'},\dots y_{N-1})$ satisfies
$\tilde{f}_i(0,\dots,0)=0$ by replacing $\tilde{f}_i$ to 
$\tilde{f}_i+\lambda\tilde{f}_1$ if necessary.\par
Take any closed point $q=(q_{l'},\dots, q_{N-1})$ of 
$\sigma^{-1}(0_N)\subset \psi^{-1}(0_N)=\Spec( \CC[y_{l'},\dots,y_{N-1}])$.
It holds that $\tilde{f}_i(q)=0$ for $i=1,\dots, k$ since $\tilde{f}_i$ obviously 
belongs to the ideal of $M_1 \subset U_1$.
By \eqref{eq:tld-f-1}, one can pick some $m$ with $1\leq m\leq k$ such that
$\tilde{f}_1,\dots, \tilde{f}_m$ are linearly independent in 
${\frak m}_q/{\frak m}_q^2$ and that
$\tilde{f}_{m+1},\dots \tilde{f}_{k}\in {\frak m}_q^2$,
where ${\frak m}_q$ means the maximal ideal of $q\in\psi^{-1}(0_N)$ in $U_1$.
%
%
In case where $m<k$, let us look over $f_L\ (m<L\leq k)$ further.
In describing $f_L$ as $f_L=\sum_{l'\leq i,j\leq N} \lambda_{ij}x_ix_j$,
remark that $\lambda_{NN}=0$ since $\tilde{f}_l(0,\dots,0)=0$.
Since one can replace $x_{l'},\dots, x_{N-1}$ by the natural action of 
$\operatorname{GL}(N-l')$, one can express $f_L$ as
\begin{align}\label{eq:exprs-fk}
f_L= & \lambda_{N,N-1}\cdot x_{N-1}x_N + \sum_{l'\leq i\leq N-2} \lambda_i\cdot x_i^2
 \qquad (\lambda_{N-1,N}\neq 0) \qquad \text{or} \notag\\
f_L= & \sum_{l'\leq i\leq N-1} \lambda_i\cdot x_i^2.
\end{align}
However, in the former case, one can check that $\tilde{f}_L$
is everywhere non-singular, and thus it can not belong to ${\frak m}_q^2$.
This is contradiction, and so $f_L$ is expressed as in \eqref{eq:exprs-fk}.
Then $\tilde{f}_L=\sum_{i=l'}^{N-1} \lambda_i\cdot y_i^2$, and besides
$\tilde{f}_L=\sum_{i=l'}^{N-1} \lambda_i\cdot (y_i-q_i)^2$
since $\tilde{f}_L \in {\frak m}_q^2$.
Thus $\rk \tilde{B}_L=\rk B_L$, where
$B_L$ is the bilinear form on $({\frak m}_0/{\frak m}_0^2) ^{\vee}$
induced by $f_L$, where ${\frak m}_0$ is the ideal of $0\in U$, and 
$\tilde{B}_L$ is the bilinear form on $({\frak m}_q/{\frak m}_q^2)^{\vee}$
induced by $\tilde{f}_L$, where ${\frak m}_q$ is the ideal of $q\in U_1$.
Thereby, for the quotient vector space $V$ of ${\frak m}_q/{\frak m}_q^2$
by $\langle \tilde{f}_1,\dots,\tilde{f}_m \rangle_{\CC}$, 
\[\rk \tilde{B}_L|_V +2m \geq \rk \tilde{B}_L= \rk B_L \geq 2k+1 \qquad 
\text{for all $L$ such that}\ m<L\leq k\] 
from Remark \ref{rem:rkBV-rkB} below and assumptions in Theorem \ref{thm:2k+1-CanSing}.
In such a way, we can show the following:\par
(*) Let
$\tilde{g}$ be any nonzero $\CC$-linear combination of 
$\tilde{f}_{m+1},\dots,\tilde{f}_k$.
Since $\tilde{g}$ belongs to ${\frak m}_q^2$, it induces a bilinear form
on $V^{\vee}$, and then its rank is $2(k-m)+1$ or more.\par
Because of the choice of $m$,
$R'_1:=R_{1,q}/\langle \tilde{f}_1,\dots,\tilde{f}_m \rangle$ is a regular local ring.
Since $k-m<k$, the local ring
$R'_1/\langle \tilde{f}_{m+1},\dots, \tilde{f}_{k} \rangle=
R_{1,q}/\langle \tilde{f}_1,\dots,\tilde{f}_k \rangle$
is c.i., normal, and its closed point $q$ is at worst its canonical singularity
by inductive hypothesis and the fact (*).
A natural surjective map
$R_{1,q}/\langle \tilde{f}_1,\dots,\tilde{f}_k \rangle \rightarrow \sO_{M_{1,q}}$
is isomorphic, since they are integral and some their nonempty open subsets are
birational to $M$. Consequently we obtain Claim \ref{clm:M1-cansing}.
\end{proof}
Next, let us calculate $K_{M_1}-\sigma^*(K_M)$.
Concerning $\psi: U_1 \rightarrow U$ with exceptional divisor $E$, it holds that
\begin{equation}\label{eq:KU1-KU}
 K_{U_1}-\psi^*(K_U)=(\dim U-1-\dim S)\cdot E
\end{equation}
by \cite[Thm. 6.1.7.]{Ishii:text}.
Because $M$ (resp. $M_1$) is a l.c.i. and normal subscheme of the nonsingular scheme
$U$ (resp. $U_1$) by Claim \ref{clm:S-A1-flat},
and because $x_N$ is the generator of the ideal of $E$
on $D_+(x_N)$, the adjunction formula deduces that
\begin{equation}\label{eq:KM-KU}
\sO(K_M)=\Bigl. \sO(K_U+\sum_{i=1}^k \langle f_i\rangle) \Bigr|_M \ \text{and} \
\sO(K_{M_1})=\Bigl. 
\sO(K_{U_1}+\sum_{i=1}^k \langle \tilde{f}_i=(f_i/\iota_E^2) \rangle)\Bigr|_{M_1},
\end{equation}
where $\iota_E$ is a generator of the ideal of $E$.
From \eqref{eq:KU1-KU} and \eqref{eq:KM-KU},
$ K_{M_1}-\sigma^*K_M =(\dim U-1-\dim S-2k)E.$
By its definition, $S$ is contained in $\Sing(f_1)$, so
$\dim S\leq \dim\Sing(f_1)\leq \dim U -(2k+1)$ from assumptions 
in Theorem \ref{thm:2k+1-CanSing}, and hence
$\dim U-1-\dim S-2k\geq 0$.
Accordingly the divisor $K_{M_1}-\sigma^*(K_M)$ is positive, 
and thereby $p$ is a canonical singularity of $M$ from Claim \ref{clm:M1-cansing}.
We have completed the proof of Theorem \ref{thm:2k+1-CanSing}.
\end{proof}

\begin{rem}\label{rem:rkBV-rkB}
Let $W$ be a $\CC$-vector space, $V$ a quotient vector space of $W$,
$B$ be a bilinear form on $W^{\vee}$ and $B|_V$ the induced bilinear
form on $V^{\vee}$. Then $\rk B|_V +2(\rk W-\rk V)\geq \rk B$.
\end{rem}
\begin{proof}
We verify this when $\rk W-\rk V=1$; In general case, one can show this
by induction on $\rk W-\rk V$.
Let $w$ be the generator of $\Ker(W \rightarrow V)$.
If $B(w,w)=\lambda\neq 0$, then $W$ is naturally identified with 
$\CC\cdot w\oplus \Ker B(w,\cdot)$,
and $B$ is naturally decomposed into $\lambda\id \oplus B|_V$, and so
$\rk B\leq 1+\rk B|_V$.
If $B(w,w)=0$, then one can express $B$ as a matrix
\begin{equation*}
\left(
\begin{array}{c|cc}
0 & 0 & * \\
\hline
0 & E_{\rk B|_V} & 0 \\
* & 0 & 0
\end{array}
\right)
\end{equation*}
according to $W\simeq \CC\cdot w\oplus V$, and thus
$\rk B\leq \rk B|_V+2$.
\end{proof}
\begin{thm}\label{thm:suffcdtn-moduli-can}
Let $E$ be a stable sheaf on a non-singular projective surface.
Suppose that any non-zero homomorphism $f\in \Hom(E,E(K_X))^{\circ}$
satisfies that the rank of
\begin{equation}\label{eq:def-H1f}
H^1(\ad(f)): \Ext^1(E,E) \rightarrow \Ext^1(E,E(K_X))\qquad 
(\alpha \mapsto f\circ\alpha-\alpha\circ f)
\end{equation}
is $2\ext^2(E,E)^{\circ}+1$ or more.
Then the moduli scheme of stable sheaves on $X$ is l.c.i., normal at $E$
and $E$ is at worst its canonical singularity.
\end{thm}
\begin{proof}
This results from Theorem \ref{thm:2k+1-CanSing} and
Fact \ref{fact:moduli-lci}.
\end{proof}

\vspace{3pt}

\section{Rank of $H^1(\ad(f))$ in Case I}\label{sctn:rkH1-caseA1}
Let $E$ be a singular point of $M(c_2)$ and hence
there is a non-zero traceless homomorphism $f:E \rightarrow E(K_X)$.
\begin{defn}\label{defn:B}
Let $B$ denote the largest effective divisor on $X$ such that
$f$ splits into
$f: E^{\vee\vee} \rightarrow E^{\vee\vee}(K_X-B) \hookrightarrow E^{\vee\vee}(K_X)$,
where the latter map is a natural injection.
Since $H$ is $c_2$-suitable, $B$ is supported on fibers of $\pi$, and so
it is a rational multiple of $\cf$ by Setting \ref{stng:XHS}.
\end{defn}
Here we make the following assumption.
\begin{assumption}\label{assump:I}
The sheaf $E$ comes under Case I in Fact \ref{fact:RestrGenericFib}.
\end{assumption}
We try to estimate the rank of
$H^1(\ad(f)): \Ext^1(E,E)\rightarrow \Ext^1(E,E(K_X))$ from below
in view of Theorem \ref{thm:suffcdtn-moduli-can}. 
%
\begin{lem}\label{lem:DetNeq0-I}
Under Assumption \ref{assump:I},
any non-zero traceless homomorphism $f:E \rightarrow E(K_X)$
satisfies that $\det(f)\neq 0$.
\end{lem}
\begin{proof}
If $\det(f)=0$, then $f^2=0$ by Hamilton-Caylay's theorem. 
Thus we have a natural injection $Im(f)\subset Ker(f)(K_X)$,
$Ker(f)$ and $Im(f)$ are of rank one, 
and so their fiber degrees are zero by $c_2$-suitability of $H$,
but this contradicts to Assumption \ref{assump:I}.
%
\end{proof}
We can split $\det f \in \Gamma(\sO(2K_X))$ as $\det f = \alpha \tau^2$,
where $B'$ is an effective divisor on $X$, 
$\tau=\tau_{B+B'}\in\Gamma({\mathcal O}(B+B'))$ is the natural section, 
a line bundle ${\mathcal L}={\mathcal O}(K_X-B-B')$ on $X$,
and $\alpha \in \Gamma(X, {\mathcal L}^2)$ is square-free.
Remark that $B'\in \QQ\cdot\cf$ by Setting \ref{stng:XHS}.
By \cite[Sect. I.17]{BPV:text}, $\alpha$ induces 
a flat double covering $\nu_0: Y_0 \rightarrow X$
from a normal surface $Y_0$ and
a section $s$ in $\Gamma(Y_0, \nu_0^*{\mathcal L})$ such that
$s^2+\nu_0^*\alpha=0$. 
The divisor $Z(s)$ given by $s$ is locally integral since the support of
$\alpha$ is so by Setting \ref{stng:XandH}.
Remark that $Y_0$ is connected,
since $\alpha$ has no square root in $\Gamma(X, \sL)$ from Assumption \ref{assump:I}.
Recall $\eta'$ and $C$ at Fact \ref{fact:RestrGenericFib}.
By their definitions, $Y_{0 \bar{\eta}}$ is defined over $\eta'$, so there
is a morphism $Y_{0 \eta}\rightarrow \eta'$, and it extends to $Y_0\rightarrow C$
since $C \rightarrow \PP^1$ is finite and $Y_0$ is normal.
%
%
Let $\phi:Y \rightarrow Y_0$ be the canonical resolution of singularities
(\cite[Sect. III.7]{BPV:text}).
\begin{equation}\label{eq:Y0-C}
\xymatrix{
Y_{0 \eta} \ar[r] \ar@{.>}[d]  & X_{\eta} \ar[d]& 
 &  Y \ar[r]^{\phi} & 
 Y_0 \ar[r]^{\nu_0} \ar[d]_{\pi_C} & X \ar[d]^{\pi} \\
\eta' \ar[r] & \eta & &    & C \ar[r]^{\nu'} & \PP^1}
\end{equation}
Denote $\nu_0\circ \phi$ by $\nu: Y \rightarrow X$.
Note that $\pi_C: Y \rightarrow C$ is an elliptic fibration, and that
$Y_0$ ($C$ and $Y$, resp.) has a natural action $\sigma$ of $\ZZ/2$,
since it is a double covering over $X$ ($\PP^1$ and a blowing up of $X$, resp.).
\begin{lem}\label{lem:Y0-Cartier}
(i) Every singularity of $Y_0$ is rational (Defn. \ref{def:D-dim}). \\
(ii) $K_Y=\nu^*(K_X\otimes \sL)$ and $\chi({\mathcal O}_Y)=2\chi(\Ox)=2d$.\\
(iii) Every Weil divisor on $Y_0$ is Cartier.
\end{lem}
\begin{proof}
(i) results from \cite[Prop. III.3.4, Thm. III.7.1]{BPV:text}, and
(ii) does from \cite[Sect. V.22]{BPV:text}.
As to (iii), let $D$ be a positive Weil divisor on $Y_0$. 
Since $Y_0\setminus Z(s)$ is non-singular,
for any $p\in Y_0$, one can find some open set $p\in U$ such that 
$Z(s)\cap U$ is integral and $D|_{U\setminus Z(s)}$ is
given by $\lambda \in S^{-1}A(U)$, where $S$ is the multiplicative set generated by $s$.
Then $s^l \lambda$ belongs to a Cartier divisor $D_0$ on $U$ for some $l\in \ZZ$.
Since $D-D_0|_{U\setminus Z(s)}=0$, $D-D_0=nZ(s)$ for some $n\in \ZZ$ by 
\cite[Example II.6.6.2]{Ha:text}, and thus $D=D_0+nZ(s)$ is a Cartier divisor on $U$.
%
\end{proof}
{\footnotesize 
\begin{rem}
If singular fibers on $X$ are $(I_k)$ (not necessarily $(I_1)$)
or $(mI_0)$, then $2D$ is Cartier for every Weil divisor $D$ on $Y_0$
(cf. \cite[Example II.6.5.2.]{Ha:text}).
\end{rem}
}
%
Since $\nu_0^*f$ has eigenvalues $\pm s\tau$, we have two exact sequences
on $Y_0$:
\begin{equation}\label{eq:defn-Fpm}
0 \longrightarrow F_{\pm}:=Ker(\nu_0^*f \pm s\tau) \longrightarrow
 \nu_0^* E \longrightarrow G_{\pm}:=Im(\nu^*_0 f\pm s\tau)
\longrightarrow 0.
\end{equation}
Since $(\nu_0^*f)^2-s^2\tau^2= (\nu_0^*f)^2 +\det f=0$ by 
theorem of Hamilton-Cayley, $G_{\pm}$ naturally becomes a subsheaf
of $F_{\mp}(K_X)$.
Let $D$ be the first Chern class of $F_-$, that is a Cartier divisor by
Lemma \ref{lem:Y0-Cartier}.
\begin{lem}\label{lem:D+sD}
We have a natural $\ZZ/2$-equivariant isomorphism
$D+\sigma(D)+K_X-B=0$ in $\Pic(Y_0)$.
Here $\ZZ/2$ acts on $D+\sigma(D)$ naturally, and
on $K_X-B$ and $\sO$ trivially.
\end{lem}
\begin{proof}
Assume that $E$ is locally free; the proof goes similarly in general case.
For a natural section $\tau_{B'}\in\Gamma(X,{\mathcal O}(B'))$,
$\nu_0^*f+s\tau_{B'}:\nu_0^*E \rightarrow \nu_0^*E(K_X-B)$ splits into
$\nu_0^*E \rightarrow G_+ \hookrightarrow F_-(K_X-B) \subset \nu_0^*E(K_X-B)$.
Since $c_1(G_+)=-\sigma(D)$ and $c_1(F_-(K_X-B))=D+K_X-B$,
the divisor $D+\sigma(D)+K_X-B$ is represented by a positive one.
If $D+\sigma(D)+K_X-B \neq 0$, then 
$\nu_0^*f +s\tau_{B'}|_C $ should be zero for some prime effective divisor $C$ on $Y_0$.
Remark that $C+\sigma(C)$ descends to a divisor $C_0$ on $X$.
The restriction of $\nu_0^* f+s\tau_{B'}$ to $F_-$ agrees with
$(\times 2s\tau_{B'}): F_- \rightarrow F_-(K_X-B)$, and thus
also the restriction $s\tau_{B'}\in \Gamma(Y_0, {\mathcal O}(K_X-B))$ to $C$
is zero. Hence $\nu_0^* f|_C$ itself is zero.
In case where $C\neq \sigma(C)$, also $\nu_0^* f|_{\sigma(C)}$ is zero,
and thereby $f|_{C_0}$ is zero.
In case where $C=\sigma(C)$, 
$C$ is contained in the ramification locus of $\nu_0$, that is $Z(s)$. 
By its definition $Z(s)$ is reduced, and $Z(s)$ is locally irreducible
since every fiber of $X$ is irreducible.
Since $Z(s)$ is locally integral and $C$ is prime, 
$C$ agrees with a connected component of $Z(s)$ as subschemes.
From this one can check that $\nu_0^*f|_{2C}$ is zero, 
and hence $f|_{C_0}$ is zero.
In each cases the restriction of $f:E\rightarrow E(K_X-B)$ to $C_0$ is zero,
which contradicts to the choice of $B$.
Therefore $\nu_0^*f+s\tau_{B'}: G_+ \subset F_-(K_X-B)$ induces an isomorphism
$\iota:\sO(-\sigma(D)) \simeq \sO(D+K_X-B)$.
Moreover, since $\sigma(\nu_0^* f+s\tau_{B'})=\nu_0^*f-s\tau_{B'}$, one can
verify $\sigma(\iota(a))=\sigma(\iota)(\sigma(a))$ for any local section $a$
of $G_+$. Thus $\iota$ is $\ZZ/2$-equivariant.
\end{proof}
\begin{rem}\label{rem:DegFilt=0}
For the generic point $\eta' \rightarrow C$, the degree of line bundles
$\sO(D)_{\eta'}$ and $\sO(\sigma(D))_{\eta'}$ on $Y_{\eta'}$ are zero or less,
since $E_{\eta'}$ is semistable. Because of Lemma \ref{lem:D+sD},
the degree of them are zero.
Thus, for $\pi_C: Y \rightarrow C$ and  every closed point $q\in C$,
reduced Hilbert polynomials of $\sO(-2D+B)_{\pi_C^{-1}(q)}$ and
$\sO(2D+K_X)_{\pi_C^{-1}(q)}$ equal to that of $\sO_{\pi_C^{-1}(q)}$.
\end{rem}
%
The natural map $\ad(f): Hom_X(E,E) \rightarrow Hom_X(E,E(K_X))$ defined by
$\ad(f)(a)=f\circ a-a\circ f$ induces the following distinguished triangle (d.t.)
in $D^b(X)$
\begin{equation}\label{eq:f-D(X)}
 RHom_X(E,E) \overset{\ad(f)}{\longrightarrow} RHom_X(E,E(K_X))
 \longrightarrow Mc(\ad(f)) \longrightarrow \cdot ,
\end{equation}
where $Mc$ means mapping cone.
By applying the functor $R\Gamma_X(\cdot)$ to this d.t.,
we get the following d.t. in $D^b(\CC)$.
\begin{equation*}\label{eq:f-D(C)}
 R\Hom_X(E,E) \overset{\ad(f)}{\longrightarrow} R\Hom_X(E,E(K_X))
 \longrightarrow R\Gamma_X(Mc(\ad(f))) \longrightarrow \cdot 
\end{equation*}
Let $H^i(\ad(f)): \Ext^i(E,E) \rightarrow \Ext^i(E,E(K_X))$
and $\sH^i(\ad(f)): Ext^i(E,E) \rightarrow Ext^i(E,E(K_X))$ denote
homomorphisms induced by $\ad(f)$.
This d.t. leads to an exact sequence:
\begin{multline*}
 0 \longrightarrow \Hom(E,E) \overset{H^0(\ad(f))}\longrightarrow \Hom(E,E(K_X)) 
 \longrightarrow \HH^0(Mc(\ad(f))) \\
 \longrightarrow \Ext^1(E,E)
 \overset{H^1(\ad(f))}{\longrightarrow} \Ext^1(E,E(K_X)). 
\end{multline*}
Since $E$ is stable, one can check from this exact sequence that
\begin{multline}\label{eq:dimIm-2ext2=}
\rk\left( H^1(\ad(f)): \Ext^1(E,E) \rightarrow \Ext^1(E,E(K_X)) \right)-2\ext^2(E,E)^0 \\
= -\chi(E,E)+2p_g(X)-h^0(Mc(\ad(f))).
\end{multline}
We shall look at $h^0(Mc(\ad(f)))$ further.
Since $RHom_X(E,E)$ belongs to $D^{[0,1]}(X)$, $Mc(\ad(f))$ does to $D^{[-1,1]}(X)$,
and so $Mc(\ad(f))[-1]$ does to $D^{[0,2]}(X)$.
By exact sequences associated with spectral sequences \cite[Thm. 5.12]{CartanE:Homology},
we have an exact sequence
\[ 0 \longrightarrow \HH^1(\sH^{-1}(Mc(\ad(f)))) \longrightarrow
  \HH^0(Mc(\ad(f))) \longrightarrow \HH^0(\sH^0(Mc(\ad(f)))), \quad\text{and hence}\]
%
\begin{equation}\label{eq:hMc}
h^0(Mc(\ad(f))) \leq h^1(\sH^{-1}(Mc(\ad(f))))+h^0(\sH^0(Mc(\ad(f)))). 
\end{equation}
The d.t. \eqref{eq:f-D(X)} induces a long exact sequence in $\Coh(X)$:
\begin{multline}\label{eq:Epq-Mc}
0 \longrightarrow \sH^{-1}(Mc(\ad(f))) \longrightarrow
 Hom(E,E) \overset{H^0(\ad(f))}{\longrightarrow} Hom(E,E(K_X)) \longrightarrow \\
 \sH^0(Mc(\ad(f))) \longrightarrow Ext^1(E,E) \overset{H^1(\ad(f))} {\longrightarrow}
 Ext^1(E,E(K_X)) \longrightarrow \sH^1(Mc(\ad(f))) \longrightarrow 0.
\end{multline}
Now, $\ad(f)$ and
the map $f_+: Hom_X(E,E) \rightarrow Hom_X(E,E(K_X))$ defined by
$f_+(a)=f\circ a+a\circ f$ give exact sequences in $\Coh(X)$:
\begin{align}\label{eq:defn-FGRQ}
 0 \longrightarrow {\mathcal F}=Ker(\ad(f)) \longrightarrow End(E) &
   \longrightarrow {\mathcal G}=Im(\ad(f)) \longrightarrow 0,  \\
 0 \longrightarrow {\mathcal R}=Ker(f_+) \longrightarrow End(E) &
   \longrightarrow {\mathcal Q}=Im(f_+) \longrightarrow 0.  \notag
\end{align}
From \eqref{eq:Epq-Mc} and \eqref{eq:defn-FGRQ},
$\sH^{-1}(Mc(\ad(f)))$ equals to ${\mathcal F}$.
\begin{lem}\label{lem:sF}
$\sF\simeq \Ox \oplus\  \sO(B-K_X)\otimes I_Z$, where $Z$ is a zero-dimensional
subscheme in $X$.
\end{lem}
\begin{proof}
Assume that $E$ is locally free; the proof goes similarly in general case.
We shall use the following commutative diagram in $\Coh(Y_0)$:
\begin{equation}
\xymatrix{
 &  0 \ar[d] & 0 \ar[d] & 0 \ar[d]  & \\
0 \ar@{.>}[r] & Hom(G_-, F_+)  \ar[d] \ar@{.>}[r] & \nu_0^*(\sF) \ar[d] \ar@{.>}[r]& 
Hom(G_+,G_+) \ar[d]  \ar@{.>}[r] & 0\\
0 \ar[r] & Hom(\nu_0^*E, F_+) \ar[d] \ar[r] & End(\nu_0^* E) \ar[d] \ar[r] & 
Hom(\nu_0^*E,G_+) \ar[d] \ar[r] & 0\\
 & Hom'(F_-,F_+) \ar[d] & \nu_0^*(\sG) \ar[d] & Hom'(F_+,G_+) \ar[d] & \\
 & 0  & 0 & 0 & }
\end{equation}
Here, the second row and all columns are exact;
the second row and the first and third columns come from \eqref{eq:defn-Fpm},
and the second column does from \eqref{eq:defn-FGRQ};
the sheaf $Hom'(F_-,F_+)$ in the $(3,1)$-component means 
the image sheaf of natural map $Hom(\nu_0^*E,F_+) \rightarrow Hom(F_-,F_+)$, and
the sheaf $Hom'(F_+,G_+)$ in the $(3,3)$-component is defined similarly.
The map $f:E\rightarrow E(K_X)$
keeps subsheaves $F_{\pm}$ and quotient sheaves $G_{\pm}$ unchanged,
acts on $F_-$ and $G_+$ by multiplication by $s\tau$, and
acts on $F_+$ and $G_-$ by multiplication by $-s\tau$.
Thus one can check that
the sheaf $Hom(G_-, F_+)$ in the $(1,1)$-component is naturally
contained in $\nu_0^*(\sF)$ and that
we can induce a map from the $(1,2)$-component to the $(1,3)$-component,
since $\ad(f)$ acts on $Hom'(F_+,G_+)$ by multiplication by $2s\tau$, which is injective.
By diagram chasing, one can check that the first row is exact,
since the trace map ${\mathcal O}_{Y_0} \rightarrow End(\nu_0^* E)$ gives a splitting
of the right side.
The first Chern class of $(1,1)$-component is $c_1(F_+)-c_1(G_-)=\sigma(D)+D=B-K_X$
from Lemma \ref{lem:D+sD}, and so
$(1,1)$-component equals $\nu_0^* {\mathcal O}(B-K_X)$. 
This isomorphism 
$\nu_0^*(\sF) \simeq \nu_0^*(\Ox \oplus {\mathcal O}(B-K_X))$ is $\sigma$-equivariant,
and so it deduces this lemma by descent theory.
\end{proof}
As to $\sH^0(Mc(\ad(f)))$, \eqref{eq:Epq-Mc} and \eqref{eq:defn-FGRQ} deduce
exact sequences:
\begin{align}\label{eq:sH0(Mc)}
 0 \longrightarrow Cok(H^0(\ad(f))) \longrightarrow &
\sH^0(Mc(\ad(f))) \longrightarrow  Ker(H^1(\ad(f)))  \longrightarrow 0, \\
0 \longrightarrow \sR(K_X)/\sG \longrightarrow &
 Cok(H^0(\ad(f))) \longrightarrow \sQ(K_X)  \longrightarrow 0. \notag
\end{align}
\begin{lem}\label{lem:sQ}
$\sQ\simeq \Ox \oplus\  \sO(K_X-B)\otimes I_T$, where $T$ is a zero-dimensional
subscheme in $X$.
\end{lem}
\begin{proof}
We prove this in case where $E$ is locally free.
Let us consider the following commutative diagram in $\Coh(Y_0)$
such that the second row and all columns are exact:
\begin{equation*}
\xymatrix{
 &  0 \ar[d] & 0 \ar[d] & 0 \ar[d]  & \\
 & Hom(G_+, F_+)  \ar[d] \ar@{.>}[r] & \nu_0^*(\sR) \ar[d] \ar@{.>}[r]& 
Hom(G_-,G_+) \ar[d] & \\
0 \ar[r] & Hom(\nu_0^*E, F_+) \ar[d] \ar[r] & End(\nu_0^* E) \ar[d] \ar[r] & 
Hom(\nu_0^*E,G_+) \ar[d] \ar[r] & 0\\
0 \ar@{.>}[r] & Hom(F_+,F_+) \ar[d] \ar@{.>}[r]^{\iota}& \nu_0^*(\sQ) \ar[d] 
\ar@{.>}[r]^{p} & Hom(F_-,G_+) \ar[d]^j & \\ 
 & Ext^1(G_+,F_+) \ar[d] & 0 & Ext^1(G_-,G_+) \ar[d] & \\
 & 0 &  & 0 & }
\end{equation*}
Similarly to the proof of Lemma \ref{lem:sF}, we can induce
homomorphisms in the first row, and hence homomorphisms in the third row.
One can check that $\sR$ is contained in the kernel of trace map,
so it induces a homomorphism $\operatorname{tr}: \sQ \rightarrow \Ox$.
By this homomorphism and diagram chasing, one can verify that the third row
deduces a splitting exact sequence
\[ 0 \longrightarrow Hom(F_+,F_+) \longrightarrow \nu_0^*(\sQ)
 \longrightarrow Ker(j) \longrightarrow 0.\]
Since the support of $Ext^1(G_-,G_+)$ is zero-dimensional,
$c_1(Ker(j))=c_1(Hom(F_-,G_+))=-D-\sigma(D)=K_X-B$.
%
\end{proof}
Since 
$\chi(E,E)=\chi(End(E))-\chi(Ext^1(E,E))=\chi(\sF)+\chi(\sG)-\chi(Ext^1(E,E))$
and $h^0(\sF)\leq h^0(End(E))=1$, the following estimation holds:
\begin{align}\label{eq:dimIm-2ext2>}
 & \rk(H^1(\ad(f)))-2\ext^2(E,E)^0 \\
\geq & -\chi(E,E)+2p_g-[h^1(\sF)+h^0(\sH^0(Mc(\ad(f)))] \quad (\text{By }
\eqref{eq:dimIm-2ext2=}, \eqref{eq:hMc}) \notag \\
\geq & -\chi(E,E)+2p_g-[h^1(\sF)+h^0(\sR(K_X)/\sG)+h^0(\sQ(K_X))+h^0(Ker(\sH^1(\ad(f))))]
\quad (\text{By }\eqref{eq:sH0(Mc)}) \notag \\
= & l(Ext^1(E,E))+2p_g-[\chi(\sF)+\chi(\sG)+h^0(\sR(K_X)/\sG)+
 h^1(\sF)+h^0(\sQ(K_X))+l(Ker(\sH^1(\ad(f))))]
 \notag \\
=& l(Im(\sH^1(\ad(f))))+2p_g-[h^0(\sF)+h^2(\sF)+h^0(\sR(K_X)/\sG)+h^0(\sQ(K_X))] \notag \\
 & +\chi(\sR(K_X)/\sG)-\chi(\sR(K_X)) \notag \\
\geq &
-\chi(\sR(K_X)) +l(Im(\sH^1(\ad(f))))+2p_g-[h^1(\sR(K_X)/\sG)+1+2p_g \notag \\
 & +h^2({\mathcal O}(B-K_X)\otimes I_Z) +h^0({\mathcal O}(2K_X-B)\otimes I_T)]
\quad(\text{By Lemma \ref{lem:sF}, Lemma \ref{lem:sQ}}) \notag \\
= & -\chi(\sR(K_X))+l(Im(\sH^1(\ad(f)))) \notag \\
  & -[1+h^0({\mathcal O}(2K_X-B)) +h^0({\mathcal O}(2K_X-B)\otimes I_T)
 +h^1(\sR(K_X)/\sG)]. \notag
\end{align}
With \eqref{eq:dimIm-2ext2>} in mind,
let us calculate $\chi(\sR(K_X))$.
Because of the construction of flat covering $\nu_0$,
\[\chi(\nu_0^*\sR(K_X))=\chi(\sR(K_X))+\chi(\sR(K_X)\otimes \sL^{-1})=2\chi(\sR(K_X)),\]
since
$c_1(\sR)$ equals $B-K_X\in \QQ\cdot\cf$ by Lemma \ref{lem:sQ}.
When $E$ is locally free, the commutative diagram in the proof of
Lemma \ref{lem:sQ} and the snake lemma give an exact sequence
\[ 0 \longrightarrow Hom(G_+, F_+)  \longrightarrow \nu_0^*(\sR) \longrightarrow
Hom(G_-,G_+) \longrightarrow Ext^1(G_+,F_+) \longrightarrow 0.\]
Since $Y_0$ is Cartier by Lemma \ref{lem:Y0-Cartier} and
the first Chern class of $Hom(G_+, F_+)$ (resp. $Hom(G_-,G_+)$) is
$2\sigma(D)$ (resp. $-2\sigma(D)+B-K_X$),
we have an exact sequence on $Y_0$
\begin{equation*}
 0 \longrightarrow {\mathcal O}_{Y_0}(2\sigma(D))\otimes I_{\sigma U'} 
  \longrightarrow \nu_0^* \sR
  \longrightarrow {\mathcal O}_{Y_0}(-2\sigma(D)+B-K_X)\otimes I_{\sigma U} 
\longrightarrow 0, 
\end{equation*}
where $U,\ U'$ are zero-dimensional subschemes of $Y_0$.
Applying $\sigma$ to this sequence and twisting it by $K_X$, we get
an exact sequence on $Y_0$
\begin{equation}\label{eq:decompR-I}
 0 \longrightarrow {\mathcal O}_{Y_0}(2D+K_X)\otimes I_{U'} 
  \longrightarrow \nu_0^* \sR(K_X)
  \longrightarrow {\mathcal O}_{Y_0}(-2D+B)\otimes I_U \longrightarrow 0.
\end{equation}
Using Lemma \ref{lem:Y0-Cartier}, one can check that
\[\chi(Y_0,{\mathcal O}_{Y_0}(2D+K_X))=\chi(Y, \phi^*{\mathcal O}_{Y_0}(2D+K_X))
=2D^2+2\chi(\Ox)=2D^2+2d, \]
where the intersection number $D^2$ is calculated in $Y$, for
Lemma \ref{lem:D+sD} implies $D\cdot \nu^*\cf=0$.
Summing up, we have that $\chi(\sR(K_X))=2D^2+2d-(l(U)+l(U'))/2$.
\begin{prop}\label{prop:-D2>4d}
Under Assumption \ref{assump:I}, $-D^2 \geq 4d$.
\end{prop}
\begin{proof}
 If $h^0(\phi^*{\mathcal O}(D))\neq 0$, then $\phi^*(D)$ is rationally
 equivalent to a effective divisor on $Y$ which is supported on
 fibers of $Y\rightarrow C$, for $\phi^*(D)\cdot \cf=0$.
 Thus the restriction of $\phi^*{\mathcal O}_{Y_0}(D)$ to
 $Y_{\eta'}$ is isomorphic to ${\mathcal O}_{Y_{\eta'}}$.
 This contradicts to Assumption \ref{assump:I}.
 Similarly, $h^2(\phi^*{\mathcal O}(D))=0$.
 This implies that $0\geq \chi(\phi^*{\mathcal O}(D))=D^2/2+2d$.
\end{proof}
From \eqref{eq:dimIm-2ext2>} and Proposition \ref{prop:-D2>4d}, 
\begin{multline}\label{eq:dimIm-2ext2-h1-sRKoG-I}
\rk(H^1(\ad(f)))-2\ext^2(E,E)^0 \geq 6d+(l(U)+l(U'))/2 \\
+l(Im(\sH^1(\ad(f)))) 
-[1+h^0({\mathcal O}(2K_X-B)) +h^0({\mathcal O}(2K_X-B)\otimes I_T)
 +h^1(\sR(K_X)/\sG)]. 
\end{multline}
Now we shall consider $h^1(\sR(K_X)/\sG)$.
Similarly to \eqref{eq:decompR-I}, the proof of Lemma \ref{lem:sF} deduces
an exact sequence
\begin{equation}\label{eq:decompG-I}
0 \longrightarrow {\mathcal O}(2D+K_X-B)\otimes I_{W'} \longrightarrow
 \nu_0^* \sG \longrightarrow {\mathcal O}(-2D)\otimes I_W 
\longrightarrow 0, 
\end{equation}
where $W$ and $W'$ are some zero-dimensional subschemes of $Y_0$.
Thus we have the following commutative diagram \eqref{eq:decompR(K)/G-I}
whose rows and columns are exact;
its first row is \eqref{eq:decompG-I}, its second row is \eqref{eq:decompR-I},
and $\tau_B$ is a natural section of $\Gamma(\sO(B))$.
\begin{equation}\label{eq:decompR(K)/G-I}
\xymatrix{
 & 0 \ar[d] & 0 \ar[d] & 0 \ar[d] & \\
0 \ar[r] & {\mathcal O}(2D+K_X-B)\otimes I_{W'} \ar[r] \ar[d]^{\times \tau_B} &
 \nu_0^* \sG \ar[r] \ar[d] & {\mathcal O}(-2D)\otimes I_W \ar[r] 
 \ar[d]_{\times \tau_B} & 0 \\
0 \ar[r] & {\mathcal O}_{Y_0}(2D+K_X)\otimes I_{U'} \ar[r] \ar[d] &
 \nu_0^* \sR(K_X) \ar[r] \ar[d] & {\mathcal O}_{Y_0}(-2D+B)\otimes I_U 
\ar[r] \ar[d] & 0 \\
0 \ar[r] & Cok_l \ar[r] \ar[d] & \nu_0^* \sR(K_X)/\sG \ar[r] \ar[d] &
 Cok_r \ar[r] \ar[d] & 0 \\
 & 0 & 0 & 0 & \\}
\end{equation}
For a sheaf $F$ of dimension $\leq 1$,
let $tor(F)$ denote the maximal subsheaf of $F$
of zero dimension, and $pur(F)$ the quotient $F/tor(F)$.
From the third column of \eqref{eq:decompR(K)/G-I}, one can check that 
$pur(Cok_r)\simeq {\mathcal O}(-2D+B)|_{\niB} \otimes I_{U\cap \niB}$,
where $U\cap \niB=\Spec({\mathcal O}_{Y_0}/I_U+(\tau_B))$, that
$h^1(Cok_r) \leq h^1({\mathcal O}(-2D+B)|_{\niB})+l(U\cap \niB)$
and thereby
\begin{multline}\label{eq:h1-XandY}
h^1(\nu_0^* \sR(K_X)/\sG)\leq h^1({\mathcal O}(-2D+B)|_{\niB}) \\
+h^1(\sO(2D+K_X)|_{\niB}) +l(U\cap \niB)+l(U'\cap \niB).
\end{multline}
Since $h^1(\nu_0^*\sR(K_X)/\sG)=h^1(\sR(K_X)/\sG)+h^1(\sR(K_X)/\sG\otimes \sL^{\vee})$, 
the following should be useful for estimation of $h^1(\sR(K_X)/\sG)$.
\begin{prop}\label{prop:h1-sRKoG}
We denote by $\Lambda(B)$ the number of connected components $B_0$ of $B$
such that $B_0$ is lying over a multiple fiber $F$ and that
$\sL|_F$ is NOT isomorphic to $\sO_F$. Then
\[ h^1(\sR(K_X)/\sG) -h^1(\sR(K_X)/\sG \otimes \sL^{\vee}) \leq 2\Lambda(B).\]
\end{prop}
\begin{proof}
First remark that $\sL|_{B_0}=\sO(K_X-B-B')|_{B_0}$ is isomorphic to $\sO_{B_0}$ 
if $B_0$ is not lying over multiple fibers.
Let $m$ be the multiplicity of the fiber over which $B_0$ is lying, and
$F$ the reduction of $B_0$. $B_0$ is indicated as $B_0=nF$ with $n\in\NN$.
It suffices to show that
\begin{equation}\label{eq:h1-leq2}
h^1(\sR(K_X)/\sG|_{B_0}) -h^1(\sR(K_X)/\sG \otimes \sL^{\vee}|_{B_0}) \leq 2.
\end{equation}
Hereafter, we abbreviate $\sR(K_X)/\sG|_{B_0}$ to $\sR(K_X)/\sG$.
On positive divisor $B_0$ that is not necessarily reduced,
recall that the stability of ${\mathcal T}\in\Coh(B_0)$ is defined 
(e.g. \cite[Section 1.2]{HL:text}), and that 
\begin{equation}\label{eq:SerreDlty}
h^1({\mathcal T})=\hom({\mathcal T}, \omega_{B_0}),\ \text{where } \omega_{B_0}
 =\sO(K_X+B_0)|_{B_0}
\end{equation}
(e.g. \cite[Theorem III.7.6, 7.11]{Ha:text}).
Let $\{ \HN_k(pur(\sR(K_X)/\sG))\}$ denote
the Harder-Narashimhan filtration of $pur(\sR(K_X)/\sG)$
\[ \HN_0=pur(\sR(K_X)/\sG) \supset \HN_1 \supset \cdots \supset \HN_k \supset
   \HN_{k+1} \cdots, \]
$\gr_k^{HN}(\sR(K_X)/\sG)$ its $k$-th factor $\HN_k/\HN_{k+1}$, 
and $k_0$ the integer such that
the reduced Hilbert polynomial $p(\gr_k^{HN}(\sR(K_X)/\sG))$ is asymptotically
greater than $p({\sO}_F)$ if and only if $k< k_0$.
We abbreviate $\gr_{k_0}^{HN}(\sR(K_X)/\sG)$ to $\gr_0^{HN}$.
Since $\chi(\sR(K_X)/\sG)=\chi(\sR(K_X)/\sG\otimes\sL^{\vee})$, we can verify that
\begin{equation}\label{eq:ReduceToGrHN}
 h^1(\sR(K_X)/\sG) -h^1(\sR(K_X)/\sG \otimes \sL^{\vee}) 
= h^0(\gr_0^{HN})-h^0(\gr_0^{HN}\otimes \sL^{\vee})
\end{equation}
by \eqref{eq:SerreDlty} and standard arguments about cohomologies of semistable sheaves.
We may assume that $p(\gr_0^{HN})=p(\sO_F)$;
if not, the right side of \eqref{eq:ReduceToGrHN} should be zero.
Denote by $\{ \JH_l(\gr_0^{HN} ) \}$ a Jordan-H\"{o}lder filtration
of $\gr^{HN}_0$, and by $l_0$ such an integer that both of natural maps
\begin{align}\label{eq:define-l0}
 & H^0(\gr_0^{HN}) \longrightarrow H^0(\gr_0^{HN}/ \JH_{l}(\gr_0^{HN})) 
   \quad\text{and}\\
 & H^0(\gr_0^{HN}\otimes \sL^{\vee}) \longrightarrow 
H^0(\gr_0^{HN}/ \JH_{l}(\gr_0^{HN})\otimes \sL^{\vee}) \notag
\end{align}
are zero when $l=l_0$, but either of them is not zero when $l=l_0+1$. Especially,
\begin{equation}\label{eq:ReduceJHlz}
h^0(\gr_0^{HN})=h^0(\JH_{l_0}(\gr_0^{HN})) \quad\text{and}\quad
h^0(\gr_0^{HN}\otimes\sL^{\vee})=h^0(\JH_{l_0}(\gr_0^{HN})\otimes\sL^{\vee}).
\end{equation}
We indicate $\JH_{l_0}(\gr_0^{HN})$ by $\JHlz$ and 
$\gr^{JH}_{l_0}(\gr_0^{HN})$ by $\grlz$.
If the first map of \eqref{eq:define-l0} is not zero,
then there exists a map $g: \sO_{B_0} \rightarrow \JHlz$ such that
the induced map
$\sO_{B_0} \rightarrow \JHlz \rightarrow \grlz$ is not zero.
The latter map is surjective, since $\sO_{B_0}$ is semistable, $\grlz$ is
stable, and their reduced Hilbert polynomials are same.
\begin{clm}\label{clm:Cok(g)_F}
Denote by $R$ the local ring $\sO_{X,\eta(F)}$ with the maximal ideal
$(x)$, where $\eta(F)$ is the generic point of $F$.
Then $Cok(g)_{\eta(F)} \simeq R/(x^{i'})$ for some integer $i'$.
\end{clm}
\begin{proof}
Let $J_i$ be the localization of the pull-back of 
$\JH_{l_0+i}(\gr_{k_0}^{HN}(\sR(K_X)/\sG))$ by a natural surjection
$\sR(K_X) \rightarrow pur(\sRKoG) \rightarrow pur(\sRKoG)/\HN_{k_0+1}$
at $\eta(F)$.
Then $\gr_{l_0 \eta(F)}^{JH}$ is isomorphic to $J_0/J_1$.
By theory of elementary divisors on PID, we can take such an isomorphism
$J_0 \simeq R\oplus R$ that its submodule $J_1$ is isomorphic
to $ (x^i)\oplus (x^j)$ ($0\leq i\leq j\leq n$).
One can shift $g_{\eta(F)}: \sO_{B_0, \eta(F)}=R/(x^n) \rightarrow \JH_{l_0 \eta(F)}$
to $g': R \rightarrow J_0$. Consider the commutative diagram
\begin{equation*}
\xymatrix{
 R \ar[r]_{g'} & J_0\simeq R\oplus R \ar[d] \ar[r]_{p_2} & R \ar[d]^{\pi} \\
    &  J_0/J_1 \simeq R/(x^i) \oplus R/(x^j) \ar[r] & R/(x^j).}
\end{equation*}
Since the map $\sO_{B_0} \rightarrow \JHlz \rightarrow \grlz$ induced by $g$
is surjective, also the map $\pi\circ p_2\circ g': R \rightarrow R/(x^j)\neq 0$
is surjective, and so $g'\circ p_2$ is surjective.
Thus one can check that $Cok(g')$ is a quotient of $R$, so get this claim.
\end{proof}
If both of the following maps 
\begin{equation}\label{eq:H0Cokg}
H^0(\JHlz)\rightarrow H^0(Cok(g)) \quad \text{and} \quad
H^0(\JHlz\otimes \sLi)\rightarrow H^0(Cok(g)\otimes \sLi) 
\end{equation}
are zero, then $h^0(\JHlz)=h^0(Im(g))$ and 
$h^0(\JHlz\otimes\sLi)=h^0(Im(g)\otimes\sLi)$, and hence
$h^1(\sRKoG)-h^1(\sRKoG\otimes\sLi)=h^0(Im(g))-h^0(Im(g)\otimes\sLi)$
by \eqref{eq:ReduceToGrHN} and \eqref{eq:ReduceJHlz}.
In this case one can check that 
$h^1(\sRKoG)-h^1(\sRKoG\otimes\sLi)\leq 1$ in a similar fashion to
another cases below.\par
Now suppose that the left map in \eqref{eq:H0Cokg} is not zero.
$Cok(g)$ is a semistable sheaf whose reduced Hilbert polynomial is $p(\sO_F)$.
Let $\{ \JH_m(Cok(g) ) \}$ be its Jordan-H\"{o}lder filtration, 
$p:\JH_{l_0} \rightarrow Cok(g)$ a natural quotient map, and 
$m_0$ such integer that both of natural maps
\begin{align}\label{eq:m=m0}
 & H^0(\JHlz) \longrightarrow H^0(Cok(g)) \longrightarrow
 H^0(Cok(g)/ \JH_{m}(Cok(g)))
   \quad\text{and}\\
  & H^0(\JHlz\otimes\sLi) \longrightarrow H^0(Cok(g)\otimes\sLi) \longrightarrow
 H^0(Cok(g)/ \JH_{m}(Cok(g)) \otimes\sLi) \notag
\end{align}
are zero when $m=m_0$, but either of them is not zero when $m=m_0+1$.
We indicate $\JH_{m_0}(Cok(g))$ by $\JHCk$ and 
$\gr^{JH}_{m_0}(Cok(g))$ by $\grCk$. By \eqref{eq:ReduceJHlz},
\begin{equation}\label{eq:ReduceTo-pi-JHCk}
h^0(\gr_0^{HN}) =h^0(p^{-1}(\JHCk)), \quad 
h^0(\gr_0^{HN}\otimes\sLi)=h^0(p^{-1}(\JHCk\otimes\sLi)).
\end{equation}
Assume that the first map at \eqref{eq:m=m0} is not zero.
Then there is a nonzero map 
$h: \sO_{B_0} \rightarrow p^{-1}(\JHCk)\subset \JHlz$ such that
the induced map
$\sO_{B_0} \rightarrow p^{-1}(\JHCk) \rightarrow \JHCk \rightarrow \grCk$
is not zero. 
It is surjective since $\sO_{B_0}$ is semistable, $\grCk$ is stable and their
reduced Hilbert polynomials are same.
\begin{clm}\label{clm:ph-surj}
$p\circ h:\sO_{B_0 } \rightarrow p^{-1}(\JHCk) \rightarrow \JHCk$ is surjective.
\end{clm}
\begin{proof}
Localize this map at $F$.
$\JHCk_{\eta(F)}$ is a submodule of $\Cok(g)_{\eta(F)}\simeq R/(x^{i'})$ 
by Claim \ref{clm:Cok(g)_F}, and so it is isomorphic to $R/(x^{i_1})$.
The map 
$(p\circ h)_{\eta(F)}:\sO_{B_0 \eta(F)} \rightarrow \JHCk_{\eta(F)}\simeq R/(x^{i_1})$ 
is surjective since the induced map
$\sO_{B_0,\eta(F)} \rightarrow \JHCk_{\eta(F)} \rightarrow \grCk_{\eta(F)}\neq 0$ 
is surjective. Thereby $Cok(p\circ h)$ is zero-dimensional.
If $Cok(p\circ h)$ is not zero, then one can check that $p(\sO_{B_0})<p(\JHCk)$ 
from their semistability, which contradicts to the fact $p(\sO_{B_0})=p(\JHCk)$.
\end{proof}
We obtained two homomorphisms $g: \sO_{B_0} \rightarrow \JHlz$
and $h: \sO_{B_0} \rightarrow p^{-1}(\JHCk)\subset \JHlz$.
This $g$ induces $g: \sO_{B_0} \rightarrow p^{-1}(\JHCk)$, $Im(g)$ is semistable with
reduced Hilbert polynomial $p(\sO_F)$, and so $Im(g) \simeq \sO_{l_1F}$ 
with $0\leq l_1\leq n$. Similarly, $Im(h)\simeq \sO_{l_2F}$ with $0\leq l_2 \leq n$.
On $(g,h):\sO_{l_1F}\oplus \sO_{l_2F} \rightarrow p^{-1}(\JHCk)$,
$Ker(g,h)$ is either zero or semistable with reduced Hilbert polynomial $p(\sO_F)$, and
$Ker(g,h) \subset \sO_{l_1F}\oplus \sO_{l_2F} \overset{\pi_1}{\rightarrow} \sO_{l_1F}$
is injective.
Thereby $Ker(g,h)\simeq \sO_{l_3F}((l_3-l_1)F)$ with $0\leq l_3\leq l_1$.
%
%
Similarly, $Ker(g,h)$ is contained in $\sO_{l_2F}$, and 
it is isomorphic to $\sO_{l_3F}((l_3-l_2)F)$.
These imply that $l_1\equiv l_2 \pmod{m}$ by the fact below.
\begin{fact}\cite[p.169]{Frd:holvb}\label{fact:O(F)|F}
The divisor on $F$ corresponding to $\sO(F)|_F$ 
is a torsion element of order $m$.
\end{fact}
From Claim \ref{clm:ph-surj}, they give an exact sequence
\begin{equation*}
 0 \longrightarrow \sO_{l_3F}((l_3-l_2)F) \longrightarrow \sO_{l_1F}\oplus \sO_{l_2F}
 \overset{(g,h)}{\longrightarrow} p^{-1}(\JHCk) \longrightarrow 0,
\end{equation*}
which deduces an exact sequence
\begin{equation}\label{eq:DecompJHCk}
 0 \longrightarrow \sO_{l_1F} \longrightarrow p^{-1}(\JHCk)
 \overset{q}{\longrightarrow} \sO_{(l_2-l_3)F} \longrightarrow 0
\end{equation}
such that 
$q\circ h: \sO_{l_2F} \rightarrow p^{-1}(\JHCk) \rightarrow \sO_{(l_2-l_3)F}$
is a natural surjection. 
\begin{clm}\label{clm:Calc-h0-JHCk}
By assumption, $\sL|_{B_0}$ is not isomorphic to $\sO|_{B_0}$, and so
$\sL|_{B_0} =\sO_{B_0}(-sF)$ with $0< s <m$. Then 
\begin{align*}
h^0(p^{-1}(\JHCk))= &  \lceil l_1/m \rceil + \lceil (l_2-l_3)/m \rceil, \\
h^0(p^{-1}(\JHCk)\otimes \sLi)= & \lceil (l_1-s)/m \rceil + \lceil (l_2-l_3-s)/m \rceil.
\end{align*}
\end{clm}
\begin{proof}
Recall Fact \ref{fact:O(F)|F}.
Since natural map $H^0(\sO_{l_2F})\rightarrow H^0(\sO_{(l_2-l_3)F})$ is
surjective, 
$h^0(p^{-1}(\JHCk))= h^0(\sO_{l_1F}) + h^0(\sO_{(l_2-l_3)F}) =
\lceil l_1/m \rceil + \lceil (l_2-l_3)/m \rceil$.
As to the second equation, \eqref{eq:DecompJHCk} deduces 
\begin{multline}\label{eq:h0-JHCk}
 h^0(\sO_{l_1F}(sF)) + \rk [ H^0(\sO_{l_2F}(sF))\rightarrow 
 H^0(\sO_{(l_2-l_3)F} (sF))] \leq \\
 h^0(p^{-1}(\JHCk)\otimes\sLi) \leq h^0(\sO_{l_1F} (sF)) 
+ h^0(\sO_{(l_2-l_3)F}(sF)).
\end{multline}
Here let us verify that
\begin{equation}\label{eq:h0-OsF}
h^0(\sO_{sF}(sF))=0 \quad\text{and}\quad
h^0(\sO_{lF}(sF))= \lceil (l-s)/m \rceil. 
\end{equation}
From Fact \ref{fact:O(F)|F}, one can check the
left side and that $h^0(\sO_{lF}(sF))=0$ if $l\leq s$.
When $l>s$, the exact sequence
\[ 0 \longrightarrow \sO_{(l-s)F} \longrightarrow \sO_{lF}(sF) \longrightarrow
  \sO_{sF}(sF) \longrightarrow 0 \]
deduces that $h^0(\sO_{lF}(sF))=h^0(\sO_{(l-s)F})=\lceil (l-s)/m \rceil$,
and so we get \eqref{eq:h0-OsF}.\par
%
%
When $l_2-l_3 \geq s$, the commutative diagram
\begin{equation*}
\xymatrix{
0 \ar[r] & \sO_{(l_2-s)F} \ar@{>>}[d] \ar[r] & \sO_{l_2F}(sF) \ar[d] \ar[r] &
\sO_{sF}(sF) \ar@{=}[d] \ar[r] & 0 \\
0 \ar[r] & \sO_{(l_2-l_3-s)F} \ar[r] & \sO_{(l_2-l_3)F}(sF) \ar[r] &
\sO_{sF}(sF) \ar[r] & 0}
\end{equation*}
and \eqref{eq:h0-OsF} induce the commutative diagram
\begin{equation*}
\xymatrix{
H^0(\sO_{(l_2-s)F}) \ar@{>>}[d] \ar[r]^{\sim} & H^0(\sO_{l_2F}(sF)) \ar[d] \\
H^0 (\sO_{(l_2-l_3-s)F}) \ar[r]^{\sim} & H^0(\sO_{(l_2-l_3)F}(sF)).}
\end{equation*}
When $l_2-l_3 < s$, $h^0(\sO_{(l_2-l_3)F}(sF))=0$ by \eqref{eq:h0-OsF}.
In both cases, the second equation holds.
\end{proof}
From \eqref{eq:ReduceToGrHN}, \eqref{eq:ReduceTo-pi-JHCk} and
Claim \ref{clm:Calc-h0-JHCk}, 
$ h^1(\sR(K_X)/\sG) -h^1(\sR(K_X)/\sG \otimes \sL^{\vee}) = 
 \lceil l_1/m \rceil  -  \lceil (l_1-s)/m \rceil 
+ \lceil (l_2-l_3)/m \rceil -\lceil (l_2-l_3-s)/m \rceil $.
It is easy to check that this equals $2$ or less.
It is left to the reader to certify this proposition in remaining cases;
the case where the second map at \eqref{eq:define-l0} is not zero, and
the case where the second map at \eqref{eq:m=m0} is not zero.
Therefore the proof of Proposition \ref{prop:h1-sRKoG} is completed.
\end{proof}
From \eqref{eq:dimIm-2ext2-h1-sRKoG-I}, \eqref{eq:h1-XandY} and 
Proposition \ref{prop:h1-sRKoG}, we get the following result.
\begin{prop}\label{prop:dimIm-2ext2-I}
 Under Assumption \ref{assump:I} it holds that
 \begin{multline*}
 \rk(H^1(\ad(f)))-2\ext^2(E,E)^0 \geq 6d +l(Im(\sH^1(\ad(f)))) \\
 +\tfrac{1}{2} \left\{l(U)-l(U\cap \niB)+l(U')-l(U'\cap \niB)\right\}
 -\left[1+\Lambda(B)+h^0({\mathcal O}(2K_X-B)) \right. \\
 \left. +h^0({\mathcal O}(2K_X-B)\otimes I_T)
 +  \tfrac{1}{2} h^1({\mathcal O}(-2D+B)|_{\niB})
 +  \tfrac{1}{2} h^1(\sO(2D+K_X)|_{\niB}) \right]. 
 \end{multline*}
\end{prop}
\section{Singularities of $M(c_2)$ in Case I}\label{sctn:SigM-CaseA1}
%
%
%
%
\begin{thm}\label{thm:rk-geq-3ext-ver1}
Let $E\in M(c_2)$ correspond to Case I at Fact \ref{fact:RestrGenericFib}.
Suppose that (i) $d\geq (7/4)\Lambda(X)-2$ or that (ii) $2\geq \Lambda(X)$.
Then 
$\rk H^1(\ad(f))\geq 2\ext^2(E,E)^0+1$ for all nonzero $f\in \Hom(E,E(K_X))^{\circ}$.
As a result, $M(c_2)$ is of locally complete intersection  and  normal at $E$, and $E$ is at worst a canonical 
singularity of $M(c_2)$.
\end{thm}
\begin{rem}
In Theorem \ref{thm:rk-geq-3ext-ver1}, the assumption ``$E$ corresponds to Case I``
is relatively weak by Fact \ref{fact:GenericVB-CaseI}.
Conditions (i) and (ii) mean that the number $\Lambda(X)$ of multiple fibers
is relatively few.
\end{rem}
%
\begin{proof}
In $\operatorname{Div}(X)$, we can denote $B$ as
\begin{equation}\label{eq:describe-B}
 B= \pi^{-1}(\textstyle\sum_j s_jp_j) + \textstyle\sum_i (t_im_i+l_i)F_i ,
\end{equation}
where $s_j$ and $t_i$ are nonnegative integers, $p_j$ is a closed point of $\PP^1$
lying over a reduced fiber, $F_i$ is the reduction of a multiple fiber with multiplicity
$m_i$, and $l_i$ is an integer such that $0\leq l_i \leq m_i-1$.
By the canonical bundle formula \eqref{eq:CanBdleFormula},
\begin{multline}\label{eq:2(K-B)}
 2(K_X-B)\simeq \{2(d-2-\textstyle\sum_j s_j-\sum_it_i)+\Lambda(X) \}\cf+ 
 \sum_i (m_i-2-2l_i)F_i \\
 = \{ 2(d-2-n) +\Lambda(X)-\Lambda_1(B) \}\cf 
 + \textstyle\sum_i a_{1,i} F_i,
\end{multline}
where $n= \sum_j s_j +\sum_i t_i$,
$\Lambda_1(B)$ is the number of multiple fiber $F_i$ such that
$m_i-1\leq 2l_i$, and
$a_{1,i}$ is an integer with $0\leq a_{1,i} < m_i$.
Since $\det(f)\neq 0$ implies that $h^0(\sO(2(K_X-B)))\neq 0$,
\begin{equation}\label{eq:Since2(K-B)Positive}
2(d-2-n) +\Lambda(X)-\Lambda_1(B) \geq 0.
\end{equation}
We shall estimate the right side of Proposition \ref{prop:dimIm-2ext2-I}.
As to $h^0(2K_X-B)$,
\begin{multline*}
2K_X-B \simeq \{ 2d-4+\Lambda(X) -n\}\cf+ \textstyle\sum_i(m_i-2-l_i)F_i \\
= \{2d-4+\Lambda(X)-n-\Lambda_2(B) \}\cf+\textstyle\sum_i a_{2,i} F_i,
\end{multline*}
where $\Lambda_2(B)$ is the number of $F_i$ such that $l_i=m_i-1$, 
$a_{2,i}$ is an integer with $0 \leq a_{2,i} < m_i$ and hence
\begin{equation}\label{eq:h0-2K-B}
h^0(\sO(2K_X-B))=2d-3+\Lambda(X)-n-\Lambda_2(B).
\end{equation}
Next, let us estimate 
$h^1({\mathcal O}(-2D+B)|_{\niB_0}) + h^1(\sO(2D+K_X)|_{\niB_0}) =
h^0({\mathcal O}(-2D+B)|_{\niB_0}) + h^0(\sO(2D+K_X)|_{\niB_0})$ 
for any connected component $B_0$ of $B$.
\begin{case}\label{case:reduced-etale}
Assume that $B_0=s_j\cf_0$ where $\cf_0=\pi^{-1}(p_j)$ is reduced
and $\nu_0: Y_0 \rightarrow X$ is \'{e}tale at $\cf_0$.
Since $\sL|_{\cf_0}=\sO_{\cf_0}$, $h^0(\nu_0^{-1}\sO_{\cf_0})=2$.
Thus $\nu_0^{-1}(\cf_0)=\cf'_0 \sqcup \sigma(\cf'_0)$, where
$\cf'_0$ is a reduced curve in $Y_0$ such that $\nu_0: \cf'_0 \rightarrow \cf_0$
is isomorphism.
By Remark \ref{rem:DegFilt=0}, 
$\deg\sO(-2D+B)|_{\cf'_0}=\deg\sO(2D+K_X)|_{\cf'_0}=0$. As a result,
\begin{multline}\label{eq:leq-4sj}
 h^0({\mathcal O}(-2D+B)|_{\niB_0}) + h^0(\sO(2D+K_X)|_{\niB_0})=
 h^0({\mathcal O}(-2D+B)|_{s_j\cf'_0}) \\
+  h^0({\mathcal O}(-2D+B)|_{s_j\sigma(\cf'_0)})+
 h^0({\mathcal O}(2D+K_X)|_{s_j\cf'_0})+ 
 h^0({\mathcal O}(2D+K_X)|_{s_j\sigma(\cf'_0)}) \leq 4s_j.
\end{multline}
\end{case}
\begin{case}\label{case:reduced-ramify}
Assume that $B_0=s_j\cf_0$ where $\cf_0=\pi^{-1}(p_j)$ is reduced
and $\nu_0: Y_0 \rightarrow X$ ramifies at $\cf_0$.
Then $\nu_0^{-1}(\cf_0)=2\cf'_0$, where
$\cf'_0$ is a reduced curve in $Y_0$ such that $\nu_0: \cf'_0 \rightarrow \cf_0$
is isomorphism. It holds that
$\chi(\sO(-2D+B)|_{2\cf'_0})=0$ by Remark \ref{rem:DegFilt=0}, 
and then $\chi(\sO(-2D+B)|_{\cf'_0})=0$ since $\deg(\sO(\cf'_0)|_{\cf'_0})=0$.
Hence \eqref{eq:leq-4sj} holds also in Case \ref{case:reduced-ramify}.
\end{case}
Now let us consider cases where
\begin{multline}\label{eq:B0GivesMultFib}
\text{$B_0=(tm+l)F$, where $F$ is the reduction of a multiple fiber in $X$}\\
\text{with multiplicity $m$ and $0\leq l< m$.}
\end{multline}
By \eqref{eq:2(K-B)}, $\det(f)\in\Gamma(\sO(2(K_X-B)))$ satisfies 
in $\operatorname{Div}(X)$ that
\begin{equation}\label{eq:describe-Ddetf}
  \operatorname{Div}(\det(f))= 
  \textstyle\sum_i (t_i'm_i-2l_i-2)F_i + (\text{reduced fibers}).
\qquad (t_i'\in \ZZ_{\geq 1})
\end{equation}
Then $ B' = \sum_i (\lfloor t_i'm_i/2 \rfloor -l_i-1 )F_i + (\text{reduced fibers})$
by its definition in Section \ref{sctn:rkH1-caseA1}, 
and then 
\begin{equation}\label{eq:describe-sL}
c_1(\sL)=K_X-B-B'\simeq \textstyle\sum_i -\lfloor t_i'm_i/2 \rfloor F_i +
(\text{reduced fibers}).
\end{equation}
\begin{case}\label{case:etale-sLF-nontrivial}
Assume that $B_0$ is as in \eqref{eq:B0GivesMultFib}, $\sL|_F \not\simeq \sO_F$
and the map $Y_0 \rightarrow X$ is \'{e}tale at $F$.
Then one can check that $m$ is even and $\sL|_F \simeq \sO((m/2)F)|_F$
from \eqref{eq:describe-Ddetf} and \eqref{eq:describe-sL}.
Since $h^0(\sO_{\nu_0^{-1}(F)})=h^0(\sO|_F)+h^0(\sL^{\vee}|_F)=1$,
$\nu_0^{-1}(F)=F'$ is connected and $\nu_0: F' \rightarrow F$ is \'{e}tale,
and accordingly $F'$ is a nonsingular integral curve.
Remark that $\deg(\sO(F')|_{F'})=0$ and $\deg(\sO(2D)|_{F'})=0$
by Remark \ref{rem:DegFilt=0}, and that 
$h^0(\sO(tF')|_{F'})=h^0(\sO(tF)_F)+h^0(\sO(tF)\otimes\sL^{-1}|_F)
= h^0(\sO(tF)|_F)+ h^0(\sO((t+(m/2))F)|_F)$ is nonzero iff $(m/2)|t$.
Consequently $\sO(F')|_{F'}$ gives a torsion divisor of order $m/2$ and
\begin{equation}\label{eq:h0-2D+B-1st}
 h^0(\sO(-2D+B)|_{\nu_0^{-1}(B_0)=(tm+l)F'})\leq 2t+h^0(\sO(-2D+B)|_{lF'}).
\end{equation}
\begin{clm}
 In Case \ref{case:etale-sLF-nontrivial},
$h^0(\sO(-2D+B)|_{lF'})=h^0(\sO(2D+K_X)|_{lF'})=0$ when $l\leq (m/2)-1$, that is,
$\lfloor 2l/m \rfloor =0$.
\end{clm}
\begin{proof}
 Suppose that $h^0(\sO(-2D+B)|_{lF'})\neq 0$.
 Then there should be an integer $\lambda$ such that 
 $0\leq \lambda \leq l-1\leq (m/2)-2$ and that there is a isomorphism
 $\iota: \sO_{F'} \simeq \sO(-2D+B-\lambda F')|_{F'} $.
 This gives a $\ZZ/2$-isomorphism
 $\iota\cdot\sigma(\iota): \sO_{F'} \simeq \sO(-2D-2\sigma(D)+2B-2\lambda F')|_{F'} $.
 Since Lemma \ref{lem:D+sD} states that $D+\sigma(D)+K_X-B$ is $\ZZ/2$-isomorphic
 to $0$,
 we have an $\ZZ/2$-isomorphism
 $\sO_{F'}\simeq \sO(2K_X-2\lambda F')|_{F'}=
 \sO(-2(1+\lambda)F')|_{F'}=\nu_0^* (\sO(-2(1+\lambda)F)|_{F})$.
 From Luna's \'{e}tale slice theorem $\nu_0: F' \rightarrow F$ is a principal $\ZZ/2$-bundle,
 and hence we obtain an isomorphism 
 $\sO_F \simeq \sO(-2(1+\lambda)F)|_{F}$ on $F$ by \'{e}tale descent theory.
 Accordingly $m|2(1+\lambda)$, so $(m/2)|1+\lambda$, but this is impossible
 since $1\leq \lambda+1 \leq (m/2)-1$.
 In the same way, if $h^0(\sO(2D+K_X)|_{lF'})\neq 0$, then one can deduce that
 $(m/2)| l-\lambda$. This is impossible since $1 \leq l-\lambda \leq l \leq (m/2)-1$.
\end{proof}
Equation \eqref{eq:h0-2D+B-1st} deduces in Case \ref{case:etale-sLF-nontrivial} that 
\begin{align}\label{eq:upperbd-case3}
 h^0(\sO(-2D+B)|_{\nu_0^{-1}(B_0)}) \leq & 2t+ \lfloor 2l/m \rfloor =
\lfloor  2(tm+l)/m \rfloor  \quad \text{and} \\
 h^0(\sO(2D+K_X)|_{\nu_0^{-1}(B_0)}) \leq  & \lfloor 2(tm+l)/m \rfloor. \notag
\end{align}
\end{case}
\begin{case}\label{case:ramify-sLF-nontrivial}
Assume that $B_0$ is as in \eqref{eq:B0GivesMultFib}, $\sL|_F \not\simeq \sO_F$
and $Y_0 \rightarrow X$ ramifies at $F$.
Then one can check that $m$ is odd and 
$\sL|_F \simeq \sO( -\tfrac{m-1}{2}F)|_F \not\simeq \sO_F$
from \eqref{eq:describe-Ddetf} and \eqref{eq:describe-sL}.
We also have that
$\nu_0^{-1}(F)=2F'$ with reduced divisor $F'$ such that $\sigma(F')=F'$,
$\nu_0: F' \rightarrow F$ is isomorphic, and hence
$\sO(F')|_{F'}$ gives a torsion divisor on $F'$ with order $m$.
\begin{clm}\label{clm:ramify-h0-vanish}
Define the number $\Lambda_3(B_0)$ by $\Lambda_3(B_0)=1$ if $l=(m-1)/2$
and $\Lambda_3(B_0)=0$ otherwise. Then
$h^0(\sO(-2D+B)|_{\nu_0^{-1}(lF)=2lF'}) \leq \lfloor 2l/m \rfloor +\Lambda_3(B_0)$,
and
$h^0(\sO(2D+K_X)|_{2lF'}) \leq \lfloor 2l/m \rfloor$. 
\end{clm}
\begin{proof}
Assume that $h^0(\sO(-2D+B)|_{2lF'})\neq 0$ and $l\leq (m-3)/2$.
Then there should be an integer $\lambda$ such that 
$0\leq \lambda \leq 2l-1$ and that
$\sO(-2D+B-\lambda F')|_{F'} \simeq \sO_{F'}$.
By applying $\sigma$ to this, we also get that
$\sO(-2\sigma(D)+B-\lambda F')|_{F'} \simeq \sO_{F'}$,
and by unifying them and using Lemma \ref{lem:D+sD},
\[ \sO_{F'}\simeq \sO_{F'}(2(\lambda F'-B+D+\sigma(D)))=\sO_{F'}(2(\lambda F'-K_X ))
=\sO_{F'}(2(\lambda+2)F').\]
Thus $m|2(2+\lambda)$, and $m|(2+\lambda)$ since $m$ is odd,
but this is impossible because $2 \leq \lambda+2 \leq 2l+1 \leq m-2$.
Hence one can get the first inequality.
Next, assume that $h^0(\sO(2D+K_X)|_{2lF'})\neq 0$ and $l\leq (m-1)/2$.
Then there should be an integer $\mu$ such that 
$0\leq \mu \leq 2l-1$ and that
$\sO(2D+K_X-\mu F')|_{F'} \simeq \sO_{F'}$.
By a similar way to arguments above, 
\[ \sO_{F'}\simeq \sO_{F'}(2(D+\sigma(D)+K_X-\mu F'))=\sO_{F'}(2(B-\mu F'))
 = \sO_{F'}(2(2l-\mu)F').\]
Thus $m|2(2l-\mu)$, and $m|(2l-\mu)$ since $m$ is odd,
but this is impossible because $1\leq 2l-\mu \leq 2l \leq m-1$.
\end{proof}
\end{case}
\begin{case}\label{case:sLF-trivial}
Assume that $B_0$ is as in \eqref{eq:B0GivesMultFib} and
$\sL|_F \simeq \sO_F$.
Then one can check that $t'$ at \eqref{eq:describe-Ddetf} is even,
$Y_0 \rightarrow X$ is \'{e}tale at $F$,
$\nu_0^{-1}(F)=F' \sqcup \sigma(F')$ since $h^0(\nu_0^{-1}(\sO_F))=2$,
$\nu_0: F' \rightarrow F$ is isomorphic, 
$\deg(\sO(2D)|_{F'})=0$ by Remark \ref{rem:DegFilt=0},
and that the order of $\sO(F')|_{F'}$ is $m$. Thus 
$h^0(\sO(-2D+B)|_{\nu_0^{-1}(B_0)})= h^0(\sO(-2D+B)|_{(tm+l)F'})
 + h^0(\sO(-2D+B)|_{(tm+l)\sigma(F')}) \leq 2t+ h^0(\sO(-2D+B)|_{lF'})+
h^0(\sO(-2\sigma(D)+B)|_{lF'})$.
Suppose both $h^0(\sO(-2D+B)|_{lF'})$ and $h^0(\sO(-2\sigma(D)+B)|_{lF'})$
are nonzero.
Then $\sO(-2D+B-\lambda F')|_{F'} \simeq \sO_{F'}$ and
$\sO(-2\sigma(D)+B-\mu F')|_{F'} \simeq \sO_{F'}$ with some integers
$0\leq \lambda,\ \mu \leq l-1$.
Since $D+\sigma(D)+K_X-B=0$, these deduce that
$\sO((\lambda+\mu+2)F')|_{F'} \simeq \sO_{F'}$, but
this is impossible when $2l\leq m-1$, because $2\leq \lambda+\mu+2 \leq 2l$.
Remark that $2l\leq m-1$ iff $\lfloor 2l/m \rfloor =0$, and so
$h^0(\sO(-2D+B)|_{\nu_0^{-1}(B_0)}) \leq \lfloor 2l/m \rfloor +1$.
Therefore in Case \ref{case:sLF-trivial} we have
\begin{align}\label{eq:upperbd-case5}
h^0(\sO(-2D+B)|_{\nu_0^{-1}(B_0)}) \leq & \lfloor 2(tm+l)/m \rfloor +1,\quad \text{and}\\
h^0(\sO(2D+K_X)|_{\nu_0^{-1}(B_0)}) \leq & \lfloor 2(tm+l)/m \rfloor +1. \notag
\end{align}
\end{case}
Summing up \eqref{eq:leq-4sj}, \eqref{eq:upperbd-case3}, Claim \ref{clm:ramify-h0-vanish}
and \eqref{eq:upperbd-case5}, we obtain that
\begin{multline}\label{eq:rk-3ext2}
\tfrac{1}{2} h^0(\sO(-2D+B)|_{\nu_0^{-1}(B_0)})+
\tfrac{1}{2} h^0(\sO(2D+K_X)|_{\nu_0^{-1}(B_0)}) \\
\leq 2\textstyle\sum_j s_j + \textstyle\sum_i \lfloor (2t_i m_i+2l_i)/m_i \rfloor
 +\tfrac{1}{2}\Lambda_3(B)+\Lambda_4(B) 
= 2n+\tfrac{1}{2}\Lambda_3(B)+\Lambda_4(B). 
\end{multline}
Here, $\Lambda_3(B)$ means the number of $F_i$ such that $B$ corresponds to
Case \ref{case:ramify-sLF-nontrivial} at $F_i$ and $l_i=(m_i-1)/2$,
$\Lambda_4(B)$ means the number of $F_i$ such that $B$ corresponds to
Case \ref{case:sLF-trivial} and hence $\sL|_F \simeq \sO_F$.
%
%
From Proposition \ref{prop:dimIm-2ext2-I}, \eqref{eq:h0-2K-B}
and \eqref{eq:rk-3ext2}, we can deduce that
\begin{multline}\label{eq:EstimateRk-2ext2-1}
 \rk(H^1(\ad(f)))-2\ext^2(E,E)^0 -1\\
\geq  2d+4+2\Lambda_2(B)-\bigl[2\Lambda(X)+\Lambda(B)+(1/2)\Lambda_3(B)
 +\Lambda_4(B) \bigr],
\end{multline}
where $\Lambda(B)$ was defined at Proposition \ref{prop:h1-sRKoG}.
By its definition, one can check that 
$\Lambda(B)+\Lambda_4(B)\leq \Lambda(X)$ and
$\Lambda_2(B)+\Lambda_3(B)\leq \Lambda_1(B)$. Therefore
\eqref{eq:EstimateRk-2ext2-1} induces that 
$ \rk(H^1(\ad(f)))-2\ext^2(E,E)^0 -1 \geq 2d+4-(7/2)\Lambda(X) =2(d+2-(7/4)\Lambda(X))$,
and \eqref{eq:Since2(K-B)Positive} and \eqref{eq:EstimateRk-2ext2-1} 
induces that
$ \rk(H^1(\ad(f)))-2\ext^2(E,E)^0 -1 \geq 2d+4-3\Lambda(X)-\Lambda_1(B)
\geq 2d+4-3\Lambda(X)-[2(d-2-n)+\Lambda(X)] =8-4\Lambda(X)+2n \geq 4(2-\Lambda(X))$.
Theorem \ref{thm:rk-geq-3ext-ver1} follows from these equations and
Theorem \ref{thm:suffcdtn-moduli-can}.
\end{proof}


\section{Some elliptic surfaces with a few singular fibers}\label{sctn:ElptSurf-FewSingFib}
In this section we shall show the following Theorem.
%
\begin{thm}\label{thm:(2,m)}
In Setting \ref{stng:XandH}, we suppose that $X$ has two multiple fibers
with multiplicities $(m_1=2, m_2=m)$ with $m\geq 3$, and $d=\chi(\sO_X)=1$.\\
(i) If $M(c_2)$ is singular at a stable sheaf $E$, then 
$E$ always comes under Case I in Fact \ref{fact:RestrGenericFib}.\\
(ii) We consider in Setting \ref{stng:XHS}.
If $c_2\geq 3$ and if $M(c_2)$ is compact (e.g. $c_2$ is odd), then
$\kappa(M(c_2))=(\dim M(c_2)+1)/2$.
\end{thm}
\begin{rem}
By \cite[Cor. 7.17]{Frd:holvb}, the number of fibers with singular reduction
is $12d$ in Setting \ref{stng:XandH}.
Thus the assumption in Theorem \ref{thm:(2,m)} implies that both multiple fibers
and fibers with singular reduction are rather few.
\end{rem}
%
Let us begin with some lemmas.
\begin{lem}\label{lem:detf-0-A}
If a torsion-free rank-two sheaf $E$ on an elliptic surface has 
a traceless homomorphism $f: E \rightarrow E(K_X)$ satisfies that $\det(f)\neq 0$,
then $E_{\breta}$ is decomposable.
\end{lem}
\begin{proof}
The determinant of 
$f_{\bar{\eta}}: E_{\bar{\eta}} \rightarrow E(K_X)_{\bar{\eta}}\simeq E_{\bar{\eta}}$
is denoted as $\det(f_{\bar{\eta}})=-a^2$ with some
$a\in H^0(X_{\bar{\eta}},\sO)=\overline{k(\PP^1)}$.
Since $(f_{\breta}+a)(f_{\breta}-a)=0$ by Hamilton-Caylay's theorem, 
we have two decompositions by degree-zero line bundles
\[ 0 \longrightarrow Ker(f_{\breta}\pm a) \overset{i_{\pm}}{\longrightarrow} 
E_{\breta}  \overset{p_{\pm}}{\longrightarrow} Im(f_{\breta}\pm a) \longrightarrow 0 \]
such that $p_+\circ i_-$ is isomorphic. Thus $E_{\bar{\eta}}$ decomposes.
\end{proof}
\begin{lem}\label{lem:2D|Xeta=0}
Suppose that a singular point $E$ of $M(c_2)$ comes under Case II or Case III
in Fact \ref{fact:RestrGenericFib}, and so there is an extension
\begin{equation}\label{eq:exseq-genfib}
0 \longrightarrow \sO_{X_{\eta}}(D) \longrightarrow E_{\eta} \longrightarrow
   \sO_{X_{\eta}}(-D) \longrightarrow 0.
\end{equation}
When $p_g(X)=0$, $\sO_{X_{\eta}}(D)\simeq \sO_{X_{\eta}}(-D)$.
\end{lem}
\begin{proof}
Suppose not.
One can extend \eqref{eq:exseq-genfib} to an extension on $X$
\[ 0 \longrightarrow F \longrightarrow E \longrightarrow G \longrightarrow 0, \]
where $F$ and $G$ are torsion-free rank-one sheaves.
This induces a diagram of exact sequences
\begin{equation*}
\xymatrix{
  & \Hom_X(G,F(K_X))=0 \ar[d] &  & \Hom_X(F,F(K_X)) \ar[d] \\
0 \ar[r] & \Hom_X(G,E(K_X)) \ar[d] \ar[r] & \Hom_X(E,E(K_X)) \ar[r] &
 \Hom_X(F,E(K_X)) \ar[d] \\
 & \Hom_X(G,G(K_X)) & & \Hom_X(F,G(K_X))=0,} 
\end{equation*}
where the upper-left part and the lower-right part are zero since
$\sO_{X_{\eta}}(D) \not\simeq \sO_{X_{\eta}}(-D)$.
Thus $0<\dim\Hom(E,E(K_X))^{\circ} \leq 2p_g(X)$, but this is impossible
when $p_g(X)=0$.
\end{proof}
{\it Proof of Theorem \ref{thm:(2,m)}}:
First, suppose that there is a traceless homomorphism
$f: E \rightarrow E(K_X)$ with $\det(f)=0$.
This gives an exact sequence
\[ 0 \longrightarrow F=Ker(f) \longrightarrow E \longrightarrow G=Im(f) 
  \longrightarrow 0. \]
Let $B$ be the curve defined at Definition \ref{defn:B}, and put $c_1(F)=D$. 
Since $(K_X-B)\cdot \sO(1)\geq 0$ from the stability of $E$, and
$K_X$ is $\QQ$-equivalent to $\{ 1-(1/m_1)-(1/m_2)\}\cf$,
\begin{equation}\label{eq:B}
 B=a_1F_1+a_2F_2 \qquad (0\leq a_i\leq m_i-1).
\end{equation}
Then one can show that $2D+K_X-B=0$ in a similar way to Lemma \ref{lem:sF},
and that $\chi(\sO_X(D))=\chi(\sO_X)=d=1$, so
$h^0(\sO(D))\neq 0$ or $h^2(\sO(D))\neq 0$. However if $h^0(\sO(D))\neq 0$
then $D=0$ since $E$ is stable, and thus $K_X=B$, but this is impossible
since $p_g(X)=0$ in case of Theorem \ref{thm:(2,m)}.
As a result $h^0(\sO(K_X-D))=h^2(\sO(D))\neq 0$, and $K_X-D$ is described as
\begin{equation}\label{eq:KX-D}
 G=K_X-D= \lambda \cf +b_1F_1 +b_2F_2 \qquad (0\leq \lambda,\ 0\leq b_i \leq m_i-1).
\end{equation}
By combining \eqref{eq:B} with \eqref{eq:KX-D}, one has
\begin{equation}\label{eq:K-B+2Dis}
 0= K_X-B+2D=(1-2\lambda)\cf+ \sum_{i=1}^2 (m_i-3-a_i-2b_i)F_i.
\end{equation}
If $\lambda\geq 1$, then its right-hand side cannot be positive, and hence $\lambda=0$.
By Fact \ref{fact:O(F)|F},
$m_1-3-a_1-2b_1$ is a multiple of $m_1$.
In this way, we have $3+a_i+2b_i=m_i l_i$ with some natural number $l_i$ for $i=1,2$.
This and \eqref{eq:K-B+2Dis} imply $3-l_1-l_2=0$, and so $l_1$ equals $1$ or $2$.
If $l_1=1$, then $3+a_1+2b_1=2$, which cannot occur since $a_1,\ b_1\geq 0$.
If $l_1=2$, then $3+a_1+2b_1=4$, from which one can check $a_1=1$ and $b_1=0$.
Since $(K_X-B)\cdot \sO(1)\geq 0$, it should hold that
$0\leq 1-(a_1+1)/m_1-(a_2+1)/m_2$, but this is impossible for $m_1=2$ and 
$a_1=1$.

\vspace{5mm}

Therefore any traceless homomorphism $f:E\rightarrow E(K_X)$ has
$\det(f)\neq 0$. 
Assume that $E$ doesn't correspond to Case I.
By Lemma \ref{lem:detf-0-A}, $E$ corresponds to Case II.
Similarly to \eqref{eq:Y0-C}, $\det(f)$ induces
a double cover $\nu_0:Y_0 \rightarrow X$ with $\ZZ/2$-action $\sigma$
such that $\nu_* \sO_{Y_0}=\sO_X \oplus {\mathcal L}^{\vee}$, 
where ${\mathcal L}=\sO(K_X-B-B')$, since $p_g(X)=0$.
Remark that $Y_0$ is non-singular so we have $Y=Y_0$ and $\nu=\nu_0$
in \eqref{eq:Y0-C}, since $2K_X=2(-\cf+F_1+(m-1)F_2)=(m-2)F_2$
and $F_2$ is nonsingular by Setting \ref{stng:XandH}.
As discussed in Section \ref{sctn:rkH1-caseA1},
there are two exact sequences \eqref{eq:defn-Fpm} of $\nu^* E$ and,
similarly to Lemma \ref{lem:D+sD},
the first Chern class $D\in\Pic(Y)$ of $F_-$ satisfies that $D+\sigma(D)+K_X-B=0$, 
where $B\subset X$ is the curve defined at Definition \ref{defn:B}.
\begin{lem}\label{lem:D2=0}
We have $D^2=0$.
\end{lem}
\begin{proof}
 By Lemma \ref{lem:2D|Xeta=0}, $\sO_{Y_{\eta'}}(2D)=\sO_{Y_{\eta'}}$ and hence
$2D$ is linear equivalent to a divisor $D_0$ such that the image of its support
by $\pi\nu:Y\rightarrow X\rightarrow C$ is zero-dimensional.
It suffices to show that $D_0^2=0$ if $D_0\subset Y$ is a connected reduced curve
such that $\pi\nu(D_0)$ is a point.\par
(i) If $\nu$ ramifies at $D_0$, then $2D_0=\nu^{-1}(F_2)$ 
since $2K_X=(m-2)F_2$.
Thereby $(2D_0, D_0)=(\nu^{-1}(F_2),D_0)=(F_2,\nu(D_0))=(F_2,F_2)=0$ from
projection formula. \par
(ii) Suppose $\nu$ is \'{e}tale at $D_0$. Then 
$\nu(D_0)$ is connected, reduced and irreducible
by Setting \ref{stng:XandH}, and as a result
$\nu^{-1}\nu(D_0)$ is locally integral, that is, every connected component
is integral.
Therefore if $D_0 \neq \sigma(D_0)$, then $D_0\cap \sigma(D_0)$ is empty,
$\nu^{-1}\nu(D_0)=D_0 \sqcup \sigma(D_0)$, and thus
$ 2D_0^2=(D_0\sqcup \sigma(D_0))^2 =(\nu^{-1}\nu(D_0))^2=
\langle \nu(D_0), \nu\nu^{-1}\nu(D_0) \rangle_X= 
\langle \nu(D_0), 2\nu(D_0) \rangle_X=2(\nu(D_0))^2=0$
by projection formula.
If $D_0=\sigma(D_0)$, then $\nu^{-1}\nu(D_0)$ 
equals $D_0$, 
and we can verify $D_0^2=0$ from projection formula.
\end{proof}
From Lemma \ref{lem:Y0-Cartier} and
Lemma \ref{lem:D2=0}, it follows that
$K_Y=\nu^*(K_X\otimes {\mathcal L})\in \QQ \cdot \nu^*(\cf)$, 
$K_Y\cdot D=0$, and $\chi(\sO_Y(D))=\chi(\sO_Y)=2\chi(\sO_X)=2$.
\begin{clm}\label{clm:K-D>0}
$h^2(\sO_Y(D))$ is not zero.
\end{clm}
\begin{proof}
Otherwise $h^0(\sO(D))\neq 0$ for $\chi(\sO_Y(D))=2$, and then also 
$h^0(\nu^*E^{\vee\vee})=h^0(E^{\vee\vee})+h^0(E^{\vee\vee}\otimes{\mathcal L}^{\vee})$
is not zero. Because $E^{\vee\vee}$ is $\mu$-semistable and
$h^0({\mathcal L}^{\otimes 2})\neq 0$, 
either of the following exact sequences exists,
where $Z$ is a zero-dimensional subscheme:
\begin{align*}
 0 \longrightarrow \sO_X \longrightarrow & E^{\vee\vee} \longrightarrow
 \sO_X \otimes I_Z \longrightarrow 0 \qquad \text{or} \\
 0 \longrightarrow {\mathcal L} \longrightarrow & E^{\vee\vee} \longrightarrow
 {\mathcal L}^{\vee} \otimes I_Z \longrightarrow 0. 
\end{align*}
However, one can check that $\hom(E,E(K_X)) \leq 4p_g(X)=0$
when the former exists. When the latter exists, 
$h^0({\mathcal L}^{\otimes 2})\neq 0$ implies
${\mathcal L}^{\otimes 2}=\sO_X$, and consequently $\hom(E,E(K_X)) \leq 4p_g(X)=0$.
\end{proof}
Since $2K_X=\sum_{i=1}^2 (m_i-2)F_i$, 
$B+B'=\sum_i \ \lfloor (m_i-2)/2 \rfloor F_i$ by its definition. Thus
$B=\sum_i a_i F_i \ (0\leq a_i \leq \lfloor (m_i-2)/2 \rfloor)$, 
\begin{align}\label{eq:LandKY}
c_1({\mathcal L})= &K_X-(B+B')=  -\cf+\textstyle\sum_i \lceil m_i/2 \rceil F_i
\qquad \text{and}\\
K_Y= & \nu^*(K_X\otimes {\mathcal L})= \sum_i (\lceil m_i/2 \rceil-1)\nu^* F_i. \nonumber
\end{align}
By Claim \ref{clm:K-D>0}, $h^2(\sO_Y(D))=h^0(\sO_Y(K_Y-D))$ is not zero, and 
thereby $G=K_Y-D$ is positive. 
From Lemma \ref{lem:D+sD} and \eqref{eq:B}, 
we can deduce a $\ZZ/2$-equivariant isomorphism
\begin{equation}\label{eq:G+sG}
G+\sigma(G)=2K_Y+\nu^*(K_X-B)= \nu^*(\cf+\textstyle\sum_i (2\lceil m_i/2 \rceil -3-a_i)
F_i).
\end{equation}
%
Now let us use the assumption in Theorem \ref{thm:(2,m)} that $(m_1,m_2)=(2,m)$.
In this case, one can verify that $a_1=0$ from \eqref{eq:G+sG} and hence
\begin{equation}\label{eq:G+sG'}
G+\sigma(G)\in \Gamma(Y, \nu^* \sO(F_1+(2\lceil m_2/2 \rceil -3-a_2)F_2))^{\sigma}.
\end{equation}
However, $2K_X=(m-2)F_2$ implies that $\nu: Y\rightarrow X$ is \'{e}tale at $F_1$, 
${\mathcal L}|_{F_1}\not\simeq \sO_{F_1}$ by \eqref{eq:LandKY}, and then
$h^0(\nu^*(\sO_{F_1}))=1$.
As a result $\nu^{-1}F_1$ is integral, but
then \eqref{eq:G+sG'} never occur.
This is contradiction, and thereby $E$ corresponds to Case I.
Consequently we arrive at Theorem \ref{thm:(2,m)}(i).
Theorem \ref{thm:(2,m)}(ii) results from
the paragraph after \eqref{eq:relJH},
the proof of Corollary \ref{cor:KodDim-Moduli}, Theorem \ref{thm:rk-geq-3ext-ver1} and
Theorem \ref{thm:(2,m)}(i). \qed
%
%
\section{Example of singularities in $M(c_2)$}\label{sctn:eg-sing}
\begin{prop}\label{prop:Ex-Sing}
Fix a positive integer $d$, a non-negative integer $\Lambda$ and
a pair of integers $(m_1,\cdots m_{\Lambda})$ with $m_i\geq 2$.
Assume that $2d\geq \max(\Lambda-2, 4-\Lambda, 5-2\Lambda)$.
Then we have
an elliptic surface $X$ over $\PP^1$ 
such that Setting \ref{stng:XandH} holds,
$\chi(\sO_X)=d$, $\Lambda(X)=\Lambda$, and its multiple fibers have
the multiplicities $m_i$, and a constant $N$ as follows.
For any $c_2\geq N$, there is a rank-two sheaf $E$ with $(c_1(E),c_2(E))=(0, c_2)$
satisfying that 
$E$ is stable with respect to any $c_2$-suitable ample line bundle $H$, $E$ is of type I, and
and $\ext^2(E,E)^{\circ}=\hom(E,E(K_X))^{\circ }$ is not zero.
\end{prop}
\begin{proof}
First, we shall find an elliptic surface $B$ with a section, which will be
the Jacobian surface $J(X)$ of $X$ mentioned below.
\begin{fact}(\cite{Kas77:Weierstrass}, \cite{Miranda81:Weierstrass}. cf.
\cite[p.181, Thm.20]{Frd:holvb})\label{fact:Weier}
Let $d$ be a positive integer.
If $g_2\in \Gamma(\PP^1, \sO(4d))$ and $g_3 \in \Gamma(\PP^1, \sO(6d))$ are general
sections, then the closed subscheme $\bar{B}$ in 
$\PP_{\PP^1}(\sO(2d) \oplus \sO(3d) \oplus \sO)$ defined by the equation
$ y^2z= 4x^3-g_2xz^2-g_3 z^3$
is a surface with at worst rational double points, such that the natural morphism
$\bar{B}\rightarrow \PP^1$ is a flat family of irreducible curves of arithmetic genus
$1$. Its minimal resolution $B$ is a relatively minimal elliptic fibration 
with a section and $\chi(\sO_B)=d$. Conversely, any minimal elliptic fibration 
$B\rightarrow \PP^1$ with a section and $\chi(\sO_B)=d$ is described in this way.
\end{fact}
With Fact \ref{fact:Weier} in mind, we choose general sections
\begin{equation}\label{eq:g2qa}
 g_2\in \Gamma(\PP^1, \sO(4d)),\ q\in\Gamma(\PP^1, \sO(2d-\Lambda+2)),\ 
 \alpha\in \Gamma(\PP^1, \sO(2d-4+2\Lambda))
\end{equation}
such that $\alpha$ gives a square-free divisor and that
\begin{equation}\label{eq:supp-empty}
\operatorname{Supp}(g_2) \cap \operatorname{Supp}(q^2\alpha)=\emptyset.
\end{equation}
This is possible because of the assumption on $d$ in this proposition.
Then let $B$ be the closed subscheme in
$\PP_{\PP^1}(\sO(2d) \oplus \sO(3d) \oplus \sO)$ defined by the equation
\begin{equation}\label{eq:J(X)}
y^2z= 4x^3-g_2\ xz^2-q^2\alpha\ z^3.  
\end{equation}
When $g_2,\ q,\ \alpha$ are general, one can verify that $B$ is a non-singular
elliptic surface over $\PP^1$ with a section and $\chi(\sO_B)=d$.
Singular fibers of $B\rightarrow \PP^1$ are integral curves with one ordinary
double point by \eqref{eq:supp-empty}.\par
Next, we choose distinct points 
$p_1,\dots, p_{\Lambda} \in \operatorname{Supp(\alpha)}$.
This is possible since $2d-4+2\Lambda\geq \Lambda$ by assumption on $d$.
The fiber $B_{p_i}$ over $p_i$ is non-singular from \eqref{eq:supp-empty}.
Using divisors of order $m_i$ on $B_{p_i}$
and the logarithmic transformation by Kodaira 
(cf. \cite[Sect. V.13]{BPV:text}, \cite[Thm. I.6.7, Thm. I.6.12]{FM:4mfds}),
we can get an algebraic elliptic surface $\pi:X\rightarrow \PP^1$ such that
$X$ has multiple fibers with multiplicities $m_i$ over $p_i$,
$X|_{(\PP^1-\{ p_i\})}$ is locally isomorphic to $B|_{(\PP^1-\{ p_i\})}$
(in analytic topology), and
its Jacobian surface $J(X)$ is isomorphic to $B$.
This $X$ satisfies assumptions in Proposition \ref{prop:Ex-Sing}.\par
Let $\nu_C: C\rightarrow \PP^1$ be the double cover given by
$\alpha\in \Gamma(\PP^1, \sO(2d-4+2\Lambda))$ with a $\ZZ/2$-action $\sigma_C$, 
$s\in \Gamma(\PP^1, \sO(d-2+\Lambda))$ be the section such that $s^2=\alpha$ and
$\eta'$ be $\Spec(k(C))$.
\begin{clm}\label{clm:De'}
Some member $D_{\eta'}$ in $\Pic^0(X_{\eta'})$ does not descend to a divisor 
on $X_{\eta}$, and satisfies that
$D_{\eta'}+\sigma_C(D_{\eta'})=0$ as $\ZZ/2$-equivariant divisors.
\end{clm}
\begin{proof}
 About $B$ at \eqref{eq:J(X)},
 a point $(x,y)=(0, qs)$ in $B_{\eta'}$ does not descend to a point 
 in $B_{\eta}$ since $\sigma_C(0,qs)=(0,-qs)\neq (0,qs)$, and satisfies that
 $(0,qs)+\sigma_C(0,qs)=0$. 
 Since $B_{\eta}\simeq J(X_{\eta})$, $(0,qs)$ corresponds to a member $D_{\eta'}$
 of $J(X_{\eta'})=\Pic^0(X_{\eta'})$, which has such properties as in 
 Claim \ref{clm:De'}. 
\end{proof}
Now we have a natural section
$\tau_{p_1+\dots +p_{\Lambda}}\in \Gamma(\PP^1, \sO(\Lambda))$, and
$\alpha'=\alpha/\tau_{p_1+\dots +p_{\Lambda}} \in \Gamma(\PP^1,\sO(sd-4+\Lambda))$.
It holds that
\[ 2 \Bigl(K_X-\textstyle\sum_i \lfloor (m_i-2)/2 \rfloor F_i \Bigr)=
 \bigl(2d-4+\Lambda \bigr)\ \cf  +\sum_i^{od} \ F_i, \]
where the symbol $\sum_i^{od}$ means the summation runs over all $i$ such that
$m_i$ is odd.
Thereby we have a square-free section
$\tilde{\alpha}=\alpha'\cdot \prod_i^{od} \tau_{F_i} \in \Gamma(X, 2{\mathcal L})$, 
where we put
$ {\mathcal L}=\sO(K_X-\sum_i \lfloor (m_i-2)/2 \rfloor F_i )$, and
$\tau_{F_i}$ is the section corresponding to $F_i$.
Let $\nu: Y\rightarrow X$ be the double cover given by $\tilde{\alpha}$
with a $\ZZ/2$-action $\sigma$, and
$t\in \Gamma(Y, {\mathcal L})$ be the section such that $t^2=\tilde{\alpha}$.
Remark that $\operatorname{Supp}(\tilde{\alpha})$ is non-singular from 
\eqref{eq:supp-empty}, and so is $Y$.\par
 From the relation between $\alpha,\ \alpha', \tilde{\alpha}$ and $t$,
 one can verify that 
 \begin{equation}\label{eq:get-piC}
 \alpha= \lambda_1 \tilde{\alpha}\cdot \Bigl\{ \prod_i \tau_{F_i}^{\lfloor m_i/2 \rfloor}
 \Bigr\}^2 \quad \text{and}\quad
 \alpha= \Bigl\{ \lambda_2 t\cdot \prod_i \tau_{F_i}^{\lfloor m_i/2 
 \rfloor} \Bigr\}^2
 \end{equation}
 with nonzero constants $\lambda_i$.
 Thus we can obtain a morphism $\pi_C: C\rightarrow \PP^1$ such that
$\nu_C \circ \pi_C: Y\rightarrow C \rightarrow \PP^1$ equals to $\pi\circ\nu$.
 Since $\tau_{F_i}$ is an unit of the ring of $X_{\eta}$,
 the induced morphism $Y_{\eta'} \rightarrow X_{\eta}\times_{\eta} \eta'$
 is isomorphic by the left side of \eqref{eq:get-piC}. 
%
%
Thus one can extend the divisor $D_{\eta'}$ on $X_{\eta'}$ 
at Claim \ref{clm:De'} to a divisor $D$ on $Y$ such that $D\cdot \nu^{-1}(\cf)=0$.
Now let us consider a rank-two vector bundle $\nu_*\sO(-D)$ on $X$.
\begin{clm}\label{clm:E0}
(a) $\nu_*\sO(-D)|_{X_{\eta}}$ is stable. \quad
(b) $\Hom(\nu_*\sO(-D),\ \nu_*\sO(-D)(K_X))^{\circ}\neq 0$.
\end{clm}
\begin{proof}
(a) By \cite[p.48]{Frd:holvb}, there is an exact sequence
 \begin{equation}\label{eq:E0-exseq}
 0 \longrightarrow \sO_Y(-\sigma(D))\otimes \nu^*{\mathcal L}^{-1} 
 \longrightarrow \nu^*\nu_* \sO(-D) \longrightarrow \sO_Y(-D) \longrightarrow 0. 
 \end{equation}
 Suppose that $\nu_*\sO(-D)|_{X_{\eta}}$ is not stable.
 Then it has such a subsheaf $F_{\eta}$ as $\deg(F_{\eta})\geq 0$.
 One can verify that $\nu^*(F_{\eta})$ is isomorphic to either 
 $\sO_Y(-D)|_{Y_{\eta'}}$ or
 $\sO_Y(-\sigma(D))\otimes \nu^*{\mathcal L}^{-1}|_{Y_{\eta'}}= 
 \sO_Y(-\sigma(D))_{Y_{\eta'}}$ 
by considering their degrees and \eqref{eq:E0-exseq}.
This contradicts to Claim \ref{clm:De'}. \\
(b) Since $K_X\otimes {\mathcal L}^{-1}=\sO(\sum_i \lfloor (m_i-2)/2 \rfloor F_i)$,
there is a non-zero section $\iota$ in $\Gamma(X, K_X\otimes {\mathcal L}^{-1})$.
Then $t\iota\in\Gamma(Y, \nu^*K_X)$ satisfies that
$\sigma(t\iota)=-t\iota$, and gives a homomorphism
\[ \nu_*(\times t\iota): \nu_*\sO(-D) \longrightarrow \nu_*\sO(-D)(K_X).\]
We have a commutative diagram
\begin{equation*}
\xymatrix{
 \sO_Y(-\sigma(D))\otimes \nu^*{\mathcal L}^{-1} \ar[r] \ar[d]^{\times \sigma(t\iota)} &
 \nu^*\nu_* \sO(-D) \ar[r] \ar[d]^{\nu_*(\times t\iota)} & \sO_Y(-D) 
 \ar[d]^{\times t\iota} \\
 \sO_Y(-\sigma(D))\otimes \nu^*{\mathcal L}^{-1}(K_X) \ar[r] &
 \nu^*(\nu_* \sO(-D)(K_X)) \ar[r]  & \sO_Y(-D)(K_X),}
\end{equation*}
where two lines are \eqref{eq:E0-exseq}, and so
$\nu^* \operatorname{tr}(\nu_*(\times t\iota))=t\iota+\sigma(t\iota)=0$.
\end{proof}
Let us consider its first Chern class.
From \cite[p.47, Prop. 27]{Frd:holvb},
$c_1(\nu_*\sO(-D))=\nu_*(-D)-c_1({\mathcal L})$, where $\nu_*$ is induced map
between divisors.
By Claim \ref{clm:De'}, $D+\sigma(D)=\nu^*(D_0)$, where $D_0$ is a divisor on $X$
whose support lies in fibers of $X\rightarrow \PP^1$.
Then $\nu_*(D)-D_0=D_1$ satisfies that $2D_1=0$, so $\chi(\sO(D_1))=d>0$,
which implies that $h^0(\sO(D_1))$ or $h^0(\sO(K_X-D_1))$ is not zero.
Thereby the support of $\nu_*(D)$ lies in fibers of $X\rightarrow \PP^1$.
From Remark \ref{rem:nu-fib}, we can assume that 
\begin{equation}
c_1(\nu_*\sO(-D))= \textstyle\sum_i^{ev} \epsilon_i F_i \qquad
(\epsilon_i=0\ \text{or}\ 1),
\end{equation}
where the symbol $\sum_i^{ev}$ means the summation runs over all $i$ such that
$m_i$ is even.
\begin{rem}\label{rem:nu-fib}
Concerning $\nu: Y\rightarrow X$, the following holds. \\
(a) $\nu^{-1}(\cf)=\cf' \sqcup \sigma(\cf')$ and $\nu_*(\cf')=\cf$. \\
(b) If $m_i$ is even, then $\nu$ is etale at $F_i$, $\nu^{-1}(F_i)$ 
is integral and $\nu_*(\nu^{-1}(F_i))=2F_i$.\\
(c) If $m_i$ is odd, then $\nu$ ramifies at $F_i$, $\nu^{-1}(F_i)=2F'_i$,
$\sigma(F_i')=F'_i$ and $\nu_*(F'_i)=F_i$.
\end{rem}
\begin{proof}
 (a) is obvious from the existence of $\nu_C$.
 (b) When $m_i$ is even, $\nu$ is etale at $F_i$ from the definition of $\tilde{\alpha}$
 and $\nu$. Since ${\mathcal L}|_{F_i}=\sO((m_i/2)F_i)|_{F_i}$ is not isomorphic
 to $\sO_{F_i}$ by Fact \ref{fact:O(F)|F}, we have $h^0(\sO)=1$, and thereby
 $\nu^{-1}(F_i)$ is integral.
 (c) When $m_i$ is odd, $\operatorname{Supp}(\tilde{\alpha})$ contains $F_i$,
 so $\nu$ ramifies at $F_i$. 
\end{proof}
\begin{clm}\label{clm:E1}
There is a rank-two sheaf $E_1$ such that $c_1(E_1)=0$, $E_{1,\eta}$ is stable, and
$\Hom(E_1,E_1(K_X))^{\circ}$ is not zero.
\end{clm}
\begin{proof}
When $c_1(\nu_*\sO(-D))=0$, it suffices to put $E_1=\nu_*\sO(-D)$.
Let us consider when $c_1(\nu_*\sO(-D))=F_1$; the proof similarly proceeds 
in general case.
There is a nonzero traceless homomorphism
$g: \nu_*\sO(-D) \rightarrow \nu_*\sO(-D)(K_X)$ by Claim \ref{clm:E0}.\par
First, we suppose that $g|_{F_1}=0$.
There is a surjection $\nu_*\sO(-D)\rightarrow L_{F_1}$ to a line bundle on $F_1$,
and let $E_1$ be its kernel. Then $c_1(E_1)=0$, 
$E_{1,\eta}\simeq \nu_*\sO(-D)|_{X_{\eta}}$ is stable, and
$g$ induces a homomorphism $g:E_1 \rightarrow E_1(K_X)$. \par
Next, we suppose that $g|_{F_1}\neq 0$.
By \eqref{eq:E0-exseq} $\nu^*(\nu_*\sO(-D)|_{F_1})$ is semistable, and so
$\nu_*\sO(-D)$ is semistable on a non-singular elliptic curve $F_1$.
From \cite{Aty:VB-ellipt}, it holds that
(a) $\nu_*\sO(-D)|_{F_1}$ is isomorphic to $L_1 \oplus L_2$ with degree-zero
line bundles $L_i$, or
(b) $\nu_*\sO(-D)|_{F_1}$ is isomorphic to ${\mathcal E}\otimes L$,
where ${\mathcal E}$ is a non-trivial extension of $\sO_{F_1}$ by $\sO_{F_1}$.
Notice that a nonzero homomorphism 
$g|_{F_1}: \nu_*\sO(-D)|_{F_1} \rightarrow \nu_*\sO(-D)(K_X)|_{F_1}$
cannot exist in Case (b) since $\sO(K_X)|_{F_1}= \sO(-F_1)|_{F_1}$ is not
isomorphic to $\sO_{F_1}$.
Thereby Case (a) holds.
Since $g|_{F_1}\neq 0$ and $\sO(K_X)|_{F_1}\not\simeq \sO_{F_1}$,
we can suppose that $\Hom(L_1, L_2\otimes \sO_{F_1}(-F_1))$ is not zero,
and then $L_1\simeq L_2\otimes \sO_{F_1}(-F_1)$ for their degree are same.
Thus $\nu_*\sO(-D)|_{F_1}$ is isomorphic to 
$\{ \sO_{F_1}\oplus \sO_{F_1}(-F_1)\}\otimes L$.
$g$ induces two homomorphisms
\begin{align*}
g_{21}:= p_2\circ g|_{F_1}\circ i_1 :\ & \sO_{F_1} \longrightarrow \sO_{F_1}(-2F_1) 
 \quad\text{and}\\
g_{12}:=p_1\circ g|_{F_1}\circ i_2 :\ & \sO_{F_1}(-F_1) \longrightarrow \sO_{F_1}(-F_1). 
\end{align*}
Assume that $g_{21}$ is zero.
Define $E_1$ to be the kernel of a natural surjection
$\nu_*\sO(-D)\rightarrow \nu_*\sO(-D)|_{F_1} \rightarrow \sO_{F_1}(-F_1)$.
Then $c_1(E_1)=0$, $E_{1,\eta}\simeq \nu_*\sO(-D)|_{X_{\eta}}$ is stable, and
$g$ induces $g: E_1\rightarrow E_1(K_X)$. 
Also when the $g_{12}$ is zero, we can similarly find such a sheaf $E_1$ as 
in Claim \ref{clm:E1}.\par
Suppose that both $g_{21}$ and $g_{12}$ are not zero. Then they are isomorphic
and we can assume that $g_{12}$ is the identity mapping.
Describe $\sO_{F_1}(-F_1)$ as $\sO_{F_1}(q-r)$, and fix nonzero sections
$\tau_{q}$ of $\Gamma(\sO_{F_1}(q))$ and $\tau_{r}$ of $\Gamma(\sO_{F_1}(r))$.
There is a nonzero constant $\lambda$ such that
$g_{21}(1)=\lambda (\tau_q/\tau_r)^2$.
We define an injective homomorphism by
\[ j: \sO_{F_1}(-r) \longrightarrow \sO_{F_1}\oplus \sO_{F_1}(q-r)
 =\sO_{F_1}\oplus \sO_{F_1}(-F_1) \qquad 
\bigl(\tau_r^{-1} \mapsto (1/\sqrt{\lambda}, \tau_q/\tau_r) \bigr),\]
and define $E_1$ to be the kernel of a natural surjection
\[ \nu_*\sO(-D) \longrightarrow \nu_*\sO(-D)|_{F_1}\simeq 
\{ \sO_{F_1}\oplus \sO_{F_1}(-F_1)\}\otimes L \longrightarrow \Cok(j)\otimes L.\]
Then one can verify that 
$c_1(E_1)=0$, $E_{1,\eta}\simeq \nu_*\sO(-D)|_{X_{\eta}}$ is stable, and
$g$ induces $g: E_1\rightarrow E_1(K_X)$. 
\end{proof}
\begin{rem}\label{rem:E2}
For such a sheaf $E_1$ as in Claim \ref{clm:E1}, there is a subsheaf
$E_2$ of $E_1$ such that $E_1/E_2\simeq \CC$ and that 
$\Hom(E_2,E_2(K_X))^{\circ}\neq 0$.
\end{rem}
\begin{proof}
Take a nonzero element $g$ of $\Hom(E_1,E_1(K_X))^{\circ}$.
Let $p$ be a closed point in $X$ such that $E_1$ is locally free at $p$.
By standard linear algebra,
the linear map $g\otimes k(p): E_1\otimes k(p) \rightarrow E_1(K_X)\otimes k(p)$
has an invariant quotient vector space of rank one, say $Q(p)$.
When we define $E_2$ to be the kernel of a natural surjection
$E_1 \rightarrow E_1\otimes k(p) \rightarrow Q(p)$, $g$ induces a homomorphism
$g: E_2 \rightarrow E_2(K_X)$.
\end{proof}
Fact \ref{fact:XeStable}, Claim \ref{clm:E1} and Remark \ref{rem:E2}
complete the proof of Proposition \ref{prop:Ex-Sing}.
\end{proof}
\begin{prop}\label{prop:irred}
Let $X$ be arbitrary nonsingular projective surface such that $K_X$ is nef.
Fix a (finite) polyhedral cone ${\mathcal S}$ in the closure of the ample 
cone $\Amp(X)$ of $X$ such that ${\mathcal S}\cap \partial \Amp(X) \subset \RR K_X$.
For an ample line bundle $H$ on $X$, let $\bar{M}_H(c_1,c_2)$ denote
the moduli scheme of $H$-semistable rank-two sheaves with fixed Chern classes
$(c_1,c_2)\in \Pic(X)\times\ZZ$.
If $H$ belongs to ${\mathcal S}$ and if
$4c_2-c_1^2$ is sufficiently large w.r.t. $X$ and ${\mathcal S}$, then
$M_H(c_1,c_2)$ is normal, and
$\bar{M}_H(c_1,c_2)$ is irreducible.
\end{prop}
\begin{proof}
Fix an ample line bundle $H_0\in {\mathcal S}$, and $H$ be an ample line bundle
lying on the line segment connecting $H_0$ and $K_X$.
By \cite{GL:irreducible}, $\bar{M}_{H_0}(c_1,c_2)$ is irreducible
when $c_2$ is sufficiently large w.r.t. $(X,H_0,c_1)$.
We put 
\[A_{H_0}(c_1,c_2)=\{ E\in M_H(c_1,c_2) \bigm| \Ext^2(E,E)^0\neq 0\}. \]
The number of moduli of $A_{H_0}(c_1,c_2)$ is less than
$3c_2+C\sqrt{c_2}+C$, where $C$ is a constant depending only on $(X,H_0,c_1)$
by \cite[Lemma 1.3]{Li:kodaira}.
Similarly to \cite[Lem. 2.2]{Yam:flip}, one can verify the following:
Let $P_H(H_0) \subset \bar{M}_H(c_1,c_2)$ be the set of $H$-semistable sheaves that is not $H_0$-semistable,
and let $P_{H_0}(H) \subset \bar{M}_{H_0}(c_1,c_2)$ be the set $H_0$-semistable sheaves that is not $H$-semistable.
Then both $\dim P_H(H_0)$ and $\dim P_{H_0}(H)$ are less than $3c_2+B\sqrt{c_2}+B$, where
$B$ is a constant depending only on $(X,c_1,H_0)$.
Combining these facts, we obtain this proposition.
\end{proof}
%
%
%
%
\section{When $E$ applies to Case II and is locally free}\label{sctn:CaseII}

In this section, we works in Setting \ref{stng:f20-locfree}, and
show Theorem \ref{thm:R-geq-1}.

\begin{stng}\label{stng:f20-locfree}
$X$ satisfies that $d=1$ and $\Lambda(X)=2$.
$E\in M(c_2)$ satisfies $\Ext^2(E,E)^0\neq 0$,
applies to Case II in Fact \ref{fact:RestrGenericFib}, and is locally free.
\end{stng}

\begin{fact}\label{fact:relExt}
(\cite{Banica:Ext})
(a) If ${\mathcal F}$ is a torsion-free $\sO_X$-module, then
it is flat over $\PP^1$. \\
(b) Put $\omega_{X/\PP^1}=\omega_X\otimes \omega_{\PP^1}^{-1}$.
If ${\mathcal F},\ {\mathcal G}$ is flat over $\PP^1$, then \\
$ Hom_{\PP^1}(Ext^1_{\pi}({\mathcal G},{\mathcal F}),\sO_{\PP^1}) \simeq
Hom_{\pi}({\mathcal F},{\mathcal G}\otimes \omega_{X/\PP^1})$.
\end{fact}

\begin{lem}\label{lem:f20}
(1) Suppose $E\in M(c_2)$ applies to Case II in Fact  \ref{fact:RestrGenericFib}.
Then any nonzero homomorphism $f\in\Hom(E,E(K_X))^0$ satisfies $f^2=0$, so we get
an exact sequence
\begin{equation}\label{eq:exseq-f}
  0 \longrightarrow F=Ker(f)=\sO(D)\otimes I_{Z'} \overset{I}{\longrightarrow} E 
  \overset{P}{\longrightarrow}
  G=Im(f)=\sO(-D)\otimes I_Z \longrightarrow 0,
\end{equation}
inclusion $\iota:G \hookrightarrow F(K_X)$, and
a positive divisor $B\subset X$ such that $2D+K_X-B=0$.
It holds that $(K_X-B)\cdot \sO(1)>0$ or that $(K_X-B)\cdot \sO(1)=0$ and $l(Z')>l(Z)$.
In addition, the natural map $\Hom(I_Z, \sO(B)\otimes I_{Z'})=\Hom(G,F(K_X))\rightarrow \Hom(E,E(K_X))^0$ 
is isomorphic.\\
(2) Conversely, if a torsion free sheaf $E$ is decomposed
as in \eqref{eq:exseq-f} satisfying all conditions in (1), and if
\eqref{eq:exseq-f} does not split at the generic fiber $X_{\eta}$, then
$E$ is stable with respect to any $c_2(E)$-suitable ample line bundle.
\end{lem}
\begin{proof}
Suppose $f^2\neq 0$. Then one can describe $E_{\bar{\eta}}$ as in \eqref{eq:defn-Fpm}.
Natural map $F_- \rightarrow \nu_0^* E \rightarrow G_+$ induces the splitting of 
\eqref{eq:defn-Fpm}${}_{\bar{\eta}}$, so $E_{\bar{\eta}}$ should be decomposable. Thereby $f^2=0$.
It is easy to check the left of the proof.
\end{proof}

\begin{lem}[16/11/3, 12/3]\label{prop:ext1(EE)decomp} 
In Setting \ref{stng:f20-locfree}, 
one can describe the torsion parts and the pure parts of
the following natural exact sequences.
{\footnotesize
\begin{equation}\label{eq:ext1(EE)decomp}
 \xymatrix{
  Ext^1_{\pi}(G,F) \ar[d] \ar[r] &
  Ext^1_{\pi}(G,E) \ar[d] \ar@{>>}[r] &
  Ext^1_{\pi}(G,G) \ar[d] \\
  Ext^1_{\pi}(E,F) \ar@{>>}[d] \ar[r] &
  Ext^1_{\pi}(E,E) \ar@{>>}[d] \ar@{>>}[r] &
  Ext^1_{\pi}(E,G) \ar@{>>}[d] \\
  Ext^1_{\pi}(F,F) \ar[r] &
  Ext^1_{\pi}(F,E) \ar@{>>}[r] &
  Ext^1_{\pi}(F,G) }
\end{equation}
}
{\footnotesize
\begin{equation}\label{eq:ext1(EEK)decomp}
\xymatrix{
  Ext^1_{\pi}(G,F(K_X)) \ar[d] \ar[r] &
  Ext^1_{\pi}(G,E(K_X)) \ar[d] \ar@{>>}[r] &
  Ext^1_{\pi}(G,G(K_X)) \ar[d] \\
  Ext^1_{\pi}(E,F(K_X)) \ar@{>>}[d] \ar[r] &
  Ext^1_{\pi}(E,E(K_X)) \ar@{>>}[d] \ar@{>>}[r] &
  Ext^1_{\pi}(E,G(K_X)) \ar@{>>}[d] \\
  Ext^1_{\pi}(F,F(K_X)) \ar[r] &
  Ext^1_{\pi}(F,E(K_X)) \ar@{>>}[r] &
  Ext^1_{\pi}(F,G(K_X)) 
}
\end{equation}
}
Here we put
$l'_{BZ}=  \min\left\{ l \bigm|
h^0(\sO(2K_X-B+l{\bf f})\otimes I_Z)\neq 0  \right\}$.
\begin{enumerate}[(a)]
\item 
$\pur Ext^1(G,F)=\sO(l'_{BZ}-2)$. \label{enum:(1,1a)}
\item 
$\pur Ext^1_{\pi}(E,F)\simeq \sO(-d)=R^1\pi_*(\sO_X)$. \label{enum:(2,1a)}
\item 
$\tor Ext^1_{\pi}(E,E)\rightarrow \tor Ext^1_{\pi}(E,G)$ is surjective. \label{enum:(2,2-2,3a)}
\item 
$\pur Ext^1_{\pi}(E,G)=\sO(-2)$. \label{enum:pur(2,3)}
\item 
$l(\tor Ext^1_{\pi}(E,G))=3l(Z)$. \label{enum:tor(2,3)}
\item 
$\tor Ext^1_{\pi}(E,G)\rightarrow \tor Ext^1_{\pi}(F,G)$ is surjective. \label{enum:(2,3-3,3a)}
\item 
$Ext^1_{\pi}(E,F(K_X))\rightarrow Ext^1_{\pi}(E,E(K_X))$ is injective. \label{enum:(K2,1-2,2a)}
\item 
$Hom_{\pi}(E,G)\rightarrow Hom_{\pi}(E,F(K_X))$ is isomorphic. \label{enum:(h2,3-hK2,1)}
\item 
$Hom_{\pi}(E,E(K_X)) \rightarrow Hom_{\pi}(E,G(K_X))$ is surjective. \label{enum:(hK2,2-2,3a)}
\item 
$I^*: Hom_{\pi}(E,E(K_X)) \rightarrow Hom_{\pi}(F,E(K_X))$ is surjective. \label{enum:H0(p_*(K))}
\item 
$Hom_{\pi}(F,G(K_X))=\sO(-l'_{BZ})$. \label{enum:(hK3,3a)}
\end{enumerate}
\end{lem}
\begin{proof}
\eqref{enum:(1,1a)}\
$\Hom_{\PP^1}(Ext^1_{\pi}(G,F), \sO_{\PP^1})\simeq Hom_{\PP^1}(F,G(K_X))\otimes K_{\PP^1}^{-1}
=\sO(2-l'_{BZ})$
by the below-mentioned \eqref{enum:(hK3,3a)} and Fact \ref{fact:relExt}(b). \\
\eqref{enum:(2,1a)}\
$\rk(\pur Ext^1_{\pi}(E,F)=1$ from Setting \ref{stng:f20-locfree}. 
By comparing their ranks, one can check that
$\pur Ext^1_{\pi}(E,F)$ $\rightarrow$ $\pur Ext^1_{\pi}(F,F)$ is isomorphic, and
$\pur Ext^1_{\pi}(F,F)=R^1\pi_*(\sO_X)=\sO(-d)$ by \cite[III.11.2, V.12.2]{BPV:text}.\\
\eqref{enum:(2,2-2,3a)}\ 
From \eqref{enum:H0(p_*(K))}, natural sequence
\[ 0 \longrightarrow Hom_{\pi}(G,E(K_X)) \longrightarrow Hom_{\pi}(E,E(K_X)) \longrightarrow
 Hom_{\pi}(F,E(K_X)) \longrightarrow 0 \]
is exact. From Fact \ref{fact:relExt}, natural sequence
\begin{equation}\label{eq:pur-IP_*}
0 \longrightarrow \pur Ext^1_{\pi}(E,F) \longrightarrow \pur Ext^1_{\pi}(E,E) \longrightarrow \pur Ext^1_{\pi}(E,G) \longrightarrow 0
\end{equation}
is exact. Next, $\Hom_{\pi}(F,E(K_X))\rightarrow \Hom_{\pi}(F,G(K_X))$ is zero map
since $E$ applies to Case II, so $Hom_{\pi}(G,G) \rightarrow Hom_{\pi}(E,G)$ is isomorphic.
This decomposes as
\[ Hom_{\pi}(G,G)\simeq \pi_*(\sO_X) \overset{\operatorname{Id}}{\longrightarrow}  Hom_{\pi}(E,E) \overset{P_*}{\longrightarrow}
 Hom_{\pi}(E,G), \]
and thereby $P_*$ is surjective. Hence
\[ 0 \longrightarrow Ext^1_{\pi}(E,F) \longrightarrow Ext^1_{\pi}(E,E) \longrightarrow Ext^1_{\pi}(E,G) \longrightarrow 0 \]
is exact.
This, \eqref{eq:pur-IP_*} and the Snake lemma deduce \eqref{enum:(2,2-2,3a)}.\\
\eqref{enum:pur(2,3)}\ 
$Hom_{\PP^1}(Ext^1_{\pi}(E,G),\sO_{\PP^1})\simeq Hom_{\pi}(G,E(K_X))\otimes \sO(2) \simeq
Hom_{\pi}(G,F(K_X))=\pi_*(\sO(B))\otimes \sO(2) =\sO(2)$
by Fact \ref{fact:relExt}, Setting \ref{stng:f20-locfree}, and Lemma \ref{lem:easy-prpty-BZ} \eqref{enum:B-ni}.\\
\eqref{enum:tor(2,3)}\
From Setting \ref{stng:f20-locfree},
$Hom_{\pi}(E,G)\simeq Hom_{\pi}(G,G)=\sO_{\PP^1}$, 
$\hom(E,G)=\hom(G,G)=1$,\ $\ext^2(E,G)=\hom(G,E(K_X))=\hom(G,F(K_X))=h^0(\sO(B))=1$.
Hence $\ext^1(E,G)=-\chi(E,G)+\hom(E,G)+\ext^2(E,G)=3l(Z)-2d+2=3l(Z)$, and
\begin{multline*}
h^0(\tor Ext^1_{\pi}(E,G))=h^0(Ext^1_{\pi}(E,G))-h^0(\pur Ext^1_{\pi}(E,G))
=\ext^1(E,G) \\
-h^1(Hom_{\pi}(E,G))-h^0(\pur Ext^1_{\pi}(E,G)) 
=3l(Z)-h^1(\sO_{\PP^1})-h^0(\sO_{\PP^1}(-2))=3l(Z).
\end{multline*}
\eqref{enum:(2,3-3,3a)}\ 
In the third column of \eqref{eq:ext1(EE)decomp}, 
the rank every components is one since $E$ applies to Case II.
Thereby the image of 
$Ext^1_{\pi}(G,G)$ $\rightarrow$ $Ext^1_{\pi}(E,G)$ is torsion sheaf,
and then
$\pur Ext^1_{\pi}(E,G)$ $\rightarrow$ $\pur Ext^1_{\pi}(F,G)$ is isomorphic.
One can verify \eqref{enum:(2,3-3,3a)} from the Snake lemma.\\
\eqref{enum:(K2,1-2,2a)}\ This results from the below-mentioned \eqref{enum:(hK2,2-2,3a)}.\\
\eqref{enum:(h2,3-hK2,1)}\ 
From Lemma \ref{lem:easy-prpty-BZ},
$Hom_{\pi}(E,F(K_X))\simeq Hom_{\pi}(G,F(K_X))=\pi_*(\sO(B))=\sO_{\PP^1}$, 
$Hom_{\pi}(E,G)\simeq Hom_{\pi}(G,G)=\pi_*(\sO)=\sO_{\PP^1}$,
and hence they are isomorphic.\\
\eqref{enum:(hK2,2-2,3a)}\
In the natural commutative diagram
\begin{equation}\label{eq:hK13-23a}
\xymatrix{
 \pi_* \sO_X(K_X) \ar[r] \ar[d]^{\operatorname{Id}} & Hom_{\pi}(G,G(K_X)) \ar[d]^{P^*} \\
 Hom_{\pi}(E,E(K_X)) \ar[r]^{P_*} & Hom_{\pi}(E,G(K_X)),
}
\end{equation}
$P^*$ is injective between rank-one line bundles, so
$\Cok(P^*)$ is a torsion subsheaf of torsion-free sheaf $Hom_{\pi}(F,G(K_X))$,
and thereby $P^*$ is surjective.
Since the upper vertical arrow is isomorphic, $P_*$ is surjective.\\
\eqref{enum:H0(p_*(K))}\
Since $E$ applies to Case II, $Hom_{\pi}(F,E(K_X))\rightarrow Hom_{\pi}(F,G(K_X))$ is zero map,
so \\
$Hom_{\pi}(F,F(K_X))\rightarrow Hom_{\pi}(F,E(K_X))$ is isomorphism. 
This decomposes as
\[ Hom_{\pi}(F,F(K_X))=\pi_*(\sO_X(K_X)) \overset{\operatorname{Id}}{\longrightarrow}  Hom_{\pi}(E,E(K_X)) \overset{I^*}{\longrightarrow}
 Hom_{\pi}(F,E(K_X)), \]
and tereby $I^*$ is surjective.\\
\eqref{enum:(hK3,3a)}\ $Hom_{\pi}(F,G(K_X))=\pi_*(\sO(2K_X-B)\otimes I_Z)=\sO(-l'_{BZ})$.
\end{proof}

\begin{lem}\label{lem:f_*+f^*=0}
Let ${\mathcal U}=\{U_i\}$ be an open affine cover of $X$, and
let $\gamma$ be an element of the \v{C}ech cohomology 
$\check{{\rm H}}^1({\mathcal U}, Hom_X(E,E)^0 )$ $\subset$ $\Ext^1(E,E)^0$.
Suppose that $f\in \Hom(E, E(K_X))^0$ satisfies $f^2=0$,
so we have
$I,\ P,\ B$, and $\iota$ similarly to Lemma \ref{lem:f20}, and
$\tr\bigl( I^*(\iota\circ P)_* (\gamma)\bigr) \in \HH^1(K_X) $.
Then $f^*(\gamma)+f_*(\gamma)=\tr\bigl(I^*(\iota\circ P)_* (\gamma) \bigr)\cdot\Id$
in $\Ext^1(E,E(K_X))$.
\end{lem}
\begin{proof}
$\gamma$ is represented by
a \v{C}ech cocycle $\{ \Gamma_{ij} \in \HH^0(U_{ij}, Hom(E,E)^0) \}$.
It suffices to verify that
\begin{equation}\label{eq:Ad(f)Gamma}
f^*(\Gamma_{ij})+f_*(\Gamma_{ij})=\tr \bigl(I^*(\iota\circ P)_* (\Gamma_{ij}) \bigr)\cdot\Id
\ \text{in}\ 
\HH^0(U_{ij},Hom(E,E(K_X))).
\end{equation}
Since $E$ is torsion free, it suffices to verify \eqref{eq:Ad(f)Gamma}
on $U'_{ij}:=U_{ij} \setminus \operatorname{Supp}(Z)$.
On $U'_{ij}$, one can choose the local trivialization of $E$ so as
\[
\Gamma_{ij}=
\begin{pmatrix}
p & q \\ r & -p
\end{pmatrix},\quad
f= 
\begin{pmatrix}
0 & a \\ 0 & 0 \end{pmatrix},\quad
I=
\begin{pmatrix}
1 \\ 0
\end{pmatrix},\quad
\iota\circ P=(0\ a).
\]
Then $f^*(\Gamma_{ij})+f_*(\Gamma_{ij})=
\begin{pmatrix}
ar & 0 \\ 0 & ar 
\end{pmatrix}$ and
$\tr \bigl(I^*(\iota\circ P)_* (\Gamma_{ij}) \bigr)=ar$,
and thereby \eqref{eq:Ad(f)Gamma} holds on $U'_{ij}$.
\end{proof}

\begin{lem}\label{lem:easy-prpty-BZ}
We have the following in Setting \ref{stng:f20-locfree}.
\begin{enumerate}[(a)]
\item $B=\textstyle\sum_{i=1}^2 n_i F_i \quad (0\leq n_i\leq m_i-1)$. \label{enum:B-ni}
\item $\HH^0(\sO(2K_X-B+\cf)\otimes I_Z)=0 $. \label{enum:l'zb-geq-2}
\item $1-\textstyle\sum_i \frac{1}{m_i}(n_i+1) \geq 0$.
Especially, $1/2\geq  (n_i+1)/m_i$ for $i=1$ or $2$. \label{enum:since-K-B-geq-0}
\item $\hom(E,E(K_X))^0= 1$, that is, $k=1$ at \eqref{eq:formal-moduli-intro}. \label{enum:k=1}
\item 
$\rk \left(\iota_*:\ \tor Ext^1_{\pi}(E,G) \longrightarrow
 Ext^1_{\pi}(E,F(K_X))\right)$ equals the number $R$ in \eqref{eq:defn-h-R}. \label{enum:torExt1(EG)-R}
\end{enumerate}
\end{lem}

\begin{proof}
\eqref{enum:B-ni}\ This is because $(K_X-B)\cdot \sO(1)\geq 0$ since $E$ is stable, and
$(K_X-\mathfrak f)\cdot \sO(1)<0$ since $d=1$. \quad
\eqref{enum:l'zb-geq-2}\ Since \eqref{eq:exseq-f}${}_{\eta}$ does not split and
$\pur Ext^1_{\pi}(G,F)=\sO(l'_{BZ}-2)$ by Prop. \ref{prop:ext1(EE)decomp},
it follows that $l'_{BZ}\geq 2$. \\
\eqref{enum:since-K-B-geq-0}\ This is because $(K_X-B)\cdot \sO(1)\geq 0$, since $E$ is stable.\\
\eqref{enum:k=1}\ Since $E$ is of Case II, 
$\hom(E,E(K_X))\leq \hom(G,G(K_X))+\hom(E,F(K_X))=h^0(\sO(B))=1$
by \eqref{eq:hK13-23a}. \\
\eqref{enum:torExt1(EG)-R}\
From \eqref{enum:(2,1a)} and \eqref{enum:pur(2,3)} in Lemma \ref{prop:ext1(EE)decomp} and
\eqref{eq:pur-IP_*}, $H^0(\pur Ext^1_{\pi}(E,E))=0$, so
there is an exact sequence
\[ 0 \longrightarrow H^1(Hom_{\pi}(E,E)) \longrightarrow \Ext^1(E,E)
   \longrightarrow H^0(Ext^1_{\pi}(E,E))= H^0(\tor Ext^1_{\pi}(E,E)) \longrightarrow 0. \]
For $h^1(K_X)=q(X)=0$,
$R$ equals the rank of $f_*: \Ext^1(E,E)\rightarrow \Ext^1(E,G) \rightarrow \Ext^1(E,E(K_X))$
by Lemma \ref{lem:f_*+f^*=0}.
The natural map $Hom_{\pi}(G,G)\rightarrow Hom_{\pi}(E,G)$ is isomorphic from 
Setting \ref{stng:f20-locfree}, so
$H^1(Hom_{\pi}(E,E))\rightarrow H^1(Hom_{\pi}(E,G))\simeq H^1(Hom_{\pi}(G,G))=H^1(\sO_{\PP^1})=0$ 
is zero map.
Thus we get \eqref{enum:torExt1(EG)-R}. 
\end{proof}

\begin{prop}\label{f_*-tor-nonzero} 
In Setting \ref{stng:f20-locfree}, it holds that
$R\geq l(Z)-l(Z\cap B)\geq 1$.
\end{prop}

\begin{proof}
From Lemma \ref{prop:ext1(EE)decomp} \eqref{enum:(2,2-2,3a)}, \eqref{enum:(K2,1-2,2a)}
and Lemma \ref{lem:easy-prpty-BZ} \eqref{enum:torExt1(EG)-R},
\begin{equation*}
R=\rk \left(\iota_*: \tor Ext^1_{\pi}(E,G) \longrightarrow
 Ext^1_{\pi}(E,F(K_X))\right).
\end{equation*}
In the following natural diagram,
the second row is exact and the first row is injective.
\begin{equation}\label{eq:hom(FF(K)/G)}
\xymatrix{
 \Hom(F,F(K_X)/G) \ar@{^{(}-}[r] &
 \HH^0(\tor Ext^1_{\pi}(F,G)) & \\
 \HH^0(Hom_{\pi}(E,F(K_X)/G)) \ar[u] \ar[r]
 & \HH^0(\tor Ext^1_{\pi}(E,G)) \ar[u] \ar[r]^{\iota_*} 
 & \HH^0(\tor Ext^1_{\pi}(E,F(K_X)))
}
\end{equation}
This is because the exact sequence
\[Hom_{\pi}(E,F(K_X)/G) 
\longrightarrow \tor Ext^1_{\pi}(E,G) \longrightarrow
\tor Ext^1_{\pi}(E,F(K_X))\]
consists of zero-dimensional sheaves, and
$\Hom(F,F(K_X))\simeq H^0(\sO(K_X))=0$.\par
Now we take $\alpha_0\in \Hom(F, F(K_X)/G)=\HH^0(\sO(K_X)\cdot (\sO_X/\tau_B I_Z))$.
By the homomorphism $(1,1)\rightarrow(1,2)$ in \eqref{eq:hom(FF(K)/G)},
regard $\alpha_0$ as an element of $(1,2)$.
Some inverse image $\alpha'\in(2,2)$ of $\alpha_0$ exists since
the homomorphism $(2,2)\rightarrow (1,2)$ is surjective by Lemma \ref{prop:ext1(EE)decomp}
\eqref{enum:(2,3-3,3a)}.
If $\iota_*(\alpha')=0$, then we have an inverse image $\beta\in (2,1)$ of $\alpha'$ by
the homomorphism $(2,1)\rightarrow (2,2)$.
The image of $\beta$ by $(2,1)\rightarrow (1,1)$ equals $\alpha_0$.
One can locally describe the ideal $I_Z=\langle f,g\rangle$,
and then $F\subset E$ are locally described as $(f,-g):\sO\subset \sO^{\oplus 2}$, and so
$\alpha_0 \in \HH^0(\sO(K_X)\cdot (I_Z/\tau_B I_Z))$.
As a result, 
$$\rk(\iota_*)\geq \dim\Cok
\left(\HH^0(\sO(K_X)\cdot (I_Z/\tau_B I_Z))\hookrightarrow \HH^0(\sO(K_X)\cdot (\sO_X/\tau_B I_Z)) \right).$$
The local generator $\tau_B$ of $I_B$ gives an inclusion
$\tau_B\cdot \sO_Z = \langle \tau_B\rangle /\tau_B I_Z \hookrightarrow \sO_Z/\tau_B I_Z$
since $Z$ is zero-dimensional. The above inequality implies that
$$\rk(\iota_*)\geq 
h^0(\langle \tau_B\rangle/\tau_B I_Z)-h^0(I_Z\cap \langle \tau_B\rangle/\tau_B I_Z)
 =l(\langle \tau_B\rangle/I_Z \cap \langle\tau_B\rangle).$$
It follows that $R\geq l(Z)-l(Z\cap B)$ from this and the next exact sequence:
\[ 0\longrightarrow \langle \tau_B\rangle/I_Z \cap \langle\tau_B\rangle \longrightarrow
 \sO_Z \longrightarrow
\sO_X/I_Z+\langle \tau_B\rangle =\sO_{Z\cap B} \longrightarrow 0. \]
Now suppose that $l(Z)=l(Z\cap B)$. Then $\tau_B\in I_Z$ and so
$\HH^0(\sO(2K_X-2B+\cf))\subset \HH^0(\sO(2K_X-B+\cf)\otimes I_Z)$.
Notice that $H^0(\sO(2K_X-2B+\cf))\neq 0$. Indeed,
one can assume that $1/2\geq (n_1+1)/m_1$, that is, $m_1-2n_1-2\geq 0$
from Lemma \ref{lem:easy-prpty-BZ} \eqref{enum:since-K-B-geq-0}, 
and $m_2-n_2-1\geq 0$ by Lemma \ref{lem:easy-prpty-BZ} \eqref{enum:B-ni}.
Thus $2K_X-2B+\cf=\cf+\sum_i(m_i-2n_i-2)F_i=(m_1-2n_1-2)F_1+2(m_2-1-n_2)F_2$ is effiective,
and thereby $\HH^0(\sO(2K_X-B+\cf)\otimes I_Z)\neq 0$.
However this contradicts to Lemma \ref{lem:easy-prpty-BZ} \eqref{enum:l'zb-geq-2}.
%
\end{proof}

\begin{lem}\label{lem:R=lZ-h(EF(K)B)}
In Setting \ref{stng:f20-locfree}, $R=l(Z)-\hom(E,F(K_X)|_B)$.
\end{lem}
\begin{proof}
From Lemma \ref{prop:ext1(EE)decomp} \eqref{enum:(h2,3-hK2,1)}, we have an exact sequence
\[ 0 \longrightarrow
Hom_{\pi}(E,F(K_X)/G) \longrightarrow \tor Ext^1_{\pi}(E,G) \longrightarrow Ext^1_{\pi}(E,F(K_X)). \]
This and Lemma \ref{lem:easy-prpty-BZ} \eqref{enum:torExt1(EG)-R} deduce that
$R=3l(Z)-\hom(E,F(K_X)/G)$.
Since
$F(K_X)/G= F(K_X) \otimes \sO_X/\tau_B\cdot I_Z$,
$\hom(E,F(K_X)/G)=\hom(E,F(K_X)\cdot \sO_Z)+\hom(E,F(K_X)|_B)=2l(Z)+\hom(E,F(K_X)|_B)$.
\end{proof}

\begin{prop}\label{prop:locfree-l(Z)-l(ZB)=1}
(1) In Setting \ref{stng:f20-locfree}, the following holds when $l(Z)-l(Z\cap B)=1$: \par
We decompose a l.c.i. scheme $Z$ as $Z=Z_0\ \sqcup Z_1 \sqcup Z_2$, where $Z_0\cap F_i=\emptyset$,\ and
$Z_i \subset F_i$ for $i=1,2$ as sets. Then
$Z_1\cup Z_2 \subset B$ as schemes and $l(Z_0)=1$. Moreover, $\HH^0(\sO(2K_X-B)\otimes I_{Z_1\cup Z_2})=0$, 
$K_X-B \in 2\Pic(X)$,\ and $(K_X-B)\cdot \sO(1)> 0$.\\
(2) Conversely, suppose $Z$ and $B$ satisfies all conditions in (1) and especially that $l(Z)-l(Z\cap B)=1$.
Then there is a locally free sheaf $E$ accompanied with an exact sequence \eqref{eq:exseq-f}
that does not split on the generic fiber $X_{\eta}$.
Such a sheaf $E$ satisfies $\Ext^2(E,E)^0\neq 0$,
applies to Case II in Fact \ref{fact:RestrGenericFib}, and $R=1$.
\end{prop}
\begin{proof}
(1) Denote $Z\cap B$ by $Z'=Z'_1\sqcup Z'_2$. 
Set-theoretically, we have $\operatorname{Supp}(I_{Z'}/I_Z) =\{q \}$ with some point $q$.
Remark that
\begin{equation}\label{eq:mqIZ'/IZ}
{\mathfrak m}_q \cdot (I_{Z'}/I_Z)= 0.
\end{equation}
Indeed, $Z$ is an Artinian scheme, so ${\mathfrak m}^L_q\cdot (I_{Z'}/I_Z)=0$ for some natural number $L$,
but this can not hold if ${\mathfrak m}_q \cdot (I_{Z'}/I_Z)\neq 0$
since $l(I_{Z'}/I_Z)=1$. \par
Suppose that $Z_0=\emptyset$.
From Lemma \ref{lem:easy-prpty-BZ} \eqref{enum:since-K-B-geq-0},
$1/2\geq (n_i+1)/m_i$ for $i=1$ or $2$.
We assume that it holds for $i=1$, so
$m_1-2-n_1\geq n_1$ and $2m_2-2-n_2\geq n_2$, and then
\begin{equation}\label{eq:m1-2n1F1-contain-Z'1}
 \tau_{F_1}^{m_1-2-n_1} \in \langle \tau_{F_1}^{n_1}\rangle  \subset I_{Z'_1} \quad \text{and}\quad
 \tau_{F_2}^{2m_2-2-n_2 }\in \langle \tau_{F_2}^{n_2}\rangle \subset I_{Z'_2}.
\end{equation}
Because $2K_X-B+\cf=(m_1-2-n_1)F_1+(2m_2-2-n_2)F_2$,
\[\tau_{F_1}^{m_1-2-n_1}\cdot \tau_{F_2}^{2m_2-2-n_2} \in
\HH^0 \left( \sO(2K_X-B+\cf)\otimes I_{Z'} \right). \]
From Lemma \ref{lem:easy-prpty-BZ} \eqref{enum:l'zb-geq-2}, it should holds 
\begin{equation}\label{eq:(m1-2-n1)F1-notcontain-Z1-or}
\tau_{F_1}^{m_1-2-n_1} \not\in I_{Z_1}  \quad \text{and}\quad
\tau_{F_2}^{2m_2-2-n_2 }\not\in I_{Z_2}.
\end{equation}
First suppose that $\tau_{F_1}^{m_1-2-n_1} \not\in I_{Z_1}$.
From \eqref{eq:mqIZ'/IZ} and \eqref{eq:m1-2n1F1-contain-Z'1},
it holds that $q\in F_1$, $Z'_1\neq Z_1$, $m_1-2-n_1=n_1$ and $Z'_2=Z_2$.
Then 
\begin{equation}
 \tau_{F_1}^{2m_1-2-n_1} \in I_{Z_1}  \quad \text{and} \quad
 \tau_{F_2}^{m_2-2-n_2} \in I_{Z'_2}=I_{Z_2}.
\end{equation}
Indeed, one can check that $2m_1-3-n_1\geq n_1$, and so 
$\tau_{F_1}^{2m_1-2-n_1}\in \langle \tau_{F_1}^{n_1+1} \rangle \subset \tau_{F_1}\cdot I_{Z'_1} \subset I_{Z_1}$
from \eqref{eq:mqIZ'/IZ}.
One can verify the right equation similarly to \eqref{eq:m1-2n1F1-contain-Z'1},
since $(n_2+1)/m_2 \geq 1/2$ by Lemma \ref{lem:easy-prpty-BZ}.
Because $2K_X-B+\cf=(2m_1-2-n_1)F_1+(m_2-2-n_2)F_2$,
\[\tau_{F_1}^{2m_1-2-n_1}\cdot \tau_{F_2}^{m_2-2-n_2} \in
\HH^0 \left( \sO(2K_X-B+\cf)\otimes I_{Z} \right). \]
This contradicts to  Lemma \ref{lem:easy-prpty-BZ} \eqref{enum:l'zb-geq-2}.
By \eqref{eq:(m1-2-n1)F1-notcontain-Z1-or},
$\tau_{F_2}^{2m_2-2-n_2 }\not\in I_{Z_2}$.
Then it should holds that $2m_2-2-n_2=n_2$ from \eqref{eq:m1-2n1F1-contain-Z'1},
but this contradicts to Lemma \ref{lem:easy-prpty-BZ} \eqref{enum:since-K-B-geq-0}.
Therefore we conclude that $Z_0\neq \emptyset$.
At the end, we remark that
$\HH^0(\sO(2K_X-B)\otimes I_{Z_1\cup Z_2})= \HH^0(\sO(2K_X-B+\cf)\otimes I_Z)=0$ 
by Lemma \ref{lem:easy-prpty-BZ} \eqref{enum:l'zb-geq-2}.\\
(2) It holds that $\HH^0(\sO(2K_X-B+\cf)\otimes I_{Z})=0$, and so
there is a sheaf with extension \eqref{eq:exseq-f}
such that \eqref{eq:exseq-f}${}_{\eta}$ does not split by Lemma \ref{prop:ext1(EE)decomp} \eqref{enum:(1,1a)}.
$\HH^0(\sO(2K_X-B)\otimes I_{Z'})=0$ for any subscheme $Z'\subset Z$ with $l(Z')=l(Z)-1$, and so
there is a locally free sheaf 
accompanied with an exact sequence \eqref{eq:exseq-f} by the Serre correspondence \cite[Thm. 5.1.1]{HL:text}.
Both of these conditions are open
in the family of sheaves with exact sequences \eqref{eq:exseq-f}, so
we can find a sheaf $E$ satisfing both conditions.\par
Now we shall show that $R=1$ for such $E$.
Since $Z_{12}\subset B$ as schemes, we have a natural exact sequence
\[ 0 \longrightarrow \calCN_{Z_{12}/B}:= I_{Z_{12}}/I_B \longrightarrow \sO_B \longrightarrow
 \sO_{Z_{12}} \longrightarrow 0,\qquad \text{which induces}   \]
\begin{multline}\label{eq:hom(OZ12,OB)}
0 \longrightarrow \Hom_B(\sO_B, \sO(B)|_B ) \longrightarrow \Hom_B(\calCN_{Z_{12}/B}, \sO(B)|_B ) \\
\longrightarrow \Ext^1_B(\sO_{Z_{12}}, \sO(B)|_B) \longrightarrow \Ext^1_B(\sO_B, \sO(B)|_B).
\end{multline}
By duality theorem \cite[Sect. 3.11.]{Ha:residue},
$\ext^1_B(\sO_{Z_{12}},\sO(B)|_B)=h^0(\sO_{Z_{12}}\otimes \sO(K_X)|_B)=l(Z_{12})$.
From Fact \ref{prop:order-O(F)F} $\hom(\sO_B, \sO(B)|_B)=0$.
It follows that $\ext^1_B(\sO_B,\sO(B)|_B)=0$, since
$\chi(B,\sO(B)|_B)=\chi(X,\sO(B)|_B)=0$.
Therefore we get $ \hom_B(\calCN_{Z_{12}/B},\sO(B)|_B)=l(Z_{12})$ from \eqref{eq:hom(OZ12,OB)}.
Since $Z_{12}\subset B$ as schemes, 
a surjection $E \rightarrow G=\sO_X(-D)\otimes I_Z$ induces
a surjection $E|_B \rightarrow \sO(-D)\otimes \calCN_{Z_{12}/B}$, and thereby
\begin{multline*}
 \hom(E,F(K_X)|_B) \geq \hom(\sO(-D)\otimes \calCN_{Z_{12}/B},F(K_X)|_B) \\
=\hom(\calCN_{Z_{12}/B}, \sO(B)|_B)=l(Z_{12}). 
\end{multline*}
From this and Lemma \ref{lem:R=lZ-h(EF(K)B)}, $R\leq l(Z)-l(Z_{12})=1$,
and from Proposition \ref{f_*-tor-nonzero}, $R\geq 1$.
We get $R=1$ consequently.
\end{proof}

\begin{prop}\label{prop:example-Case2-R1}
For fixed integer $c$,
there actually exist $Z$ and $B$ such that $l(Z)=c$ and that all conditions hold in 
Proposition \ref{prop:locfree-l(Z)-l(ZB)=1} (1) in the following cases.
As a result, there actually exist stable sheaves $E$ such that $c_2(E)=c$ and all conditions
in Proposition \ref{prop:locfree-l(Z)-l(ZB)=1} (2) hold.\par
(1) Both $m_1$ and $m_2$ are odd, $(m_1,m_2)\neq (3,3)$ and $c\geq (m_1+1)/2$. \par
(2) $m_1$ is odd, $m_2\neq 2$ is even, and $c\geq (m_1+1)/2$. \par
(3) $m_1\geq 6$ is even, $m_2 \neq 3$ is odd, and $c\geq 2\lfloor m_1/4\rfloor +2$.\par
(4) Both $m_1$ and $m_2$ are even, $(m_1,m_2)\neq (4,4),\ (6,6)$ and $c\geq (m_1/2)+2$. 
\end{prop}
\begin{proof}
$B$ and $Z_1\cup Z_2\subset B$ satisfies "Moreover" paragraphs in Proposition \ref{prop:locfree-l(Z)-l(ZB)=1} (1)
if and only if the following (a)--(e) are valid:\par
(a) $1> \sum_i (n_i+1)/m_i$. \qquad 
(b) For some $i$, $m_i<2(n_i+1)$, that is, $n_i \geq \lfloor m_i/2 \rfloor$ for some $i$. \par
(c) $K_X-B \in 2\Pic(X)$. \qquad
(d) $Z_1 \subset n_1F_1$ and $Z_2 \subset n_2F_2 $ as subschemes. \par
(e) $Z_1 \not\subset (m_1-2-n_1)F_1$ or $Z_2 \not\subset (m_2-2-n_2)F_2$.  \\
Here (b) is equivalent to saying that $h^0(\sO(2(K_X-B)))=0$, that is valid
if $Z_1\subset Z_2\subset B$ and $\HH^0(\sO(2K_X-B)\otimes I_{Z_1\cup Z_2})=0$.\\
(1) We can assume that $m_1\leq m_2$.
We put $n_1=\lfloor m_1/2 \rfloor=(m_1-1)/2$, 
and put $n_2=1$ if $m_1 \equiv 1 \pmod{4}$, and $n_2=0$ if $m_1 \equiv 3 \pmod{4}$.
Then one can check $B$ satisfies (a), (b), (c) if $(m_1,m_2)\neq(3,3)$.
Since $m_1-2-n_1<n_1$,
there is a l.c.i. subscheme $Z'_1\subset n_1 F_1$ such that $l(Z'_1)=m_1-1-n_1$ and that
$Z'_1 \not\subset (m_1-2-n_1)F_1$.
Using it, one can readily find $B$ and $Z\supset Z'_1$ such that $l(Z)=c$ 
satisfing all conditions in Proposition \ref{prop:locfree-l(Z)-l(ZB)=1} (1)
if $c\geq m_1-n_1=(m_1+1)/2$. \par
One can verify (2)--(4) similarly to (1).
In case of (2), we can put $n_1=\lfloor m_1/2 \rfloor=(m_1-1)/2$ and put $n_2=1$.
In case of (3), we can put $n_2=0$ and
let $n_1$ be the smallest odd integer such that $n_2\geq \lfloor m_1/2\rfloor$,
that is, $2\lfloor m_1/4\rfloor +1$.
In case of (4), we can put $n_1=m_1/2$, $n_2=1$ if $4\not |m_1$ and $n_2=0$ if $4|m_1$.
\end{proof}

\end{document}

%% file: newcmds.tex
\newcommand{\Aa}{{\mathbb{A}}}
\newcommand{\Amp}{{\operatorname{Amp}}}

\newcommand{\CC}{{\mathbb{C}}}

\newcommand{\Coh}{{\operatorname{Coh}}}
\newcommand{\Cok}{\operatorname{Cok}}

\newcommand{\ext}{\operatorname{ext}}
\newcommand{\Ext}{\operatorname{Ext}}

\newcommand{\Hilb}{\operatorname{Hilb}}
\newcommand{\Hom}{\operatorname{Hom}}
\newcommand{\id}{\operatorname{id}}

\newcommand{\Ker}{\operatorname{Ker}}

\newcommand{\NN}{{\mathbb{N}}}

\newcommand{\Ox}{{\mathcal O}_X}

\newcommand{\Pic}{\operatorname{Pic}}
\newcommand{\PP}{{\bf{P}}}

\newcommand{\QQ}{{\mathbb{Q}}}

\newcommand{\rk}{{\operatorname{rk}}}
\newcommand{\RR}{{\mathbb{R}}}

\newcommand{\Sing}{\operatorname{Sing}}

\newcommand{\Spec}{\operatorname{Spec}}

\newcommand{\tr}{\operatorname{tr}}
\newcommand{\ZZ}{{\mathbb{Z}}}